\documentclass[12pt,reqno]{amsart}

\usepackage{amsmath,amssymb,amsthm}
\usepackage{natbib}
\usepackage{bm}
\usepackage{threeparttable}
\usepackage{tabularx}
\usepackage{lscape}
\usepackage{longtable}
\usepackage{color}
\usepackage{graphicx}

\makeatletter

\@addtoreset{equation}{section}
\makeatletter

\theoremstyle{plain}
\newtheorem{theorem}{Theorem}[section]
\newtheorem{lemma}{Lemma}[section]
\newtheorem{corollary}{Corollary}[section]
\newtheorem{proposition}{Proposition}[section]
\theoremstyle{definition}
\newtheorem{remark}{Remark}[section]

\DeclareMathOperator{\Var}{Var}
\DeclareMathOperator{\Cov}{Cov}

\DeclareMathOperator{\op}{op}

\newcommand{\indep}{\mathop{\perp\!\!\!\!\perp}}
\newcommand{\bx}{\bm{x}}
\newcommand{\E}{\mathbb{E}}
\renewcommand{\Pr}{\mathbb{P}}
\newcommand{\bbeta}{\bm{\beta}}
\renewcommand{\hat}{\widehat}
\renewcommand{\tilde}{\widetilde}

\begin{document}
\title[Estimation and inference under misspecification]{Estimation and inference for linear panel data models under misspecification when both $n$ and $T$ are large}
\thanks{K. Kato is supported by the Grant-in-Aid for Young Scientists (B) (22730179, 25780152), the Japan Society for the Promotion of Science. }
\author[Galvao]{Antonio F. Galvao}
\author[Kato]{Kengo Kato}

\address[A.F. Galvao]{
Department of Economics, University of Iowa, W284 Pappajohn Business Building, 21 E. Market Street, Iowa City, IA 52242. 
}
\email{antonio-galvao@uiowa.edu}

\address[K. Kato]{
Graduate School of Economics, University of Tokyo, 7-3-1 Hongo, Bunkyo-ku, Tokyo 113-0033, Japan.
}
\email{kkato@e.u-tokyo.ac.jp}

\date{First version: February 13, 2013. This version: \today}

\begin{abstract}
This paper considers fixed effects (FE) estimation for linear panel data models under possible model misspecification when both the number of individuals, $n$, and the number of time periods, $T$, are large. We first clarify the probability limit of the FE estimator and argue that this probability limit can be regarded as a pseudo-true parameter. We then establish the asymptotic distributional properties of the FE estimator around the pseudo-true parameter when $n$ and $T$ jointly go to infinity. Notably, we show that the FE estimator suffers from the incidental parameters bias of which the top order is $O(T^{-1})$, and even after the incidental parameters bias is completely removed, the rate of convergence of the FE estimator depends on the degree of model misspecification and is either $(nT)^{-1/2}$ or $n^{-1/2}$. Second,  we  establish asymptotically valid inference on  the (pseudo-true) parameter. Specifically, we derive the asymptotic properties of the clustered covariance matrix (CCM) estimator and the cross section bootstrap, and show that they are robust to model misspecification. This establishes a rigorous theoretical ground for the use of  the CCM estimator and the cross section bootstrap when model misspecification and the incidental parameters bias (in the coefficient estimate) are present.  We conduct Monte Carlo simulations to evaluate the finite sample performance of the estimators and inference methods, together with a simple application to the unemployment dynamics in the U.S.
\end{abstract}

\maketitle

% \title{Estimation and inference for panel data models under misspecification when both $n$ and $T$ are large\footnote{The authors would like to express their appreciation to SeoJeong Lee and Kazuhiko Hayakawa for useful comments and discussion regarding this paper. Computer programs to replicate the numerical analyses are available from the authors. All the remaining errors are ours.}}
% \author{
% Antonio F. Galvao\thanks{Department of Economics, University of Wisconsin-Milwaukee, Bolton Hall 852, 3210 N. Maryland Ave., Milwaukee, WI 53201; and University of Iowa, W210 Pappajohn Business Building, 21 E. Market Street, Iowa City, IA 52242. E-mail: agalvao@uwm.edu.} \quad 
% Kengo Kato\thanks{Department of Mathematics, Graduate School of Science, Hiroshima University, 1-3-1 Kagamiyama, Higashi-Hiroshima, Hiroshima 739-8526, Japan. E:mail: kkato@hiroshima-u.ac.jp.}}
%  

\maketitle

\section{Introduction}

It is well known that for a cross-section data set, the ordinary least squares (OLS) estimator is typically  consistent for 
the coefficient vector of the best linear approximation to the conditional mean, even if the conditional mean is not necessarily linear \citep{W80,W82}. Such a ``robust'' nature of OLS is one of the reasons why OLS is popular in empirical studies \citep[][Chapter 3]{AP08}.

Suppose now that a panel data set is available. In such a case, we are typically interested in estimating the partial effects of the observed explanatory variables, $\bx_{it}$, on the conditional mean of the dependent variable, $y_{it}$, conditional on $\bx_{it}$ and the unobservable individual effect $c_{i}$, where $i$ denotes the index for individuals and $t$ denotes the index for time.\footnote{See, e.g., \cite{W01}, Chapter 10. We follow the notation used in this reference.}
For that purpose, a popular strategy is to model the conditional mean $\E[y_{it} \mid  \bx_{it},c_{i}]$ as additive in both $c_{i}$ and $\bx_{it}$, and linear in $\bx_{it}$. The focus is then on estimating the coefficient vector of $\bx_{it}$.  The fixed effects (FE) estimator, which is identical to the OLS estimator treating the individual effects as parameters to be estimated, is often used to estimate the coefficient vector. While for models without strict exogeneity, such as dynamic panel data models, the FE estimator is generally inconsistent when $n$ (the number of individuals) goes to infinity and $T$ (the number of time periods) is fixed because of the incidental parameters problem \citep{NS48,N81,L00}, it is still a fundamental estimator.  In particular, when $n$ and $T$ jointly go to infinity, the FE estimator becomes consistent. Furthermore, after the incidental parameters bias is properly corrected, the FE estimator is known to have a centered limiting normal distribution provided that $n/T^{3} \to 0$, which restricts $T$ to be mildly large but allows $T$ to be small relative to $n$. 
Such asymptotic properties of the FE estimator under large $n$ and $T$ asymptotics have been extensively studied in the econometrics literature, especially for panel autoregressive (AR) models, partly motivated by the fact that panel data sets with mildly large $T$ have become available in empirical studies.\footnote{See, e.g., \cite{K95,HK02,AA03,BC05,BK06,PS07,H07a,O08,O10,L10}.}

The previous discussion presumes that the model is correctly specified, that is,  the conditional mean $\E[y_{it} \mid  \bx_{it},c_{i}]$ is truly additive in $c_{i}$ and $\bx_{it}$, and linear in $\bx_{it}$. The goals of this paper are twofold. The first objective is to study the asymptotic properties of the FE estimator under possible model misspecification when both $n$ and $T$ are large. 
The asymptotics used is the joint asymptotics where $n$ and $T$ jointly go to infinity (more precisely, we index $T$ by $n$ and let $T=T_{n} \to \infty$ as $n \to \infty$). 
Assume that  $\{ (c_{i},y_{i1},\bx_{i1},y_{i2},\bx_{i2},\dots) \}_{i=1}^{\infty}$ is independent and identically distributed (i.i.d.), and for each $i \geq 1$, conditional on $c_{i}$, 
$\{ (y_{it},\bx_{it}')' \}_{t=1}^{\infty}$ is stationary and weakly dependent.
Suppose that $\E[ y_{it} \mid \bx_{it},c_{i}]$ may not be additive in $c_{i}$ and $\bx_{it}$, nor linear in $\bx_{it}$.
Under this setting, we show that the probability limit of the FE estimator is identical to the coefficient vector on $\bx_{it}$ of the best {\em partial} linear approximation to $\E[y_{it} \mid \bx_{it},c_{i}]$, which gives some rational to use the FE estimator (and its variant) even when the model is possibly misspecified. We regard this probability limit as a pseudo-true parameter  (see Section \ref{sec: interpretation} for the discussion on interpretation -- or plausibility -- of this probability limit; especially if $\E[y_{it} \mid \bx_{it},c_{i}]$ is indeed additive in $c_{i}$ and $\bx_{it}$, and linear in $\bm{x}_{it}$, then the pseudo-true parameter coincides with the ``true'' coefficient on $\bm{x}_{it}$ in $\E[ y_{it} \mid \bm{x}_{it}, c_{i}]$).
We then establish the asymptotic distributional properties of the FE estimator around the pseudo-true parameter when $n$ and $T$ jointly go to infinity. 
We demonstrate that, as in the correct specification case, the FE estimator suffers from the incidental parameters bias of which the top order is $T^{-1}$. 
Moreover, we show that, after  the incidental parameters bias is completely removed, the rate of convergence of the FE estimator depends on the degree of model misspecification and is either $(nT)^{-1/2}$ or $n^{-1/2}$.

The second goal of the paper is to establish asymptotically valid inference on the (pseudo-true) parameter vector. Since the FE estimator has the bias of order $T^{-1}$, the first step is to reduce the bias to $O(T^{-2})$. For that purpose, one can use existing general-purpose bias reduction methods proposed in the recent nonlinear panel data literature.\footnote{Hence there is no new result in the bias correction part.} For example, one can use the half-panel jackknife (HPJ) proposed by \cite{DJ09}. We refer to \cite{HK11} and \cite{AB09} for alternative approaches on bias correction for fixed effects estimators in panel data models.
%  Alternatively, a direct approach to bias correction is to analytically estimate the bias term. In this case, one typically estimates the first order bias term by using the technique of heteroscedasticity and autocorrelation consistent (HAC) covariance matrix estimation  \citep[see][]{HK11}. Another approach is to use bias reducing priors on individual effects (see, for example, \cite{AB09}, and references therein). See also \cite{AH07} for a review on bias correction for fixed effects estimators in nonlinear panel data models.
After the bias is properly reduced, the FE estimator has the centered limiting normal distribution provided that $n/T^{3} \to 0$ or $n/T^{4} \to 0$ depending on the degree of model misspecification. We are then interested in estimating the covariance matrix or quantiles of the centered limiting normal distribution. 
To this end, we study the asymptotic properties of the clustered covariance matrix (CCM) estimator \citep{A87} and the cross section bootstrap \citep{K08} under the prescribed setting. We show that the CCM estimator (with an appropriate estimator of the parameter vector) is consistent in a suitable sense and can be used to make asymptotically valid inference on the parameter vector provided that $n/T^{3} \to 0$ or $n/T^{4} \to 0$. This shows that inference using the CCM estimator is ``robust'' to model misspecification. Also the cross section bootstrap can consistently estimate the centered limiting distribution of the FE estimator  without any knowledge on the degree of model misspecification and hence is robust to model misspecification, and moreover, interestingly,   {\em without any growth restriction on $T$}. 
The second feature of the cross section bootstrap is notable and shows (in a sense) that the incidental parameters bias does not appear in the bootstrap distribution. 

Allowing for potential model misspecification is of importance in practice. In particular, this paper is of practical importance because it provides an interpretation for the FE estimator under potential misspecification, and additionally, it proposes methods for inference in linear panel data model with large $n$ and $T$ that are robust to model misspecification. However, the study of estimation and inference for linear panel data models that are robust to model misspecification is scarce.\footnote{\cite{AP08}, p.166, remarked that ``The set of assumption leading to (5.1.2) is more restrictive than those we used to motivate regression in Chapter 3; we need the linear, additive functional form to make headway on the problem of {\em unobserved} confounders using panel data with no instruments''.}
An exception is \cite{L10} where he considered the lag order misspecification of panel AR models and established the asymptotic properties of the FE estimator under  possible misspecification of the lag order.
However, his focus is on the incidental parameters bias and he did not study the inference problem on the pseudo-true parameter. Moreover, \cite{L10} did not cover a general form of model misspecification. 
This paper fills this void. Furthermore, the asymptotic properties of the CCM estimator and the cross section bootstrap when model misspecification and the incidental parameters bias (in the coefficient estimate) are present have not been studied in a systematic form and hence is under-developed. 
\cite{H07b} investigated the asymptotic properties of the CCM estimator when $n$ and $T$ are large but did not allow the case where the incidental parameters bias appears, nor did he cover model misspecification. 
\cite{K08} studied the asymptotic properties of the cross section bootstrap when $n$ and $T$ are large, but ruled out the case where the incidental parameters bias appears, nor did he cover model misspecification as well. 
Hence we believe that this paper is the first one  that establishes a rigorous theoretical ground on the use of  the CCM estimator and the cross section bootstrap when model misspecification and the incidental parameters bias are present. 
It is important to notice that, even without
model misspecification, these asymptotic properties of the CCM and cross section
bootstrap when the incidental parameters bias is present are new.

We conduct Monte Carlo simulations to evaluate the finite sample performance of the estimators and inference methods under misspecification.  
 We are particularly interested in the empirical coverage of the 95\% nominal confidence interval. 
The empirical coverage probability using the CCM and cross-section bootstrap, especially  the cross section bootstrap applied to pivotal statistics, is good.  
We also apply the procedures discussed in this paper to a model of unemployment dynamics at the U.S. state level. The results generate speed of adjustment of the unemployment rate towards the state specific equilibrium of about 17\%. In addition, the analysis of estimates indicates that increments in economic growth are associated with smaller unemployment rates.

The organization of this paper is as follows. In Section \ref{sec: interpretation}, we discuss the interpretation of FE estimator under misspecification. In Section \ref{sec: asymptotics}, we present the theoretical results on the asymptotic properties of the FE estimator under misspecification. In Section \ref{sec: inference}, we presents the  results on the inference methods. In Section \ref{sec: MC}, we  report a Monte Carlo study to assess the finite sample performance of the estimators and inference methods, together with a simple application to a real data. 
Section \ref{sec: conclusion} concludes. 
We place all the technical proofs to the Appendix. We also include additional theoretical and simulation results in the Appendix. 

{\bf Notation}: For a generic vector $\bm{z}$, let $z^{a}$ denote the $a$-th element of $\bm{z}$. For a generic matrix $A$, let $A^{ab}$ denote its $(a,b)$-th element. For a generic vector $\bm{z}_{it}$ with index $(i,t)$, $\bar{\bm{z}}_{i} = T^{-1} \sum_{t=1}^{T} \bm{z}_{it}$. 
Let $\| \cdot \|$ denote the Euclidean norm. For any matrix $A$, let $\| A \|_{\op}$ denote the operator norm of $A$.
For any symmetric matrix $A$, let $\lambda_{\min}(A)$ denote the minimum eigenvalue of $A$.
We also use the notation $\bm{z}^{\otimes 2} = \bm{z} \bm{z}'$ for a generic vector $\bm{z}$.

{\bf Note on asymptotics}: In what follows, we consider the asymptotic framework in which $T=T_{n} \to \infty$ as $n \to \infty$, so that if we write $n \to \infty$, it automatically means that $T \to \infty$. The limit is always taken as $n \to \infty$. This asymptotic scheme is used to capture the situation where $n$ and $T$ are both large.

\section{Interpretation of fixed effects estimator under misspecification}
\label{sec: interpretation}

In this section we clarify the probability limit, which we will regard as a pseudo-true parameter, of the FE estimator under the joint asymptotics and discuss interpretation (or plausibility) of the pseudo-true parameter. The discussion is to some extent parallel to the linear regression case but there is a subtle difference due to the appearance of individual effects. 

Suppose that we have a panel data set $\{ (c_{i},y_{it},\bx_{it}): i=1,\dots,n;\ t=1,\dots,T \}$, where $c_{i}$ is an unobservable individual-specific random variable taking values in an abstract (Polish) space, $y_{it}$ is a scalar dependent variable and $\bx_{it}$ is a vector of $p$ explanatory variables.
Typically, the random variable $c_{i}$, which is constant over time,  represents an individual characteristic such as ability or firm's managerial quality which we would include in the analysis if it were observable \cite[see][Chapter 10]{W01}.
 Assume that $\{ (c_{i},y_{i1},\bx_{i1},\dots,y_{iT},\bx_{iT}) \}_{i=1}^{n}$ is i.i.d., and for each $1 \leq i \leq n$, conditional on $c_{i}$, 
$\{ (y_{it},\bx_{it}')' \}_{t=1}^{T}$ is a realization of a stationary weakly dependent process.\footnote{That is, the data are i.i.d. across individuals, but for each individual, conditional on the individual effect, the data are (weakly) dependent across time.} Here the marginal distribution of $(c_{i},y_{it},\bx_{it})$ is invariant with respect to $(i,t)$. 
Typically, we are interested in estimating the partial effects of $\bx_{it}$ on the conditional mean $\E[y_{it} \mid \bx_{it},c_{i}]$ with keeping $c_{i}$ fixed. A ``standard'' linear panel data model assumes that 
the conditional mean is of the form $g(c_{i}) + \bx_{it}'\bbeta$ with unknown function $g$ and vector $\bbeta$, and redefines $c_{i}$ by $g(c_{i})$ since, in any case, $c_{i}$ is unobservable and modeling a functional form for the individual effect is virtually meaningless \citep[see][Chapter 5]{AP08}. In this paper, the ``correct'' specification refers
to that the conditional mean $\E[y_{it} \mid  \bx_{it},c_{i}]$ is written in the form $g(c_{i})+\bx_{it}'\bbeta$, and ``model
misspecification'' signifies any violation of this condition. For instance, this can happen if there are omitted variables or if nonlinearity occurs in the model. We discuss more details below (see Examples 1 and 2 below for concrete examples). 

The FE estimator defined by 
\begin{equation}
\hat{\bbeta} = \left \{ \frac{1}{nT} \sum_{i=1}^{n} \sum_{t=1}^{T} (\bx_{it} - \bar{\bx}_{i}) (\bx_{it} - \bar{\bx}_{i})' \right \}^{-1} \left \{ \frac{1}{nT} \sum_{i=1}^{n} \sum_{t=1}^{T} (\bx_{it} - \bar{\bx}_{i}) (y_{it} - \bar{y}_{i}) \right \}
\label{fe}
\end{equation}
is consistent for the coefficient vector on $\bx_{it}$ as $n$ goes to infinity and $T$ is fixed if the specification is correct (for the moment, assuming that $\hat{\bbeta}$ exists) and additionally the strict exogoneity assumption 
$\E[ y_{it} \mid \bx_{i1},\dots,\bx_{iT}, c_{i}] = \E[y_{it} \mid \bx_{it},c_{i}]$ is met. 
If the strict exogoneity assumption is violated, then the FE estimator is not fixed-$T$ consistent, but as $T \to \infty$ with $n$, the FE estimator becomes consistent for the coefficient vector on $\bx_{it}$ provided that the specification is correct. 
Suppose now that the specification is not correct, i.e., $\E[y_{it} \mid \bx_{it},c_{i}]$ may not be written in the form $g(c_{i}) + \bx_{it}'\bbeta$, and consider the probability limit of the FE estimator when $n$ and $T$ jointly go to infinity. 
Proposition \ref{prop1} ahead shows that, subject to some technical conditions, we have, as $n \to \infty$ and $T=T_{n} \to \infty$, 
\begin{equation}
\hat{\bbeta} \stackrel{\Pr}{\to}  \E[ \tilde{\bx}_{it}\tilde{\bx}_{it}']^{-1} \E[ \tilde{\bx}_{it} \tilde{y}_{it}] =: \bbeta_{0},
\label{ptrue}
\end{equation}
where $\tilde{y}_{it} = y_{it} - \E[ y_{it} \mid c_{i} ]$ and $\tilde{\bx}_{it} = \bx_{it} - \E[ \bx_{it} \mid c_{i}]$ (for a moment, assume that some moments exist). To gain some insight, we provide a heuristic derivation of this probability limit under the sequential asymptotics  where $T \to \infty$ first and then $n \to \infty$. By definition, we have
\begin{equation*} 
\hat{\bbeta} = \left \{ \frac{1}{nT} \sum_{i=1}^{n} \sum_{t=1}^{T} \tilde{\bx}_{it}\tilde{\bx}'_{it} - \frac{1}{n} \sum_{i=1}^{n} \bar{\tilde{\bx}}_{i}\bar{\tilde{\bx}}_{i}' \right \}^{-1} \left \{ \frac{1}{nT} \sum_{i=1}^{n} \sum_{t=1}^{T} \tilde{\bx}_{it} \tilde{y}_{it} - \frac{1}{n} \sum_{i=1}^{n} \bar{\tilde{\bx}}_{i} \bar{\tilde{y}}_{i} \right \}.
\end{equation*}
Since $\{ (y_{it},\bx_{it}')' \}_{t=1}^{T}$ is weakly dependent conditional on $c_{i}$, as $T \to \infty$ first, we have $T^{-1} \sum_{t=1}^{T} \tilde{\bx}_{it}\tilde{\bx}'_{it} \stackrel{\Pr}{\to} \E[ \tilde{\bx}_{i1}\tilde{\bx}_{i1}' \mid c_{i}], \ \bar{\tilde{\bx}}_{i} \stackrel{\Pr}{\to} \bm{0}, \ T^{-1} \sum_{t=1}^{T} \tilde{\bx}_{it} \tilde{y}_{it} \stackrel{\Pr}{\to} \E[ \tilde{\bx}_{i1} \tilde{y}_{i1} \mid c_{i}]$ and $\bar{\tilde{y}}_{i} \stackrel{\Pr}{\to} 0$, so that 
$\hat{\bbeta} \stackrel{\Pr}{\to} \{ n^{-1} \sum_{i=1}^{n} \E[ \tilde{\bx}_{i1}\tilde{\bx}_{i1}' \mid c_{i}] \}^{-1} \{ n^{-1} \sum_{i=1}^{n} \E[ \tilde{\bx}_{i1} \tilde{y}_{i1} \mid c_{i}] \}$. By the law of large numbers, the right side converges in probability to $\bbeta_{0}$ as $n \to \infty$. 
% By Proposition \ref{prop1} below, this is true even when $n$ and $T$ jointly go to infinity (that is, $T=T_{n} \to \infty$ as $n \to \infty$).  

In what follows, we discuss an interpretation of $\bbeta_{0}$ defined in (\ref{ptrue}). A direct interpretation is that $\bbeta_{0}$ is the coefficient vector of the best linear approximation to $\E[\tilde{y}_{it} \mid \tilde{\bx}_{it}]$, but
this interpretation does not explain the connection with the primal object of estimating the partial effects of $\bx_{it}$ on $\E[y_{it} \mid \bx_{it},c_{i}]$. 
However, the link emerges from the following discussion. Let $g_{0}(c_{i}) + \bx_{it}'\bm{b}_{0}$ be the best partial linear predictor of 
$y_{it}$ on $(c_{i},\bx_{it})$, i.e., 
\begin{equation*}
\E[ ( y_{it} - g_{0}(c_{i}) - \bx_{it}'\bm{b}_{0})^{2}] = \min_{g \in L_{2}(c_{1}), \bm{b} \in \mathbb{R}^{p}} \E[ (y_{it} - g(c_{i})-\bx_{it}'\bm{b})^{2}],
\end{equation*}
where $L_{2}(c_{1}) = \{ g : \E[ g(c_{1})^{2} ] < \infty \}$. By a simple calculation, the explicit solution $(g_{0},\bm{b}_{0})$ is given by 
\begin{equation*}
g_{0} (c_{i}) = \E[ y_{it} \mid c_{i} ] - \E[\bx_{it} \mid c_{i}]'\bm{b}_{0}, \ \bm{b}_{0} = \E[ \tilde{\bx}_{it}\tilde{\bx}_{it}']^{-1}\E[ \tilde{\bx}_{it} \tilde{y}_{it}] = \bbeta_{0},
\end{equation*}
and hence $\bbeta_{0}$ is the coefficient vector on $\bx_{it}$ of the best partial linear predictor.
Moreover, it is not difficult to see that $g_{0}(c_{i}) + \bx_{it}'\bm{b}_{0}$ is indeed the best partial linear approximation to $\E[y_{it} \mid \bx_{it},c_{i}]$, i.e., 
\begin{equation}
\E[ ( \E[y_{it} \mid \bx_{it},c_{i}]  - g_{0}(c_{i}) - \bx_{it}'\bm{b}_{0})^{2}] = \min_{g \in L_{2}(c_{1}), \bm{b} \in \mathbb{R}^{p}} \E[ (\E[y_{it} \mid \bx_{it},c_{i}] - g(c_{i})-\bx_{it}'\bm{b})^{2}]. \label{papprox}
\end{equation}
Therefore, the vector $\bbeta_{0}$ defined by (\ref{ptrue}) is identical to the coefficient vector on $\bx_{it}$ of the best partial linear approximation to $\E[y_{it} \mid \bx_{it},c_{i}]$ in (\ref{papprox}).
Just as the coefficient vector of the best linear approximation to the conditional mean is a parameter of interest in the cross section case, as discussed in Chapter 3 of \cite{AP08}, 
the vector $\bbeta_{0}$ here can be regarded as a plausible parameter of interest in the panel data case. Hence, in this paper,  we consider $\bbeta_{0}$ to be a parameter of interest and treat $\bbeta_{0}$ as a pseudo-true parameter.

\begin{remark}
Clearly  if $\E[ y_{it} \mid \bx_{it},c_{i}]$ is indeed additive in $c_{i}$ and $\bx_{it}$, and linear in $\bx_{it}$, 
then $\bbeta_{0}$ coincides with the ``true'' coefficient on $\bx_{it}$ in $\E[ y_{it} \mid \bx_{it},c_{i}]$. 
\end{remark}

As in the cross section case, letting the approximation error denote $\epsilon_{it} = y_{it} - g_{0}(c_{i}) - \bx_{it}'\bbeta_{0} = \tilde{y}_{it} - \tilde{\bx}_{it}'\bbeta_{0}$, we have a regression form
\begin{equation}
y_{it} = g_{0}(c_{i}) + \bx_{it}'\bbeta_{0} + \epsilon_{it}, \ \E[ \epsilon_{it} \mid c_{i} ] = 0, \ \E[ \bx_{it} \epsilon_{it} ] = 0.
\label{regression}
\end{equation}
Importantly, the ``error term'' $\epsilon_{it}$ here may not satisfy the conditional mean restriction $\E[ \epsilon_{it} \mid \bx_{it},c_{i}] = 0$ due to possible model misspecification. 

In (\ref{regression}), there are two scenarios on violation of the conditional mean restriction $\E[ \epsilon_{it} \mid \bx_{it},c_{i}] = 0$. 
One is the case where $\E[ \epsilon_{it} \mid \bx_{it},c_{i}] \neq 0$ with  positive probability but $\E[ \tilde{\bx}_{it} \epsilon_{it} \mid c_{i}] = \bm{0}$ a.s. 
The other is the case where $\E[ \tilde{\bx}_{it} \epsilon_{it} \mid c_{i}] \neq \bm{0}$ with  positive probability. 
Depending on these two cases, the asymptotic properties of the FE estimator {\em do} change drastically (see Section 3).
Generally, both cases can happen. We give three simple examples to fix the idea.

{\em Example 1}. Panel AR model with misspecified lag order. Suppose that the true data generating process follows a panel AR(2) model
\begin{equation*}
y_{it} = c_{i} + \phi_{1} y_{i,t-1} + \phi_{2} y_{i,t-2} + u_{it}, \ c_{i} \indep \{ u_{it} \}_{t \in \mathbb{Z}}, \ c_{i} \in \mathbb{R}, \ \{ u_{it} \}: i.i.d., \ \E[ u_{it} ] = 0,
\end{equation*}
where $(\phi_{1},\phi_{2})$ is such that $\phi_{2} + \phi_{1} < 1, \phi_{2} - \phi_{1} < 1$ and $-1 < \phi_{2} < 1$. Conditional on $c_{i}$, $\{ y_{it} \}_{t \in \mathbb{Z}}$ is stationary and typically weakly dependent. By a simple calculation, we have 
$\E[ y_{it} \mid c_{i}] = c_{i}/(1-\phi_{1}-\phi_{2})$. Letting $\tilde{y}_{it} = y_{it} - \E[y_{it} \mid c_{i}]$, we have $\tilde{y}_{it}= \phi_{1}\tilde{y}_{i,t-1} + \phi_{2} \tilde{y}_{i,t-2} + u_{it}$. Hence $\{ \tilde{y}_{it} \}_{t \in \mathbb{Z}}$ is independent of $c_{i}$. Suppose now that we incorrectly fit a panel AR(1) model. 
Note that in this case, we have 
\begin{align*}
\E[ y_{it} \mid y_{i,t-1},c_{i} ] =c_{i} + \phi_{1} y_{i,t-1} + \phi_{2} \E[ y_{i,t-2} \mid y_{i,t-1},c_{i} ]. 
\end{align*}
Hence $\E[ y_{it} \mid y_{i,t-1},c_{i} ]$ is generally a nonlinear function of $c_{i}$ and $y_{i,t-1}$ except for the cases where $\phi_{2} = 0$ or the distribution of $u_{it}$ is normal. 
Here the solution of the equation $\E[ \tilde{y}_{i,t-1} (\tilde{y}_{it} - \beta_{0} \tilde{y}_{i,t-1})] = 0$ is the first order autocorrelation coefficient of $\{ \tilde{y}_{it} \}$, i.e., $\beta_{0} = \Cov (\tilde{y}_{it},\tilde{y}_{i,t-1})/\Var(\tilde{y}_{i,t-1})$. Letting $\epsilon_{it} = \tilde{y}_{it} - \beta_{0} \tilde{y}_{i,t-1} = (\phi_{1}-\beta_{0})\tilde{y}_{i,t-1} + \phi_{2}\tilde{y}_{i,t-2} + u_{it}$, we have $\tilde{y}_{it} = \beta_{0} \tilde{y}_{i,t-1} + \epsilon_{it}$, i.e., $y_{it} = (1-\beta_{0})c_{i}/(1-\phi_{1}-\phi_{2}) + \beta_{0} y_{i,t-1} + \epsilon_{it}$. By the independence of $\{ \tilde{y}_{it} \}$ from $c_{i}$, we have 
$\E[ \tilde{y}_{i,t-1} \epsilon_{it} \mid c_{i}] = \E[ \tilde{y}_{i,t-1} \epsilon_{it}] = 0$ a.s. 
\cite{L10} studied such panel AR models with misspecified lag order in detail. This paper covers lag order misspecification as a special case.

{\em Example 2}. Static panel with a mismeasured regressor. Suppose the true data generating process is as following
\begin{equation*}
y_{it} = c_{i} + \phi x_{it}^{*} + u_{it}, \ \E[ u_{it} \mid x_{it}^{*},c_{i}] = 0.
\end{equation*}
In addition, 
\begin{equation*}
x_{it} = x_{it}^{*} + v_{it}, \  \E[ v_{it} \mid x_{it}^{*},c_{i} ] = 0. 
\end{equation*}
Suppose that $c_{i}, u_{it}$ and $v_{it}$ are independent, and we incorrectly fit a model using $x_{it}$ instead of $x_{it}^{*}$. 
Here we have 
\[
\tilde{x}_{it} = x_{it} - \E[ x_{it} \mid c_{i} ]= x_{it}^{*} + v_{it} - \E[ x_{it}^{*} \mid c_{i} ] = \tilde{x}_{it}^{*} + v_{it}, 
\]
where $\tilde{x}_{it}^{*} = x_{it}^{*} - \E[ x_{it}^{*} \mid c_{i} ]$, and $\tilde{y}_{it} = y_{it} - \E[ y_{it} \mid c_{i} ] = \phi \tilde{x}_{it}^{*} + u_{it}$. 
Hence the solution to the equation $\E[ \tilde{x}_{it} ( \tilde{y}_{it} - \beta \tilde{x}_{it}) ] = 0$ is given by 
\[
\beta_{0} = \E[ \tilde{x}_{it}^{2} ]^{-1} \E[ \tilde{y}_{it} \tilde{x}_{it} ] = \frac{\E[ (\tilde{x}_{it}^{*})^{2}]}{\E[ (\tilde{x}_{it}^{*})^{2}] + \E[ v_{it}^{2} ]} \phi. 
\]
Finally, by $\epsilon_{it} = \tilde{y}_{it} - \beta_{0} \tilde{x}_{it}$, we have $\E[ \tilde{x}_{it} \epsilon_{it} \mid c_{i}] =  \E[ (\tilde{x}_{it}^{*})^{2} \mid c_{i} ] \phi - (\E[ (\tilde{x}_{it}^{*})^{2} \mid c_{i} ] + \E[ v_{it}^{2} ]) \beta_{0}$, which is generally non-zero.

{\em Example 3}. Random coefficients AR model.  Suppose that the true data generating process follows the following random coefficients AR(1) model:
\begin{equation*}
y_{it} = c_{i} y_{i,t-1} + u_{it}, \ c_{i} \indep \{ u_{it} \}_{t \in \mathbb{Z}}, \ | c_{i} | < 1, \ u_{it} \sim N(0,1), i.i.d. 
\end{equation*}
In this case, $\E[ y_{it} \mid y_{i,t-1}, c_{i}] = c_{i} y_{i,t-1}$. It is routine to verify that 
\begin{equation*}
y_{it} \mid c_{i} \sim N(0,1/(1-c_{i}^{2})).
\end{equation*}
Suppose that we incorrectly fit a panel AR(1) model. Here $\tilde{y}_{it} = y_{it} - \E[ y_{it} \mid c_{i} ] = y_{it}$ and the solution of the equation $\E[ y_{i,t-1} (y_{it} - \beta_{0} y_{i,t-1})] = 0$ is given by 
\begin{equation*}
\beta_{0} = \frac{\E[ c_{i} y_{i,t-1}^{2}]}{\E[ y_{i,t-1}^{2} ]} = \frac{\E[ c_{i}/(1-c_{i}^{2})]}{\E[ 1/(1-c_{i}^{2})]}.
\end{equation*}
Here $\epsilon_{it} = y_{it} - \beta_{0} y_{i,t-1}$, and 
\begin{equation*}
\E[ y_{i,t-1} \epsilon_{it} \mid c_{i} ] = \E[ c_{i} y_{i,t-1}^{2} - \beta_{0} y_{i,t-1}^{2} \mid c_{i} ]= \frac{c_{i}-\beta_{0}}{1-c_{i}^{2}},
\end{equation*}
which is non-zero a.s. if $c_{i}$ obeys a continuous distribution. 

\begin{remark}[Interpretation under pseudo-likelihood setting]
 The results in this paper could be interpreted as corresponding to the pseudo-likelihood model. Under the additional assumptions of independence and normality, the resulting (conditional) maximum likelihood estimator (MLE) of $\bm{\beta}$ (given $c_{1},\dots,c_{n}$) is identical to the FE estimator. Thus, the FE estimator defined in (\ref{fe}) can be viewed as a pseudo-MLE.
\end{remark}

\begin{remark}[Discussion on \cite{LGF12}]
\cite{LGF12} made an interesting observation about the pseudo-true parameter when the link function is misspecified for the generalized linear model (see their Corollary 1). That is, the pseudo-true parameter is proportional to the true one up to nonzero scalar. However, their setting is significantly different from ours; first of all in Corollary 1 they assumed that the true model satisfies a generalized linear model but only the link function is misspecified, and only cross section data are available. In our case basically no ``model'' is assumed (and hence the ``true parameter'' is not well-defined in general), so that their result does not extend to our setting. 
\end{remark}

\begin{remark}[Alternative estimators]
In this paper we focus on the FE estimator. This is because the FE estimator is widely used in practice and, as we have shown, the probability limit of the FE estimator under misspecification admits a natural and plausible interpretation, parallel to the linear regression case. There could be alternative estimators; for example, we could consider the average of the individual-wise OLS estimators, i.e., let $\hat{\bbeta}_{i}$ denote the OLS estimator obtained by regressing $y_{it}$ on $(1,\bx_{it}')'$ with each fixed $i$, and consider the estimator $\hat{\bbeta}^{ave} = n^{-1} \sum_{i=1}^{n} \hat{\bbeta}_{i}$. However, this estimator does not share the interpretation that the FE estimator possesses. In some cases the probability limit of $\hat{\bbeta}^{ave}$ happens to be identical to $\bbeta_{0}$, but not in general. To keep tight focus, we only consider the FE estimator in what follows. 
% {\color{red}  In this paper we focus on the properties of the FE estimator defined in \ref{fe}. Alternative estimators might be defined and used. For example, alternatively to $\bm{\beta}_0$ in \ref{fe} one might be interested in the properties of $\bm\beta_{1} =\E[\E[\tilde{x}_{it}\tilde{x}_{it}' | c_{i}]^{-1} \E[\tilde{x}_{it}\tilde{y}_{it}' | c_{i}]]$. In this case, analogously to the result in \ref{ptrue}, $\frac{1}{n}\sum_{i=1}^n \hat{\bm\beta}_{i} \stackrel{p}{\to} \bm\beta_{1}$ as $n,T \to \infty$, where $\hat{\bm\beta}_{i}$ for each $i$ is the OLS estimator obtained from the pure time-series regression of $y_{it}$ on $\bm{x}_{it}$ and an individual specific constant. However, we focus on the FE since it is widely used in applied work, and the interpretation of the FE estimator is connected with the primal object of interest, which is estimating the partial effects of $\bm{x}_{it}$ on $\mathbb{E}[y_{it} \mid \bm{x}_{it},c_{i}]$. } 
\end{remark}

% Suppose that $c_{i}$ and $x_{it}$ are scalar, and $\E[ y_{it} | x_{it},c_{i}] =  h(x_{it},c_{i})$, where $h(\cdot,\cdot)$ is a possibly nonlinear function. Letting $\tilde{y}_{it} = y_{it} - \E[ y_{it} | c_{i}]$ and $\tilde{x}_{it} = x_{it} - \E[ x_{it} | c_{i}]$, the solution to the equation $\E[ \tilde{x}_{it} (\tilde{y}_{it} - \beta_{0} \tilde{x}_{it}) ] = 0$ is given by
% $\beta_{0} = \E[ \tilde{x}_{it}(h(x_{it},c_{i}) - \E[ h(x_{it},c_{i}) | c_{i}])]/\E[ \tilde{x}_{it}^{2}]$. By taking $\epsilon_{it} = \tilde{y}_{it} - \beta_{0} \tilde{x}_{it}$, we have $\tilde{y}_{it} = \beta_{0} \tilde{x}_{it} + \epsilon_{it}$, i.e., $y_{it} =(  \E[ h(x_{it},c_{i}) | c_{i}] - \beta_{0} \E[ x_{it} | c_{i}]) + \beta_{0} x_{it} + \epsilon_{it}$. In this case, unless additional assumptions are satisfied, $\E[ \tilde{x}_{it} (h(x_{it},c_{i}) - \E[ h(x_{it},c_{i}) | c_{i}] - \beta_{0} \tilde{x}_{it}) | c_{i}] = 0$ a.s. does not hold  despite $\E[ \tilde{x}_{it} (h(x_{it},c_{i}) - \E[ h(x_{it},c_{i}) | c_{i}] - \beta_{0}\tilde{x}_{it}) ] = 0$,
% which implies that $\E[ \tilde{x}_{it} \epsilon_{it} | c_{i} ] \neq 0$ with  positive probability.

\section{Asymptotic properties of fixed effects estimator under misspecification}
\label{sec: asymptotics}

In this section, we study the asymptotic properties of the FE estimator under possible model misspecification (i.e., $\E[y_{it} \mid  \bx_{it},c_{i}]$ is not assumed to be additive in $c_{i}$ and $\bx_{it}$, nor linear in $\bx_{it}$). We make the following regularity conditions.

\begin{description}
\item[(A1)] $(c_{i},y_{it},\bx_{it}) \in S \times \mathbb{R} \times \mathbb{R}^{p}$, where $S$ is a Polish space. $\{ (c_{i},y_{i1},\bx_{i1},y_{i2},\bx_{i2},\dots) \}_{i=1}^{\infty}$ is i.i.d., and for each $i \geq 1$, conditional on $c_{i}$, 
$\{ (y_{it},\bx_{it}')' \}_{t=1}^{\infty}$ is  a stationary $\alpha$-mixing process with mixing coefficients $\alpha(k \mid c_{i})$. Assume that there exists a sequence of constants  $\alpha(k)$ such that $\alpha(k \mid c_{i}) \leq \alpha(k)$ a.s. for all $k \geq 1$, and $\sum_{k=1}^{\infty} k \alpha(k)^{\delta/(8+\delta)} < \infty$ for some $\delta > 0$.
\item[(A2)] Define $\tilde{y}_{it} = y_{it} - \E[y_{it} \mid c_{i}]$ and $\tilde{\bx}_{it} = \bx_{it} - \E[\bx_{it} \mid c_{i}]$ (assume that $\E[y_{it} \mid c_{i}]$ and $\E[\bx_{it} \mid c_{i}]$ exist). There exists a constant $M > 0$ such that $\E [ | \tilde{y}_{it} |^{8+\delta} \mid c_{i}] \leq M$ a.s. and $\E[ \| \tilde{\bx}_{it}\|^{8+\delta} \mid c_{i} ] \leq M$ a.s., where $\delta > 0$ is given in (A1). 
\item[(A3)]  Define the matrix $A = \E[ \tilde{\bx}_{it} \tilde{\bx}_{it}']$. Assume that the matrix $A$ is nonsingular. 
\end{description}

In condition (A1), $c_{i}$ is an unobservable, individual-specific random variable allowed to be dependent with $\bx_{it}$ in an arbitrary form. 
Condition (A1) assumes that the observations are independent in the cross section dimension, but allows for dependence in the time dimension conditional on individual effects.  We refer to Section 2.6 of \cite{FY03} for some basic properties of mixing processes. 
Condition (A1) is similar to Condition 1 of \cite{HK11}, and allows for a flexible time series dependence. The mixing condition is also similar to Assumption 1 (iii) of \cite{G11}, while she allowed for cross section dependence. The mixing condition is used only to bound covariances and moments of sums of random variables, and not crucial for the central limit theorem.
Therefore, in principle, it could be replaced by assuming directly such bounds. We assume the mixing condition to make the paper clear.
Note that because of stationarity assumption made in (A1), the marginal distribution of $(c_{i},y_{it},\bx_{it})$ is invariant with respect to $(i,t)$, i.e., $(c_{i},y_{it},\bx_{it}) \stackrel{d}{=} (c_{1},y_{1,1},\bx_{1,1})$.\footnote{The stationarity assumption means that for dynamic models, the initial condition, $y_{i0}$, is  drawn from the stationary distribution, conditional on $c_{i}$.} The stationary assumption rules out time trends, but is needed to well-define the pseudo-true parameter $\beta_{0}$, and maintained in this paper. Extensions to non-stationary cases will need different analysis and are not covered in this paper.
Condition (A2) is a moment condition. As usual, there is a trade-off between the mixing condition and the moment condition.  
Condition (A2) implies that $\E[ | \epsilon_{it} |^{8+\delta} \mid c_{i}] \leq M'$ a.s. for some constant $M' > 0$. 
Condition (A3) is a standard full rank condition. 
%If $\bm{x}_{it}$ contains an element that does not vary over time for any $i$, then the corresponding element of $\tilde{\bm{x}}_{it}$ is equal to zero for all $t$, and in this case, condition (A3) could not be satisfied. Thus, (A3) excludes the analysis of time-constant variables.

Now we introduce some notation. Recall the FE estimator: 
\begin{align}
\hat{\bbeta} &= \left \{ \frac{1}{nT} \sum_{i=1}^{n} \sum_{t=1}^{T} (\bx_{it} - \bar{\bx}_{i}) (\bx_{it} - \bar{\bx}_{i})' \right \}^{-1} \left \{ \frac{1}{nT} \sum_{i=1}^{n} \sum_{t=1}^{T} (\bx_{it} - \bar{\bx}_{i}) (y_{it} - \bar{y}_{i}) \right \}  \notag \\
&=: \hat{A}^{-1} \hat{S}, \label{fe2}
\end{align}
where $\hat{A}=(nT)^{-1} \sum_{i=1}^{n} \sum_{t=1}^{T} (\bm{x}_{it} - \bar{\bm{x}}_{i}) (\bm{x}_{it} - \bar{\bm{x}}_{i})'$ and $\hat{S}=(nT)^{-1} \sum_{i=1}^{n} \sum_{t=1}^{T} (\bm{x}_{it} - \bar{\bm{x}}_{i}) (y_{it} - \bar{y}_{i})$.
Under conditions (A1)-(A3), the matrix $\hat{A}$ on the right side is nonsingular with probability approaching one (see Lemma \ref{lem2} in Appendix \ref{Appendix A}). 

Recall that $\bbeta_{0}$ is defined by (\ref{ptrue}) and $\epsilon_{it}$ is defined by $\epsilon_{it} = \tilde{y}_{it} - \tilde{\bx}_{it}'\bbeta_{0}$ (see the previous section). Define 
\begin{equation*}
B_{T} =  \sum_{| k | \leq T-1} \left ( 1 - \frac{|k|}{T} \right ) \E[ \tilde{\bx}_{1,1} \epsilon_{1,1+k} ], \
D_{T} =  \sum_{| k | \leq T-1} \left ( 1 - \frac{|k|}{T} \right ) \E[ \tilde{\bx}_{1,1} \tilde{\bx}_{1,1+k}' ].
\end{equation*}
Here and in what follows, for the notational convenience, terms like $\E[\tilde{\bx}_{1,1} \epsilon_{1,1+k}]$ for $k < 0$ are understood as $\E[\tilde{\bx}_{1,1-k} \epsilon_{1,1}]$, i.e., 
\[
\E[\tilde{\bx}_{1,1} \epsilon_{1,1+k}] := \E[ \tilde{\bx}_{1,1-k} \epsilon_{1,1} ], \ k < 0.
\]
We shall obey the same convention to other such terms. For example, $\E[ \tilde{\bx}_{1,1} \tilde{\bx}_{1,1+k}' ] := \E[ \tilde{\bx}_{1,1-k} \tilde{\bx}_{1,1}' ]$ for $k < 0$. 
We first note that under conditions (A1)-(A3),  both $B_{T}$ and $D_{T}$ are well behaved in the following sense.

\begin{lemma}
\label{lem1}
Under conditions (A1)-(A3), we have
$\sum_{k=-\infty}^{\infty} |k| \| \E[\tilde{\bx}_{1,1} \epsilon_{1,1+k}] \| < \infty$ and $
\sum_{k=-\infty}^{\infty}\| \E[ \tilde{\bx}_{1,1} \tilde{\bx}_{1,1+k}' ] \|_{\op} < \infty$. 
\end{lemma}

All the technical proofs for Section 2 are gathered in Appendix \ref{Appendix A}. 

We now state the asymptotic properties of the FE estimator.   
Define 
\begin{equation*}
d_{nT} = 
\begin{cases}
nT, \ & \text{if $\E[ \tilde{\bx}_{it} \epsilon_{it} \mid c_{i} ] = \bm{0}$ a.s.}, \\
n, \ & \text{otherwise}, 
\end{cases}
\qquad \text{and} \qquad
\Sigma_{nT} = \E\left[ \left ( \frac{1}{nT} \sum_{i=1}^{n} \sum_{t=1}^{T} \tilde{\bx}_{it} \epsilon_{it} \right )^{\otimes 2} \right],
\end{equation*}
where recall that $\bm{z}^{\otimes 2} = \bm{z} \bm{z}'$. 

\begin{proposition}
\label{prop1}
Suppose that conditions (A1)-(A3) are satisfied. Letting $T=T_{n} \to \infty$ as $n \to \infty$, we have: 
\begin{align}
&\hat{\bbeta} - \bbeta_{0} + A^{-1} \left \{ \sum_{m=0}^{\infty} T^{-m-1} (D_{T}A^{-1})^{m} \right \} B_{T} \notag \\
&=A^{-1} \left ( \frac{1}{nT} \sum_{i=1}^{n} \sum_{t=1}^{T} \tilde{\bx}_{it} \epsilon_{it} \right ) + O_{\Pr}[n^{-1/2}\max\{d_{nT}^{-1/2},T^{-1}\}]. \label{Bahadur}
\end{align}
Therefore, we have: 
\begin{equation*}
\sqrt{d_{nT}}\left [ \hat{\bbeta} 
- \bbeta_{0} + A^{-1} \left \{ \sum_{m=0}^{\infty} T^{-m-1} (D_{T}A^{-1})^{m} \right \} B_{T} \right ] \stackrel{d}{\to} N(\bm{0},A^{-1} \Sigma A^{-1}),
\end{equation*}
where $\Sigma = \lim_{n \to \infty} (d_{nT_{n}}\Sigma_{nT_{n}})$ (the limit on the right side exists). 
\end{proposition}

We stress that Proposition \ref{prop1} holds without any specific growth condition on $T$. Also note that this proposition implies that $\hat{\bbeta} \stackrel{\Pr}{\to} \bbeta_{0}$ as long as $T=T_{n} \to \infty$. 

We discuss some implications of Proposition \ref{prop1}. First, Proposition \ref{prop1} shows that the FE estimator has the bias term of the form 
\begin{align}
&-A^{-1} \left \{ \sum_{m=0}^{\infty} T^{-m-1} (D_{T}A^{-1})^{m} \right \} B_{T} \notag \\
&\quad = - \frac{1}{T} A^{-1}B_{T} - \frac{1}{T^{2}} A^{-1}D_{T}A^{-1}B_{T} - \cdots - \frac{1}{T^{m}} A^{-1} (D_{T}A^{-1})^{m-1} B_{T} - \cdots, \label{ipbias}
\end{align}
which, following the literature, we call the ``incidental parameters bias''. There are two sources that contribute to the incidental parameters bias. 
The main source is conditional correlation between $\bx_{is}$ and $\epsilon_{it}$ for $s \neq t$ conditional on $c_{i}$, which arises  from using $\bar{\bx}_{i}$ instead of $\E[\bx_{i1} \mid c_{i}]$ in $\hat{S}$. Another source, which only contributes to higher order terms, is conditional correlation between $\bx_{is}$ and $\bx_{it}$ for $s \neq t$ conditional on $c_{i}$, which arises from using $\bar{\bx}_{i}$ instead of $\E[\bx_{i1} \mid c_{i}]$ in $\hat{A}$. 
Proposition \ref{prop1} makes explicit the incidental parameters bias of any order, which appears to be new even in the correct specification case (but under the current set of assumptions).\footnote{\cite{DJ10} considered higher order bias corrections for the panel AR model with exogenous variables, fixed effects, and unrestricted initial observations. The assumptions behind their paper are different from ours: here more general models (not restricted to panel AR models, and allowing for model misspecification) are covered, but \cite{DJ10} covered non-stationary cases. } The expansion (\ref{Bahadur}) is important in investigating the asymptotic properties of the cross section bootstrap in Section 4.  

Second, Proposition \ref{prop1} shows that, aside from the incidental parameters bias, the rate of convergence of the FE estimator depends on the degree of model misspecification, i.e., after the incidental parameters bias is completely removed, the FE estimator is $\sqrt{nT}$-consistent if $\E[ \tilde{\bx}_{it}\epsilon_{it} \mid c_{i} ] = \bm{0}$ a.s. and $\sqrt{n}$-consistent otherwise. In fact, the rate of convergence of the FE estimator (after the incidental parameters bias is  removed) depends on the 
order of the covariance matrix of the term
\begin{equation*}
\frac{1}{nT}\sum_{i=1}^{n} \sum_{t=1}^{T} \tilde{\bx}_{it} \epsilon_{it} = \frac{1}{nT} \sum_{i=1}^{n} \sum_{t=1}^{T} (\tilde{\bx}_{it} \epsilon_{it} - \E[\tilde{\bx}_{i1} \epsilon_{i1} \mid c_{i}]) + \frac{1}{n} \sum_{i=1}^{n} \E[ \tilde{\bx}_{i1} \epsilon_{i1} \mid c_{i}].
\end{equation*}
By this decomposition,
\begin{align*}
\Sigma_{nT} &= \E \left [ \left \{ \frac{1}{nT} \sum_{i=1}^{n} \sum_{t=1}^{T} (\tilde{\bx}_{it} \epsilon_{it} - \E[\tilde{\bx}_{i1} \epsilon_{i1} \mid c_{i}]) \right \}^{\otimes 2} \right ] + \E \left [ \left ( \frac{1}{n} \sum_{i=1}^{n} \E[ \tilde{\bx}_{i1} \epsilon_{i1} \mid c_{i}] \right )^{\otimes 2} \right ] \\
&=\frac{1}{nT} \E \left [ \left \{ \frac{1}{\sqrt{T}} \sum_{t=1}^{T} (\tilde{\bx}_{1t} \epsilon_{1t} - \E[\tilde{\bx}_{i1} \epsilon_{i1} \mid c_{1}]) \right \}^{\otimes 2} \right ] +
\frac{1}{n} \E[\E[\tilde{\bx}_{1,1} \epsilon_{1,1} \mid c_{1}]^{\otimes 2}].
\end{align*}
Since $\{ (y_{it},\bx_{it}')' \}_{t=1}^{\infty}$ is weakly dependent conditional on $c_{i}$, the first term is $O\{ (nT)^{-1} \}$. On the other hand, the second term is zero if $\E[ \tilde{\bx}_{it} \epsilon_{it} \mid c_{i} ] = \bm{0}$ a.s., but $\asymp n^{-1}$ otherwise, which shows that $\Sigma_{nT}  = O\{ (nT)^{-1} \}$ if $\E[ \tilde{\bx}_{it} \epsilon_{it} \mid c_{i} ] = \bm{0}$ a.s. and $\| \Sigma_{nT} \|_{\op} \asymp n^{-1}$ otherwise. Intuitively, unless $\E[ \tilde{\bx}_{it} \epsilon_{it} \mid c_{i} ] = \bm{0}$ a.s., $\tilde{\bx}_{it} \epsilon_{it}$ is unconditionally equicorrelated  across $t$, so that 
$\Sigma_{nT}$ has the slow rate $n^{-1}$ (if $\E[ \tilde{\bx}_{it} \epsilon_{it} \mid c_{i} ] = \bm{0}$ a.s., by condition (A1),  the covariance between $\tilde{\bx}_{is} \epsilon_{is}$ and $\tilde{\bx}_{it}\epsilon_{it}$ converges to zero sufficiently fast as $|s-t| \to \infty$, so $\Sigma_{nT}$ has the faster rate $(nT)^{-1}$). 

In some cases, at least theoretically, it may happen that $\E[ \tilde{\bx}_{it} \epsilon_{it} \mid c_{i} ] \neq \bm{0}$ with  positive probability but for some nonzero $\bm{r} \in \mathbb{R}^{p}$, $\bm{r}'A^{-1} \E[ \tilde{\bx}_{it} \epsilon_{it} \mid c_{i} ] = 0$ a.s., which corresponds to the case where the matrix $\E[\E[\tilde{\bx}_{1,1} \epsilon_{1,1} \mid c_{1}]^{\otimes 2}]$ is nonzero but degenerate.  In such a case, we have the expansion
\begin{align}
&\bm{r}'\left [ \hat{\bbeta} - \bbeta_{0} + A^{-1} \left \{ \sum_{m=0}^{\infty} T^{-m-1} (D_{T}A^{-1})^{m} \right \} B_{T} \right] \notag \\
&=\bm{r}'A^{-1} \left \{ \frac{1}{nT} \sum_{i=1}^{n} \sum_{t=1}^{T} (\tilde{\bx}_{it} \epsilon_{it} - \E[\tilde{\bx}_{i1} \epsilon_{i1} \mid c_{i}]) \right \} + 
O_{\Pr}[n^{-1/2}\max \{ n^{-1/2}, T^{-1} \}], \label{expansion2}
\end{align}
where the leading term of the right side multiplied by $\sqrt{nT}$ is asymptotically normal with mean zero and variance $\lim_{n \to \infty}[(nT_{n})\bm{r}'A^{-1} \Sigma_{nT_{n}}A^{-1} \bm{r}] < \infty$. This shows that, after subtracting the incidental parameters bias, the FE estimator may have different rates of convergence within its linear combinations. Moreover, the remainder term in the expansion (\ref{expansion2}) has the constant term of order $n^{-1}$, so that the extra bias term of order $n^{-1}$ appears in such a case.\footnote{By the proof of Proposition \ref{prop1}, it is shown that the remainder term in the expansion (\ref{expansion2}) is in fact further expanded as
$-n^{-1}\bm{r}'A^{-1}\E[\E[ \tilde{\bx}_{1,1}\tilde{\bx}_{1,1}' \mid c_{1}] A^{-1}\E[ \tilde{\bx}_{1,1}\epsilon_{1,1}' \mid c_{1}]] + O_{\Pr}[ n^{-1/2} \max \{ n^{-1},T^{-1} \}]$.}
By this, if additionally $T/n$ goes to some positive constant, the limiting normal distribution of the right side on (\ref{expansion2}) multiplied by $\sqrt{nT}$ has a nonzero bias in the mean of which the size is proportional to $\sqrt{T/n}$. However, such a case seems to be rather exceptional and we mainly focus on the case where the matrix $\Sigma$ is always nonsingular (i.e. $\Sigma$ is nonsingular in either case of $\E[ \tilde{\bx}_{it} \epsilon_{it} \mid c_{i} ] = \bm{0}$ a.s. or not).

It is possible to have an alternative expression of the bias term of order $T^{-1}$. 

\begin{corollary}
\label{cor1} Suppose that conditions (A1)-(A3) are satisfied. Then we have: $$B_{T} =  {\textstyle \sum}_{k=-\infty}^{\infty}  \E[\tilde{\bx}_{1,1} \epsilon_{1,1+k}] + O(T^{-1}) =: B+O(T^{-1}).$$
In particular, the bias term of order $T^{-1}$ in the expansion (\ref{Bahadur}) is rewritten as $T^{-1} A^{-1} B$. 
\end{corollary}

Finally, we provide some comments on the relation to the previous work. 

\begin{remark}[Relation with \cite{L10}]
Proposition \ref{prop1} is a nontrivial extension of Theorem 2 of \cite{L10}, in which he established the asymptotic properties of the FE estimator for panel AR models with exogenous variables allowing for lag order misspecification. Proposition \ref{prop1} allows for a more general form of model misspecification, including lag order misspecification as a special case, and exhausts the incidental parameters bias of any order.
\end{remark}

\begin{remark}[Relation with \cite{H07b}]
\label{rem: Hansen}
Proposition \ref{prop1} is related to \cite{H07b}. \cite{H07b} considered a  model $y_{it} = \bx_{it}'\bbeta_{0} + \epsilon_{it}$ with
$\E[ \epsilon_{it} \mid \bx_{i1},\dots,\bx_{iT} ]=0$ for all $1 \leq t \leq T$ or $\E[ \bx_{it} \epsilon_{it} ] = \bm{0}$, and showed that 
the OLS estimator is $\sqrt{n}$-consistent if there is no condition on time series dependence and $\sqrt{nT}$-consistent if a mixing condition is satisfied for time series dependence. What matters for the rate of convergence of the OLS estimator in \cite{H07b} is the order of the covariance matrix of the term $(nT)^{-1} \sum_{i=1}^{n} \sum_{t=1}^{T} \bx_{it} \epsilon_{it}$, which is $O(n^{-1})$ in  the ``no mixing'' case and $O\{ (nT)^{-1} \}$ in the mixing case. While there is a similarity, Proposition \ref{prop1} is not nested to his results in several aspects. First, in Proposition \ref{prop1}, the rate of convergence of the FE estimator (after the incidental parameters bias is removed) depends on the degree of model misspecification (i.e., $\E[ \bx_{it} \epsilon_{it} \mid c_{i}] = \bm{0}$ a.s. or not), rather than the assumption on time series dependence. Second, while his model covers panel data models with individual effects by considering $(y_{it},\bx_{it}')'$ to be transformed variables $(y_{it}-\bar{y}_{i},\bx_{it}'-\bar{\bx}_{i}')'$, the mixing assumption is not satisfied for the transformed variables as he admitted in footnote 3, so that his Theorem 3 does not apply to models with individual effects. Additionally,  his Assumption 3 essentially requires that $\E[ \bx_{is} \epsilon_{it} \mid c_{i}] = \bm{0}$ for all $1 \leq s,t \leq T$ under our setting (if we think of $(y_{it},\bx_{it}')'$ as transformed variables $(y_{it}-\bar{y}_{i},\bx_{it}'-\bar{\bx}_{i}')'$), so that his results do not cover the case where the incidental parameters bias appears. On the other hand, while we exclusively assume that $\{ (y_{it},\bx_{it}')' \}_{t=1}^{\infty}$ is mixing conditional on $c_{i}$, \cite{H07b} covered the case where no such mixing condition is satisfied. Therefore, the two papers are complementary in nature. 
\end{remark}

\begin{remark}[Relation with \cite{AH06}]
\cite{AH06} obtained a general incidental parameters bias formula for nonlinear panel data models, allowing for potential model misspecification, when $n/T$ is going to some constant. Their general result could be applied to the present setting in the case where $\E[ \tilde{\bx}_{it} \epsilon_{it} \mid c_{i} ]=\bm 0$. However, the stochastic expansion (\ref{Bahadur}) is not derived in \cite{AH06}; (\ref{Bahadur}) exhausts the incidental parameters bias up to infinite order, and covers the case where $\E[ \tilde{\bx}_{it} \epsilon_{it} \mid c_{i} ] \neq \bm 0$. This expansion is also the key for studying  the properties of cross section bootstrap. 
% \cite{AH06} discuss a modified objective function strategy to obtain estimators without bias to order $1/T$ in nonlinear dynamic panel models with multiple effects. They allow for potential model misspecification. Under some conditions, the bias correction results in \cite{AH06} can be directly applied, for instance in a dynamic model with $\E[ \tilde{\bx}_{it} \epsilon_{it} \mid c_{i} ] = \bm{0}$. However, the results on the characterization of the asymptotic bias presented in this paper are more general for the class of linear panel data.
\end{remark}

\section{Inference}
\label{sec: inference}

\subsection{Bias correction}

By Proposition \ref{prop1}, the FE estimator has the bias of order $T^{-1}$. 
In many econometric applications, $T$ is typically smaller than $n$, so that the normal approximation neglecting the bias may not be accurate in either case of $\E[ \tilde{\bx}_{it} \epsilon_{it} \mid c_{i}] = \bm{0}$ or $\E[\tilde{\bx}_{it} \epsilon_{it} \mid c_{i}] \neq \bm{0}$. 
Therefore, the first step to make inference on $\bbeta_{0}$ is to remove the bias of order $T^{-1}$ and reduce the order of the bias to $T^{-2}$.  Under model misspecification,  bias reduction methods that depend on specific models (such as panel AR models) may not work properly. Instead, we can use general-purpose bias reduction methods proposed in the recent nonlinear panel data literature. For example, the half-panel jackknife (HPJ) proposed by \cite{DJ09} is able to remove the bias of order $T^{-1}$ even under model misspecification. 
Suppose that $T$ is even. 
Let $S_{1} = \{ 1, \dots , T/2 \}$ and $S_{2} = \{ T/2+1, \dots, T \}$. For $l=1,2$, construct the FE estimator $\hat{\bbeta}_{S_{l}}$ based on the split sample 
$\{ (y_{it},\bx_{it}')' : i=1,\dots,n; t \in S_{l} \}$. Then the HPJ estimator is defined by $\hat{\bbeta}_{1/2} = 2 \hat{\bbeta} - (\hat{\bbeta}_{S_{1}}+\hat{\bbeta}_{S_{2}})/2$. Using the expansion (\ref{Bahadur}) and Corollary \ref{cor1}, as $T=T_{n} \to \infty$, we have the expansion
\begin{equation}
\hat{\bbeta}_{1/2} - \bbeta_{0} + O(T^{-2}) 
=A^{-1} \left ( \frac{1}{nT} \sum_{i=1}^{n} \sum_{t=1}^{T} \tilde{\bx}_{it} \epsilon_{it} \right ) + O_{\Pr}[n^{-1/2}\max\{d_{nT}^{-1/2},T^{-1}\}], 
\label{BahadurH}
\end{equation}
so that the bias is reduced to $O(T^{-2})$ in either case. Therefore, we have $\sqrt{d_{nT}} (\hat{\bbeta}_{1/2} - \bbeta_{0}) \stackrel{d}{\to} N(\bm{0},A^{-1} \Sigma A^{-1})$ provided that $\sqrt{d_{nT}}/T^{2} \to 0$. 
\cite{DJ09} proposed other automatic bias reduction methods applicable to general nonlinear panel data models. 
Their bias reduction methods are basically applicable to the model misspecification case.\footnote{For example, the higher order bias correction methods proposed in \cite{DJ09} could be adapted here; such higher order bias correction would be a good option when $n/T$ is large.} 

Alternatively, a direct approach to bias correction is to analytically estimate the bias term. In this case, we typically estimate the first order bias term $-A^{-1}B$ by using the technique of HAC covariance matrix estimation \citep[see][]{HK11}. Moreover, another alternative approach is to use bias reducing priors on individual effects \citep[see, for example,][and references therein]{AB09}. See also \cite{AH07} for a review on bias correction for fixed effects estimators in nonlinear panel data models.

\subsection{Clustered covariance matrix estimator}

By the previous discussion, a bias corrected estimator $\tilde{\bbeta}$ typically has the expansion
\begin{equation}
\tilde{\bbeta} - \bbeta_{0} + O(T^{-2}) 
=A^{-1} \left ( \frac{1}{nT} \sum_{i=1}^{n} \sum_{t=1}^{T} \tilde{\bx}_{it} \epsilon_{it} \right ) + O_{\Pr}[n^{-1/2}\max\{d_{nT}^{-1/2},T^{-1}\}]. \label{expansion3}
\end{equation}
Given this expansion, provided that $\sqrt{d_{nT}}/T^{2} \to 0$, the distribution of $\tilde{\bbeta}$ can be approximated by $N(\bbeta_{0},A^{-1}\Sigma_{nT}A^{-1})$. 
Statistical inference on $\bbeta_{0}$ can be implemented by using this normal approximation. 
As usual, since the matrices $A$ and $\Sigma_{nT}$ are unknown, we have to replace them by suitable estimators. 
A natural estimator of $A$ is $\hat{A}$ defined in (\ref{fe2}), which is in fact consistent (i.e., $\hat{A} \stackrel{\Pr}{\to} A$) as long as $T=T_{n} \to \infty$ (see Lemma \ref{lem2} in Appendix \ref{Appendix A}). 
We thus focus on the problem of estimating the matrix $\Sigma_{nT}$. 
We consider the estimator suggested by \cite{A87}:  
\begin{equation*}
\hat{\Sigma}_{nT} = \frac{1}{n^{2}}\sum_{i=1}^{n} \left \{ \frac{1}{T}\sum_{t=1}^{T} (\bx_{it}-\bar{\bx}_{i})\hat{\epsilon}_{it} \right \}^{\otimes 2},
\end{equation*}
where $\hat{\epsilon}_{it} = y_{it} - \bar{y}_{i} - (\bx_{it}-\bar{\bx}_{i})'\tilde{\bbeta}$ and $\tilde{\bbeta}$ is a suitable estimator of $\bbeta_{0}$. Proposition \ref{prop2} establishes the rate of convergence of $\hat{\Sigma}_{nT}$. All the technical proofs of this section are gathered in Appendix \ref{Appendix B}.

\begin{proposition}
\label{prop2}
Suppose that conditions (A1)-(A3) are satisfied. Let $\tilde{\bbeta}$ be any estimator of $\bbeta_{0}$ such that $\| \tilde{\bbeta} - \bbeta_{0} \| = O_{\Pr}[ \max \{ d_{nT}^{-1/2}, T^{-1}  \}]$. Letting $T=T_{n} \to \infty$ as $n \to \infty$, we have:
\begin{equation*}
\hat{\Sigma}_{nT} = \Sigma_{nT} + O_{\Pr}[\{ \max \{ n^{-1/2}T^{-1} d_{nT}^{-1/2},n^{-1/2}d_{nT}^{-1} \}].
\end{equation*}
In particular, as long as $T=T_{n} \to \infty$, we have $\| \hat{\Sigma}_{nT} - \Sigma_{nT} \|_{\op} = o_{\Pr}(d_{nT}^{-1})$.
\end{proposition}

\begin{remark}[Initial estimators]
In Proposition \ref{prop2},  $\tilde{\bbeta}$ can be the FE estimator. Actually, using a bias corrected estimator instead does not change the rate of convergence of $\hat{\Sigma}_{nT}$. However, in the case where $T$ is relatively small, the FE estimator can be severely biased,  which may affect the finite sample performance of $\hat{\Sigma}_{nT}$. Thus, it is generally recommended to use a bias corrected estimator of $\bbeta_{0}$ in the construction of $\hat{\Sigma}_{nT}$. 
\end{remark}

\begin{remark}[Relation with \cite{H07b}]
The results of \cite{H07b} are not directly applicable to the asymptotic properties of $\hat{\Sigma}_{nT}$ described in Proposition \ref{prop2}. This is because the individual dummies that control for fixed effects are not included in the model of \cite{H07b} (see also previous Remark \ref{rem: Hansen}). Thus, in contrast with \cite{H07b}, $\sqrt{n} d_{nT}(\hat{\Sigma}_{nT} - \Sigma_{nT})$ is not asymptotically normal with mean zero unless $\sqrt{d_{nT}}/T \to 0$. The main reason is that $\hat{\Sigma}_{nT}$ has a bias of order $n^{-1/2} T^{-1} d_{nT}^{-1/2}$ due to using $\bar{\bx}_{i}$ instead of $\E[ \bx_{i1} \mid c_{i}]$. This bias appears even if we could use $\epsilon_{it}$ in place of $\hat{\epsilon}_{it}$. However, for  inference purposes, the rate of convergence given in Proposition \ref{prop2} is sufficient, and we do not consider the bias correction to $\hat{\Sigma}_{nT}$. 
\end{remark}

Assume now that $\Sigma$ is nonsingular in either case of $\E[ \tilde{\bx}_{it} \epsilon_{it} \mid c_{i}] = \bm{0}$ a.s. or not. Consider testing the null hypothesis $H_{0}: R\bbeta_{0} = \bm{r}$, where $R$ is a $q \times p$ matrix with rank $q$ ($q \leq p$), and $\bm{r} \in \mathbb{R}^{q}$ is a constant vector. Suppose that we have a bias corrected estimator $\tilde{\bbeta}$ having the expansion (\ref{expansion3}). Then under the null hypothesis, the $t$-type (for $q=1$) and  Wald-type statistics
\begin{equation}
\hat{t}=\frac{R\tilde{\bbeta} - r}{\sqrt{R\hat{A}^{-1}\hat{\Sigma}_{nT}\hat{A}^{-1}R'}} \ (q=1) \  \text{and} \ 
\hat{F}=(R\tilde{\bbeta} - \bm{r})'[R\hat{A}^{-1}\hat{\Sigma}_{nT}\hat{A}^{-1}R']^{-1}(R\tilde{\bbeta} - \bm{r}),
\label{wald}
\end{equation} 
converge in distribution to $N(0,1)$ and $\chi_{q}^{2}$, respectively, provided that $\sqrt{d_{nT}}/T^{2} \to 0$. For example, for $q=1$, provided that $\sqrt{d_{nT}}/T^{2} \to 0$,
\begin{align*}
\hat{t}=\frac{\sqrt{d_{nT}}(R\tilde{\bbeta} - r)}{\sqrt{R\hat{A}^{-1}(d_{nT}\hat{\Sigma}_{nT})\hat{A}^{-1}R'}} =\frac{\sqrt{d_{nT}}(R\tilde{\bbeta} - r)}{\sqrt{RA^{-1}\Sigma A^{-1}R'}} (1+o_{\Pr}(1))
\stackrel{d}{\to} N(0,1),
\end{align*}
where the second equality is due to the fact that $\hat{A} = A + o_{\Pr}(1)$ and $d_{nT} \hat{\Sigma}_{nT} = d_{nT} \Sigma_{nT} + o_{\Pr}(1) = \Sigma + o_{\Pr}(1)$. 
The same machinery applies to the Wald-type statistic.
Importantly, in the construction of $t$-type or Wald-type statistics, we do not need any knowledge on the degree of model misspecification.

The resulting estimator $\hat{A}^{-1} \hat{\Sigma}_{nT} \hat{A}^{-1}$ of the covariance matrix $A^{-1} \Sigma_{nT} A^{-1}$ is often called the clustered covariance matrix (CCM) estimator in the literature. The CCM estimator is popular in empirical studies. The appealing point of the CCM estimator, as discussed in \cite{H07b}, is the fact that it is free from any user-chosen parameter such as a bandwidth. The previous discussion shows that inference using a suitable bias corrected estimator and the CCM estimator is ``robust'' to model misspecification. 

\subsection{Cross section bootstrap}

Bootstrap is generally used as a way to estimate the distribution of a statistic \citep[see][for a general reference on bootstrap]{H01}. For panel data, how to implement bootstrap is not necessarily apparent.
See \cite{K08} for some possibilities in bootstrap resamplings for panel data. 
We here study, among them, the cross section bootstrap.

Let $\bm{z}_{i} = (\bm{z}_{i1}',\dots,\bm{z}_{iT}')'$ with $\bm{z}_{it} = (y_{it},\bx_{it}')'$. The cross section bootstrap randomly draws $\bm{z}_{1}^{*},\dots,\bm{z}_{n}^{*}$ from $\{ \bm{z}_{1},\dots,\bm{z}_{n} \}$ with replacement.
The bootstrap FE estimator $\hat{\bbeta}^{*}$ is defined by $\hat{\bbeta}$ in (\ref{fe}) with $\bm{z}_{1},\dots,\bm{z}_{n}$  replaced by $\bm{z}_{1}^{*},\dots,\bm{z}_{n}^{*}$. By this definition, we can express $\hat{\bbeta}^{*}$ as the following form:
\begin{align}
\hat{\bbeta}^{*} &=  
\left \{ \frac{1}{nT} \sum_{i=1}^{n} w_{ni} \sum_{t=1}^{T} (\bx_{it} - \bar{\bx}_{i}) (\bx_{it} - \bar{\bx}_{i})' \right \}^{-1}  \left \{ \frac{1}{nT} \sum_{i=1}^{n} w_{ni} \sum_{t=1}^{T} (\bx_{it} - \bar{\bx}_{i}) (y_{it} - \bar{y}_{i}) \right \} \notag \\
&=: (\hat{A}^{*})^{-1} \hat{S}^{*}, \label{bfe}
\end{align}
where $w_{ni}$ is the number of times that $\bm{z}_{i}$ is ``redrawn'' from $\{ \bm{z}_{1},\dots,\bm{z}_{n} \}$. 
The vector $(w_{n1},\dots,w_{nn})'$ is independent of $\{ c_{i}, \bm{z}_{it} : i \geq 1, t \geq 1 \}$ and multinomially distributed with parameters $n$ and (probabilities) $n^{-1},\dots,n^{-1}$. 

We need to prepare some notation and terminology. Let $\Pr_{W}$ denote the probability measure with respect to $w_{n1},\dots,w_{nn}$. 
Let $\E_{\Pr_{W}}[ \cdot ]$ denote the expectation under $\Pr_{W}$. 
Given a vector valued statistic $\Delta_{n}$ depending on both $c_{1},\dots,c_{n},\bm{z}_{1},\dots,\bm{z}_{n}$ and $w_{n1},\dots,w_{nn}$,  
and a deterministic sequence $a_{n} > 0$, we write ``$\Delta_{n} = o_{\Pr_{W}}(a_{n})$ in probability'' if for every $\epsilon > 0$ and $\delta > 0$,
\begin{equation*}
\Pr ( \Pr_{W}(a_{n}^{-1}\| \Delta_{n} \| > \epsilon) > \delta ) \to 0 
\end{equation*}
as $n \to \infty$, and ``$\Delta_{n} = O_{\Pr_{W}}(a_{n})$ in probability'' if for every $\delta > 0$ and $\eta > 0$, there exists a constant $C>0$ such that
\begin{equation*}
\Pr ( \Pr_{W}(a_{n}^{-1}\| \Delta_{n} \| \geq C) > \delta ) \leq \eta,  
\end{equation*}
for all $n \geq 1$ (recall that $T=T_{n}$). 

We are now in position to state the main result of this section.
\begin{proposition}
\label{prop3}
Suppose that conditions (A1)-(A3) are satisfied. Letting $T=T_{n} \to \infty$ as $n \to \infty$, we have:  
\begin{equation}
\hat{\bbeta}^{*} - \hat{\bbeta} 
= A^{-1} \left \{ \frac{1}{nT} \sum_{i=1}^{n} (w_{ni}-1) \sum_{t=1}^{T} \tilde{\bx}_{it} \epsilon_{it} \right \} + O_{\Pr_{W}}[n^{-1/2}\max\{d_{nT}^{-1/2},T^{-1}\}], \label{BahadurB}
\end{equation}
in probability. Therefore, provided that $\Sigma$ is nonsingular,  we have: 
\begin{equation}
\sup_{\bx \in \mathbb{R}^{p}}|\Pr_{W}\{ \sqrt{d_{nT}} ( \hat{\bbeta}^{*} - \hat{\bbeta} ) \leq \bx \}  - \Pr\{ N(0, A^{-1} \Sigma A^{-1}) \leq \bx \} | \stackrel{\Pr}{\to} 0,
\label{bootd}
\end{equation}
where the inequalities are interpreted coordinatewise. 
\end{proposition}

Interestingly, Proposition \ref{prop3} shows that despite the fact that the original FE estimator has the incidental parameters bias of which the top order is $T^{-1}$, as shown in Proposition \ref{prop1}, the bootstrap distribution made by applying the cross section bootstrap to the FE estimator does not have the incidental parameters bias. As a consequence, the bootstrap distribution approaching the centered normal distribution holds without any specific growth condition on $T$. In fact, this is not surprising. The main source of the incidental parameters bias comes from the term $n^{-1} \sum_{i=1}^{n} \bar{\tilde{\bx}}_{i} \bar{\epsilon}_{i}$, which is linear in the cross section dimension. The bootstrap analogue of this term is thus $n^{-1} \sum_{i=1}^{n} w_{ni} \bar{\tilde{\bx}}_{i} \bar{\epsilon}_{i}$, so that the difference of these terms has mean zero with respect to $\Pr_{W}$. The same machinery applies to the term $n^{-1} \sum_{i=1}^{n} \bar{\tilde{\bx}}_{i} \bar{\tilde{\bx}}_{i}'$, so that the incidental parameters bias is completely removed in the bootstrap distribution. 

The previous discussion also has the following implication: the cross section bootstrap can not be used as a way to correct the incidental parameter bias. Recall that in the cross section case, the bootstrap can be used to correct the second order bias coming from the quadratic term; here the incidental parameters bias comes from the terms linear in the cross section dimension, so that the cross section bootstrap does not work as a way to correct the bias.

Proposition \ref{prop3} shows that, for $1 \leq a \leq p$ fixed and $\alpha \in (0,1)$,
\begin{align*}
\hat{q}(\alpha) &:= \text{conditional $\alpha$-quantile of $(\hat{\bbeta}^{*}-\hat{\bbeta})^{a}$} \\
&:=\inf \{ b \in \mathbb{R} : \Pr_{W}((\hat{\bbeta}^{*}-\hat{\bbeta})^{a} \leq b)  \geq \alpha \},
\end{align*}
we have 
\begin{equation*}
\hat{q}(\alpha) = \frac{\sqrt{(A^{-1} \Sigma A)^{aa}}\Phi^{-1}(\alpha) }{\sqrt{d_{nT}}} + o_{\Pr}(d_{nT}^{-1/2}),
\end{equation*}
where $\Phi (\cdot )$ is the distribution function of the standard normal distribution (recall that $z^{a}$ is the $a$-th element of a vector $\bm{z}$ and $A^{ab}$ denotes the $(a,b)$-element of a matrix $A$). 
Suppose that we have a bias corrected estimator $\tilde{\bbeta}$ having the expansion (\ref{expansion3}). Then by a standard argument, we can deduce that 
\begin{equation*}
\Pr (  \beta_{0}^{a} \leq \tilde{\beta}^{a} + \hat{q}(\alpha)) = \alpha + o(1),
\end{equation*}
provided that $\sqrt{d_{nT}}/T^{2} \to 0$.
Note that $\hat{q} (\alpha)$ can be computed with any precision by using simulation. Moreover, the computation of $\hat{q} (\alpha)$ does not require any knowledge on the speed of $d_{nT}$, and in this sense the cross section bootstrap is robust to model misspecification.

An analogous result holds for the HPJ estimator (see Section 4.1). 

\begin{corollary}
\label{cor4}
Suppose that conditions (A1)-(A3) are satisfied. Let $\hat{\bbeta}_{1/2}^{*}$ denote the HPJ estimator based on the bootstrap sample $\{ \bm{z}_{1}^{*},\dots,\bm{z}_{n}^{*} \}$.  Letting $T=T_{n} \to \infty$ as $n \to \infty$, we have:
\begin{equation*}
\hat{\bbeta}_{1/2}^{*} - \hat{\bbeta}_{1/2} 
= A^{-1} \left \{ \frac{1}{nT} \sum_{i=1}^{n} (w_{ni}-1) \sum_{t=1}^{T} \tilde{\bx}_{it} \epsilon_{it} \right \} + O_{\Pr_{W}}[n^{-1/2}\max\{d_{nT}^{-1/2},T^{-1}\}], 
\end{equation*}
in probability. 
Therefore, provided that $\Sigma$ is nonsingular, we have:
\begin{equation*}
\sup_{\bx \in \mathbb{R}^{p}}|\Pr_{W}\{ \sqrt{d_{nT}} ( \hat{\bbeta}_{1/2}^{*} - \hat{\bbeta}_{1/2} ) \leq \bx \}  - \Pr\{ N(0, A^{-1} \Sigma A^{-1}) \leq \bx \} | \stackrel{\Pr}{\to} 0.
\end{equation*}
\end{corollary}

This corollary follows directly from the definition of the HPJ estimator and Proposition \ref{prop3}, and hence we omit the proof. 
Basically, the conclusion of Corollary \ref{cor4} holds for other reasonable bias corrected estimators. We do not attempt to encompass generality in this direction. 

Analogous results hold for pivotal statistics. Because of the space limitation, we push the formal results on pivotal statistics to the Appendix (Appendix \ref{Appendix}).

\begin{remark}[Higher order properties]
The higher order properties of the cross section
bootstrap will be very complicated in this setting and we do not attempt to study them
here. However, it is of interest to quantify the order of the convergence in, say, (\ref{bootB}) in Appendix \ref{appendix},
which is left to future research.
\end{remark}

\begin{remark}[Relation to the previous literature]
There are some earlier works on the bootstrap for panel data. \cite{DMB04} called the cross section bootstrap in this paper the ``block bootstrap'' and studied its numerical  properties by using simulations, but did not study its theoretical properties.
\cite{K08} developed the asymptotic properties of the cross section bootstrap when the strict exogeneity is met, hence excluding the possibility that 
the incidental parameters bias appears.  \cite{G11} studies the asymptotic properties of the moving block bootstrap for panel data, which resamples the data in the time series dimension and hence is different from the cross section bootstrap.
Importantly, while \cite{G11} allowed for cross section dependence which we exclude here, she assumed that the number of time periods, $T$, is sufficiently large (typically $n/T \to 0$) so that the incidental parameters bias does not appear. Lastly, \cite{DJ09} proposed to use the cross
section bootstrap for inference for nonlinear panel data models such as panel probit
models, but did not give any theoretical result.
The asymptotic properties of the cross section bootstrap were largely unknown when the incidental parameters bias appears, even without model misspecification, and the results in this section contribute to filling this void and give useful suggestions to  empirical studies.  
\end{remark}

\begin{remark}[Weighted bootstrap]
In (\ref{bfe}), the weights $w_{n1},\dots,w_{nn}$ are multinomially distributed. It is possible to consider other weights. A perhaps simplest variation is to draw independent weights 
$w_{1},\dots,w_{n}$ from a common distribution with mean $1$ and variance $v > 0$, which corresponds to the {\em weighted  bootstrap} \citep[see, for example,][]{MK05}.
Let $\hat{\bbeta}^{W}$ denote (\ref{bfe}) with $w_{ni}$ replaced by these independent weights $w_{i}$. Then the conclusion of Proposition \ref{prop3} holds with $\hat{\bbeta}^{*}$ replaced by $\hat{\bbeta}^{W}$ and $\sqrt{d_{nT}}$ replaced by $\sqrt{d_{nT}/v}$. Since the proof is completely analogous, we omit the details for brevity. 
\end{remark}

\begin{remark}[Covariance matrix estimation]
So far, we have discussed the distributional properties of the cross section bootstrap.
Given Proposition \ref{prop3}, it is natural to estimate the asymptotic covariance matrix $A^{-1}\Sigma_{nT} A^{-1}$ by the conditional covariance matrix of $\hat{\bbeta}^{*} - \hat{\bbeta}$. However, since convergence in distribution does not imply moment convergence, Proposition \ref{prop3} does not guarantee that 
$d_{nT} \E_{\Pr_{W}}[ (\hat{\bbeta}^{*} - \hat{\bbeta})(\hat{\bbeta}^{*} - \hat{\bbeta})'] \to A^{-1}\Sigma A^{-1}$ in probability.
See \cite{S92} for some examples of inconsistency of bootstrap variance estimators. In linear regression models with pure cross section or time series data, \cite{GW05} discussed bootstrap-based covariance matrix estimation. 
In \cite{GW05}, they made a modification to the bootstrap least square estimator, to guarantee the bootstrap estimator to satisfy
a uniform integrability condition. Their modification is a sort of ``trimming'' to the second moment matrix. 
By doing so, they established the consistency of the bootstrap covariance matrix estimator.
In the present panel data case, while their modification is straightforward to adapt to the FE estimator, it is not clear whether it actually works at the proof level. In Proposition \ref{prop3}, what we have done are: (i) to implement higher order stochastic expansions to the FE estimator to exhaust the incidental parameters bias of any order; (ii) to implement (i) to the bootstrap analogue; (iii) to eliminate the incidental parameters bias by taking the difference. 
In step (ii), we need expansions of $\hat{A}^{*}$, or more precisely $(\hat{A}^{*})^{-1}$. Simply bounding $\| (\hat{A}^{*})^{-1} \|_{\op}$ from above, as \cite{GW05} did in Step 3 of the proof of their Theorem 1, will leave the bias in the bootstrap distribution and cause a problem in establishing uniform integrability. We leave this as an open problem. 
\end{remark}

\section{Numerical examples}
\label{sec: MC}

In this section, numerical examples to illustrate the methods discussed in this paper are provided. Both Monte Carlo experiments and a real data analysis are presented.
All the numerical experiments were performed on the statistical software {\tt R} \citep{R08}. 
Computer programs to replicate the numerical analyses are available from the authors.

\subsection{Simulation experiments}

We use several different designs of simulation experiments to assess the finite sample performance of the estimates and inference procedures discussed in the previous sections. 
In the first design, as a benchmark, we analyze the estimates and inference procedures under correct specification. The true data generating process (DGP) in the first design follows a panel AR(1) model and we (correctly) fit panel AR(1). 
The next four models are designed to study the estimates and inference procedures under misspecification.\footnote{An additional simulation design where the true DGP follows a panel EXPAR model is analyzed in Appendix \ref{additional simulation}.}
In the second design, the true DGP follows a panel AR(2) model (see Example 1 in Section \ref{sec: interpretation}); in the third design we extend the true AR(2) DGP and include two lags of exogenous regressors;  and in the last design, the true DGP follows a random coefficient AR(1) model (see Example 2 in Section \ref{sec: interpretation}). 
In each of these cases, we incorrectly fit a panel AR(1) model and estimate the slope parameter. Note that the first three cases correspond to ``$\E[ \tilde{\bx}_{it} \epsilon_{it} | c_{i} ] = \bm{0}$''  case and the last case corresponds to ``$\E[ \tilde{\bx}_{it} \epsilon_{it} | c_{i} ] \neq \bm{0}$'' case. The following sample sizes are considered: $n \in \{50, 100, 200\}$ and $T \in \{12, 16, 20, 24\}$. The number of Monte Carlo repetitions is 2,000. 

We consider four different estimates: the FE, GMM \citep{AB91,AB95},  HPJ estimates, and the bias corrected estimate proposed by \citet[][equation (6)]{HK02} (HK).\footnote{The HK estimate is not designed to reduce the bias when misspecification is present. We report these results only for comparison reasons.} More precisely, the GMM estimate we compare is the one-step GMM estimate formally defined in equation (8) in \cite{AA03}. 
We also investigate the small sample properties of the inference procedures, paying particular attention to the empirical coverage probability. The nominal coverage is 95\%.
Note that we are trying to construct a 95\% confidence interval for the pseudo-true parameter. 
We consider and compare the following options for inference on the pseudo-true parameter:
\begin{longtable}{c|c|c}
\hline
Option & Centering & Inference procedure \\
\hline 
FE-CCM & FE estimate & CCM with FE estimate \\
\hline
HK & HK estimate &  * \\
\hline 
GMM & GMM estimate & standard GMM variance estimate \\
\hline 
HPJ-CCM & HPJ estimate & CCM with HPJ estimate \\
\hline 
HPJ-FEB & HPJ estimate & CSB applied to FE estimate \\
\hline 
HPJ-HPJB & HPJ estimate & CSB applied to HPJ estimate \\
\hline 
HPJ-HPJPB & HPJ estimate & pivotal-CSB applied to HPJ \\
\hline 
\end{longtable}
In the HK option, we use a simple consistent estimate of the asymptotic variance based on the formula in \citet[][p.1645]{HK02}. 
Here CSB refers to ``cross section bootstrap''. The number of bootstrap repetitions in  each case is 1,000. Note that the bias and variance formula in \cite{HK02} are not valid under the misspecified settings below (nevertheless the HK estimate is consistent for the pseudo-true parameter when $n$ and $T$ jointly go to infinity as the difference between the FE and HK estimates are $O(T^{-1})$), hence it is natural to expect the HK option does not perform well in those cases (as it is not designed for covering model misspecification). Also it is expected that the FE estimate suffers from the incidental parameters bias and hence the FE-CCM option will not work well. The GMM estimate is formally not known to be consistent for the pseudo-true parameter here, but the result of \cite{O08} suggests that it is the case (and hence comparison with the GMM estimate makes some sense).\footnote{However, \cite{O08} used a different set of assumptions and the sequential  asymptotic scheme where $n \to \infty$ first and then $T \to \infty$, so his result is not directly transferred to our case. It is of interest to study the asymptotic properties of the GMM estimate under misspecification when $n$ and $T$ jointly go to infinity, which is left to future research. } 
However, it is expected that the GMM option will not perform well as it is not designed for covering model misspecification. 
The last four options are expected to work reasonably well at least when $T$ is moderately large. The precise description of the HPJ-HPJPB option is the following: in the HPJ-HPJPB, we use the HPJ estimate as the center, and apply the cross section bootstrap to the $t$-statistic. The $t$-statistic here is constructed by using the HPJ estimate together with the CCM estimate, and the initial estimate in construction of the CCM estimate is the HPJ estimate.

\subsubsection*{Panel AR models}
In the first example, the true data generating process (DGP) is a panel AR(1) model:
\begin{equation*}
y_{it} = c_{i} + \phi y_{i,t-1} + u_{it},
\end{equation*}
where $u_{it} \sim \ \text{i.i.d.} \ t (10)$ ($t$ distribution with $10$ degrees of freedom), $c_{i} \sim \ \text{i.i.d.} \ U(-0.5,0.5)$, and $\phi=0.8$. 
In generating $y_{it}$ we set $y_{i,-500} = 0$ and discard the first 500 observations, using the observations $t = 0$ through $T$ for estimation. 
% This ensures that the results are not unduly influenced by the initial values of the process.
In this case,  we correctly fit panel AR(1) and there is no misspecification in the model. The results are collected in Table \ref{table.mc.0}.

{\bf Table \ref{table.mc.0}}: The FE estimate has large bias. The HK and GMM estimates are biased when $T$ and $n$ are small, but the bias decreases as $T$ and $n$ are large, respectively. The HPJ estimate is approximately unbiased. Regarding inference, the empirical coverage of FE-CCM is close to zero, likely due to the large bias in the FE estimate. HK is  under coverage. GMM, as expected, has a good coverage property in this case, especially for large $n$. It is important to notice that the robust inference procedures, especially HPJ-HPJB and HPJ-HPJPB, also have good coverage under no model misspecification.

\vspace{0.25cm}

In the next example, the true DGP follows a panel AR(2) model: 
\begin{equation*}
y_{it} = c_{i} + \phi_{1} y_{i,t-1}+ \phi_{2} y_{i,t-2} + u_{it},
\end{equation*}
where $u_{it} \sim \ \text{i.i.d.} \ t (10)$ and $c_{i} \sim \ \text{i.i.d.} \ U(-0.5,0.5)$.
Two cases for the parameters $\phi_{1}$ and $\phi_{2}$ are considered: $\phi_{1}=\phi_{2}=0.4$ and $\phi_{1}=\phi_{2}=-0.4$. 
% In generating $y_{it}$ we also set $y_{i,-500} = 0$.

Despite that the true DGP is a panel AR(2) model, suppose that we incorrectly fit a panel AR(1) model and estimate the slope parameter. 
When $\phi_{1}=\phi_{2}=0.4$, the pseudo-true parameter is $\beta_{0}=0.67$, and when $\phi_{1}=\phi_{2}=-0.4$, $\beta_{0}=-0.28$. The results for these two cases are presented in Tables \ref{table.mc.1} and \ref{table.mc.2},  respectively. 

{\bf Table \ref{table.mc.1}}: The FE and GMM estimates are severely biased. The HK estimate is also biased as it is not designed for handling the case where misspecification is present. The HPJ estimate is able to reduce the bias substantially. The bias is  small even for modest $T$, such as $T=20$ and $T=24$. Regarding the standard deviations, the HPJ estimate has slight variance inflation relative to the FE estimate in the finite sample (which is also observed in \cite{DJ09} in a different context of estimation of nonlinear panel data models such as panel probit models). 

As for the empirical coverage, the FE-CCM, HK, and GMM options perform poorly due to the fact  that the  FE, HK, and GMM estimates are largely biased. 
The other options, namely, HPJ-CCM, HPJ-FEB, HPJ-HPJB and HPJ-HPJPB, perform reasonably well, but the HPJ-HPJPB option, as expected, seems to be the best. 
The coverage of  HPJ-HPJPB  is about 92\% for $n=100$ and $T=24$, and close to the nominal $95\%$.

{\bf Table \ref{table.mc.2}}:  In this case, the incidental parameters bias is small and all the options perform relatively well. 
However, regarding coverage, as $n$ grows,  the performance of the FE-CCM, HK, and GMM options deteriorates since the ratio between the bias and the standard deviation becomes larger in each case. On the other hand, the other options, HPJ-CCM, HPJ-FEB, HPJ-HPJB and HPJ-HPJPB, perform well regardless of the combination of $(n,T)$. 

{\bf Table \ref{table.mc.5}}: To extend the AR(2) example, we include exogenous regressors in the true DGP. In this case, the true DGP is as following: 
\begin{equation*}
y_{it} = c_{i} + \phi_{1} y_{i,t-1}+ \phi_{2} y_{i,t-2} + \rho_{1} x_{i,t}+ \rho_{2} x_{i,t-1} + u_{it},
\end{equation*}
where $u_{it} \sim \ \text{i.i.d.} \ t (10)$, $c_{i} \sim \ \text{i.i.d.} \ U(-0.5,0.5)$, and $x_{it} \sim \ \text{i.i.d.} \ N(0,1)$. 
Finally, the parameters $\phi_{1} = \phi_{2} =0.4$ and $\rho_{1}=\rho_{2}=0.5$. In this case, we incorrectly fit a panel model with regressors $(y_{i,t-1},x_{i,t-1})$ and report estimates of the slope parameter of the autoregressive term. This pseudo-true parameter value on the autoregressive term  is $0.73$. The results for this case are presented in Tables \ref{table.mc.5}. 
The results in Table \ref{table.mc.5} show evidence that the proposed methods are effective in finite sample. The bias of the HPJ estimator is small, on the other hand the bias of other estimators are large. Regarding inference, given the bias in the FE, HK and GMM, their respective coverage rates are poor. But, the empirical coverage of HPJ-HPJB and HPJ-HPJPB are close to the nominal.

\subsubsection*{Random coefficients AR model}

In the fourth example, the true DGP is 
\begin{equation*}
y_{it} = c_{i} y_{i,t-1}  + u_{it},
\end{equation*}
where  $u_{it} \sim \ \text{i.i.d.} \ N(0,1)$, $c_{i} \sim \ \text{i.i.d.} \ U(0,0.9)$. 
This model appears in Example 3 in
Section \ref{sec: interpretation}. As before, we incorrectly fit a panel AR(1) model and estimate the slope
parameter. The value of the pseudo-true parameter is $\beta_{0}=0.56$. The simulation results for this case are presented in Table \ref{table.mc.4}. 

{\bf Table \ref{table.mc.4}}: As in the previous cases, the FE estimate is largely biased and FE-CCM
performs poorly due to the bias. Note that the decreasing speed of the standard
deviation for the FE estimate as $T$ grows is relatively slow, which would reflect the
fact that the convergence rate of the FE estimate (without the bias part) under this
DGP is $n^{-1/2}$ and not $(nT)^{-1/2}$ (the asymptotic variance of the FE estimate consists
of the part decreasing like $O(n^{-1})$ and also the part decreasing like $O\{ (nT)^{-1} \}$, so even
in this case, it is not surprising that the standard deviation of the FE estimate in the
finite sample slowly decreases as $T$ increases). The HPJ estimate is able to largely
remove the bias, nevertheless there is slight variance inflation in the finite sample. In this example, the
HK and GMM estimates are able to reduce the bias, to some extent, for large time series. However, the empirical coverage of the HK and GMM options is still poor. Lastly, among the inference procedures, HPJ-HPJPB works particularly well.

\subsection{Real data analysis}

In this section, we apply the procedures discussed in the previous sections  to a model of unemployment dynamics at the U.S. state level. \cite{BC05} and \cite{Baglan10} studied this subject using a dynamic panel data model. In particular, \cite{BC05} modeled the current unemployment rate $(U_{it})$ as a function of both lagged unemployment rate and economic growth rate $(G_{i,t-1})$. 
In addition, to capture state specific effects, the model includes state individual intercepts $\eta_{i}$. The model can be written as follows: 
\begin{equation}
\label{eq.app1}
U_{it} = \gamma U_{i,t-1} + \beta G_{i,t-1} + \eta_{i} +  \varepsilon_{it},
\end{equation}
or equivalently
\begin{equation}
\label{eq.app2}
U_{it} - U_{i,t-1}= (\gamma-1) (U_{i,t-1} -\alpha_{i}) + \beta (G_{i,t-1} - \delta) + \varepsilon_{it},
\end{equation}
where $(1-\gamma) \alpha_{i}-\beta \delta = \eta_{i}$ and $\varepsilon_{it}$ is an innovation term. The model described in equation (\ref{eq.app2}) shows that changes in unemployment rate are determined by two observable components. The first is an adjustment of the unemployment rate towards a ``natural'' or ``equilibrium'' rate of unemployment, $\alpha_{i}$. This rate of unemployment equilibrium is allowed to vary across states. Moreover, the speed of adjustment of the unemployment rate towards the state specific equilibrium is equal to $1-\gamma$. Similarly, the second factor determining changes in unemployment rate is a deviation of the economic growth rate around a constant equilibrium.

The data for the unemployment rate are taken from the U.S. Bureau of Labor Statistics for the 1976--2010 period. Data for the state product are per capita personal income (thousands of dollars) from the U.S. Bureau of Economic Analysis deflated by annual implicit price deflator. The economic growth rate is taken to be the relative growth of the state product. Data are available for all 50 U.S. states and Washington D.C. We have  a panel data set of 51 subjects over 35 years ($n=51$ and $T=35$).\footnote{We performed unit root tests in both series and the null of unit root are reject at standard significance levels for all samples considered.}

We consider and compare several different inference procedures: the FE estimate with its associated CCM estimate (FE-CCM); and the HPJ estimate with inference using its associated CCM estimate and pivotal cross-section bootstrap, which we denote by HPJ-CCM and HPJ-HPJPB,  respectively.\footnote{The number of bootstrap repetitions for HPJ-HPJPB is 1,000.} For comparison, we also report the results for the one-step GMM, and the two-stage least squares (TSLS) \citep{AH82} estimates.\footnote{For  TSLS we use $U_{i,t-2}$ and $G_{i,t-2}$ as instruments.}  We present 90\% and 95\% confidence intervals in all cases. Note here that the HPJ-CCM and HPJ-HPJPB options are misspecification robust, so they provide meaningful inference even when misspecification is present.

% \textbf{We present coefficient estimates for several different estimators. The resulting estimates and corresponding 90\% and 95\% confidence intervals are presented in Table \ref{t.app}. 
% We estimate model (\ref{eq.app1}) using the simple pooled standard ordinary least squares (POLS) that ignores the individual specific effects, $\eta_{i}$. In addition, we have results for the least squares fixed effects estimator (FE-OLS), and the \cite{AH82} two-stage least squares (TSLS). Both FE-OLS and TSLS are able to control for the presence of fixed effects. Confidence intervals for POLS, FE-OLS, and TSLS estimators are computed using their corresponding asymptotic variance-covariance matrices. 
% Finally, we also present results for the bias correction method discussed previously, the \cite{DJ09} half-panel jackknife (HPJ). For the HPJ estimator we report confidence intervals computed using both the clustered covariance matrix and cross-section bootstrap. In the first case, we use the HPJ estimator to compute the clustered covariance matrix, we denote this column as HPJ-CCM. For the second, we construct the confidence intervals using the bootstrap HPJ estimator centered at the HPJ estimate from the entire sample. This column is denoted HPJ-HPJB. The point estimates for these two columns are identical (HPJ), only the methodology to calculate confidence intervals differs. The number of replications for the bootstrap estimator is equal to 1,000.}

The results for point estimates and confidence intervals are collected in Table \ref{t.app} Panel A for 1976-2010. The HPJ estimate of $\gamma$ (columns HPJ-CCM and HPJ-HPJPB) is 0.830, which implies that the speed of adjustment is approximately 17\% per year. 
The FE estimate is 0.790, implying a speed of convergence around 21\%, and the GMM estimate is 0.80 with speed of approximately 20\%. Finally, the TSLS estimate  for $\gamma$ is smaller than other estimates, and the speed of adjustment is larger, close to 67\%.
% In terms of point estimates, the  HPJ has a result that is closer to FE-CCM than TSLS. 
Regarding confidence intervals for $\gamma$, FE-CCM, HPJ-CCM, HPJ-HPJPB and GMM have confidence intervals with similar length, while  the confidence intervals of TSLS are substantially larger than the other options.
The results for HPJ-CCM and HPJ-HPJPB are very similar. 
% Overall, the results for the bias correction estimates presented in this paper show evidence of partial adjustment towards equilibrium, associated with a slower speed of adjustment of the unemployment rate towards the equilibrium than the estimates available in the literature. 
 
Now we move our attention to the economic growth rate variable. The HPJ estimate of $\beta$ is $-0.079$, which in absolute value is slightly smaller than the FE-CCM and GMM estimates but larger than the TSLS estimate. 
The confidence intervals of HPJ-CCM and HPJ-HPJPB are similar.
The FE estimate of $\beta$ is $-0.088$. This estimate is accompanied with relatively narrow confidence intervals and zero is not included in the intervals. However, the TSLS estimate of $\beta$ $-0.003$, and zero is inside both the 90\% and 95\% confidence intervals in the TSLS option.  
%  Overall, these results show evidence that the average value impact of reduction in unemployment by increasing lag of economic growth are larger for the  FE-CCM estimation procedures when compared to that of bias corrected procedures. In addition, the confidence interval estimates of FE-CCM, HPJ-CCM and HPJ-HPJPB show evidence that, in fact, increments in economic growth are associated with smaller unemployment rates.

For robustness purposes we use different subsamples to estimate the model. We consider two subsamples: (i) 1976--2001; and 
(ii) 1976--1991. The results are, respectively, collected in Panels B and C of Table \ref{t.app}.

In the first robustness exercise we drop the last 9 years of observations and consider a subsample of 26 years, from 1976 to 2001. These results are displayed in Panel B. They show point estimates for both $\gamma$ and $\beta$ close to those in Panel A, although slightly larger in absolute value. Confidence intervals are also similar to those in Panel A. 

Lastly, we consider an even smaller subsample with 19 years, from 1976 to 1991. The results are presented in Table \ref{t.app} Panel C.  Except for TSLS, the point estimates of $\gamma$ are smaller than those in the full sample case and the confidence intervals shift to the left. In particular, the FE and GMM estimates of $\gamma$ decrease substantially, from 0.790 and 0.800 in the full sample case to 0.676 and 0.669, respectively. 
The HPJ estimate also decrease to 0.721 from 0.830 in the full sample case, but not so largely as GMM.
Moreover, FE-CCM, HPJ-CCM and HPJ-HPJ-HPJPB have substantially narrower confidence intervals than TSLS. Regarding the results on  $\beta$, the point estimates are not larger in absolute value than in the full sample case (except for TSLS) with wider confidence intervals than those using the full sample.

 \section{Concluding remarks}
 \label{sec: conclusion}
 
This paper has considered fixed effects (FE) (or within group) estimation for linear
panel data models under possible model misspecification, where the conditional mean
$\E[y_{it}|\bx_{it}, c_{i}]$ may not be additive in $c_{i}$ and $\bx_{it}$, nor linear in $\bx_{it}$, when both the number
of individuals, $n$, and the number of time periods, $T$, are large. We make several
contributions to the literature. First, we have shown that the probability limit of the
FE estimator is identical to the coefficient vector on $\bx_{it}$ of the best partial linear approximation
to $\E[y_{it}|\bx_{it}, c_{i}]$ which we regard as the pseudo-true parameter. Moreover, we have established the asymptotic distributional properties of the
FE estimator around the pseudo-true parameter when $n$ and $T$ jointly go to infinity, and shown
that after subtracting the incidental parameters bias, the rate of convergence of the
FE estimate depends on the degree of model misspecification and is either $(nT)^{-1/2}$ 
or $n^{-1/2}$. Secondly, we have developed asymptotically valid inference on the pseudo-true
parameter vector. We have established the asymptotic properties of the clustered
covariance matrix estimator and the cross section bootstrap when both model
misspecification and the incidental parameters bias (in the coefficient estimate) are
present. Finally, we have conducted Monte Carlo simulations and evaluated the finite
sample performance of the FE and its bias corrected estimators, and several inference
methods and confirmed that the cross section bootstrap to pivotal statistics works particularly
well. These inference methods were applied to a study of the unemployment
dynamics in the U.S. state level.

\section*{Acknowledgments}
The authors would like to express their appreciation to Ivan Fernandez-Val, Kazuhiko Hayakawa, SeoJeong Lee, Ryo Okui, and participants in the 2013 Econometric Society North America Summer Meeting, and the NY Camp Econometrics VIII, for useful comments and discussions regarding this paper. We also would like to thank the editor, the associate editor, and three anonymous referees for their careful reading and comments to improve the manuscript.

\newpage

\appendix 

\section{Additional results on cross section bootstrap}

\label{Appendix}

In this appendix, we consider bootstrapping pivotal statistics. We keep the notation used in Section \ref{sec: inference}. 
We first consider the bootstrap version of the CCM estimator (here and in what follows, ``bootstrap'' means ``cross section bootstrap''). Let us write 
\begin{equation*}
\bm{z}_{i}^{*}= ((y_{i1}^{*},(\bx_{i1}^{*})')',\dots,(y_{iT}^{*},(\bx_{iT}^{*})')')'.
\end{equation*}
Then the bootstrap CCM estimator is defined by 
\begin{equation*}
\hat{\Sigma}_{nT}^{*} = \frac{1}{n^{2}}\sum_{i=1}^{n} \left \{ \frac{1}{T}\sum_{t=1}^{T} (\bx^{*}_{it}-\bar{\bx}^{*}_{i})\hat{\epsilon}^{*}_{it} \right \}^{\otimes 2},
\end{equation*}
where $\hat{\epsilon}^{*}_{it}  = y_{it}^{*} - \bar{y}_{i}^{*} - (\bx_{it}^{*}-\bar{\bx}_{i}^{*})'\tilde{\bbeta}^{*}$ and $\tilde{\bbeta}^{*}$ is the bootstrap version of a suitable estimator of $\bbeta_{0}$ (we formally assume that $\tilde{\bbeta}^{*}$ can be written as a statistic of $\bm{z}_{1},\dots,\bm{z}_{n}$ and $w_{n1},\dots,w_{nn}$), for example, 
$\tilde{\bbeta}^{*} = \hat{\bbeta}^{*}$ or $\hat{\bbeta}^{*}_{1/2}$. Note that 
\begin{equation*}
\hat{\Sigma}_{nT}^{*} = \frac{1}{n^{2}}\sum_{i=1}^{n} w_{ni} \left \{ \frac{1}{T}\sum_{t=1}^{T} (\bx_{it}-\bar{\bx}_{i})\hat{\epsilon}_{it}(\tilde{\bbeta}^{*}) \right \}^{\otimes 2},
\end{equation*}
where $\hat{\epsilon}_{it} (\bbeta)  = y_{it} - \bar{y}_{i} - (\bx_{it}-\bar{\bx}_{i})'\bbeta$. Then we have the following proposition.

\begin{proposition}
\label{prop4}
Suppose that conditions (A1)-(A3) are satisfied. Let  $\tilde{\bbeta}^{*}$ be such that $\| \tilde{\bbeta}^{*} - \bbeta_{0} \| = O_{\Pr_{W}}[ \max \{ d_{nT}^{-1/2}, T^{-1}  \}]$ in probability. Letting $T=T_{n} \to \infty$ as $n \to \infty$, we have:
\begin{equation*}
\hat{\Sigma}_{nT}^{*} = \Sigma_{nT} + O_{\Pr_{W}}[\{ \max \{ n^{-1/2}T^{-1} d_{nT}^{-1/2},n^{-1/2}d_{nT}^{-1} \}],
\end{equation*}
in probability. In particular, as long as $T=T_{n} \to \infty$, we have $\| \hat{\Sigma}^{*}_{nT} - \Sigma_{nT} \|_{\op} = o_{\Pr_{W}}(d_{nT}^{-1})$ in probability.
\end{proposition}

Proposition \ref{prop4} establishes the rate of convergence of $\hat{\Sigma}_{nT}^{*}$. This result is parallel to that in Proposition \ref{prop2}.
Note that the condition that $\| \tilde{\bbeta}^{*} - \bbeta_{0} \| = O_{\Pr_{W}}[ \max \{ d_{nT}^{-1/2}, T^{-1}  \}]$ in probability is satisfied with $\tilde{\bbeta}^{*} = \hat{\bbeta}^{*}$ or $\tilde{\bbeta}^{*} = \hat{\bbeta}^{*}_{1/2}$. 
Assume now that $\Sigma$ is nonsingular in either case of $\E[ \tilde{\bx}_{it} \epsilon_{it} \mid c_{i}] = \bm{0}$ a.s. or not. Consider, for the sake of simplicity, testing the null hypothesis $H_{0}: \beta^{a} = \beta_{0}^{a}$, where $1 \leq a \leq p$ is fixed, and consider the  $t$-statistic and its bootstrap version based on either the FE or HPJ estimate: 
\begin{equation*}
\hat{t} = \frac{\bm{e}_{a}'(\tilde{\bbeta} - \bbeta_{0})}{\sqrt{\bm{e}_{a}'\hat{A}^{-1}\hat{\Sigma}_{nT}\hat{A}^{-1}\bm{e}_{a}}}, \ \hat{t}^{*}=\frac{\bm{e}_{a}'(\tilde{\bbeta}^{*} - \tilde{\bbeta})}{\sqrt{\bm{e}_{a}'(\hat{A}^{*})^{-1}\hat{\Sigma}^{*}_{nT}(\hat{A}^{*})^{-1}\bm{e}_{a}}},  \ \text{with} \ \tilde{\bbeta} = \hat{\bbeta} \ \text{or} \ \hat{\bbeta}_{1/2},
\end{equation*} 
where $\bm{e}_{a}$ is the $p \times 1$ vector such that $e_{a}^{a} = 1$ and $e_{a}^{b}=0$ for $b \neq a$. Here because $\hat{A}^{*} = A + o_{\Pr_{W}}(1)$ (see Lemma \ref{lem3}) and 
$\hat{\Sigma}_{nT}^{*} = \Sigma_{nT} + o_{\Pr_{W}}(d_{nT}^{-1})$ in probability, we can deduce that under conditions (A1)-(A3), 
\begin{equation*}
\sup_{x \in \mathbb{R}} |\Pr_{W}( \hat{t}^{*} \leq x )  - \Pr(N(0,1) \leq x) | \stackrel{\Pr}{\to} 0.
\end{equation*}
This holds as long as $T=T_{n} \to \infty$ and does not require any specific growth restriction on $T$. Moreover, when $\tilde{\bbeta}  = \hat{\bbeta}_{1/2}$, under conditions (A1)-(A3), we have
\begin{equation}
\sup_{x \in \mathbb{R}} |\Pr_{W}( \hat{t}^{*} \leq x )  - \Pr(\hat{t} \leq x) | \stackrel{\Pr}{\to} 0, \label{bootB}
\end{equation}
provided that $\sqrt{d_{nT}}/T^{2} \to 0$.

\section{Additional simulation results}

\label{additional simulation}

\subsection{Effect of $T$ on the performance on estimators}

In order to shed more light on the performance of the proposed methods, Figure \ref{fig1} presets the bias and RMSE of the FE and HPJ estimators from a fixed cross-section when varying the time series $T$.  In these simulations, we only considered the AR(2) model given in the second example with $\phi_{1}=\phi_{2}=0.4$. We fix the cross-section dimension at $n=50$. The left panel shows the bias of the estimators as a function of $T$. The results show that the HPJ estimator has a small bias for small $T$, but the bias disappears even for relatively small $T$. On the contrary, the bias in the FE is large, and it remains relatively substantial even for large time series. The right panel displays the RMSE for the estimators. It shows a good performance of the HPJ estimator for moderate time dimensions.
 \begin{figure}
 \begin{center}
 \centerline{\includegraphics[width=.65\textwidth]{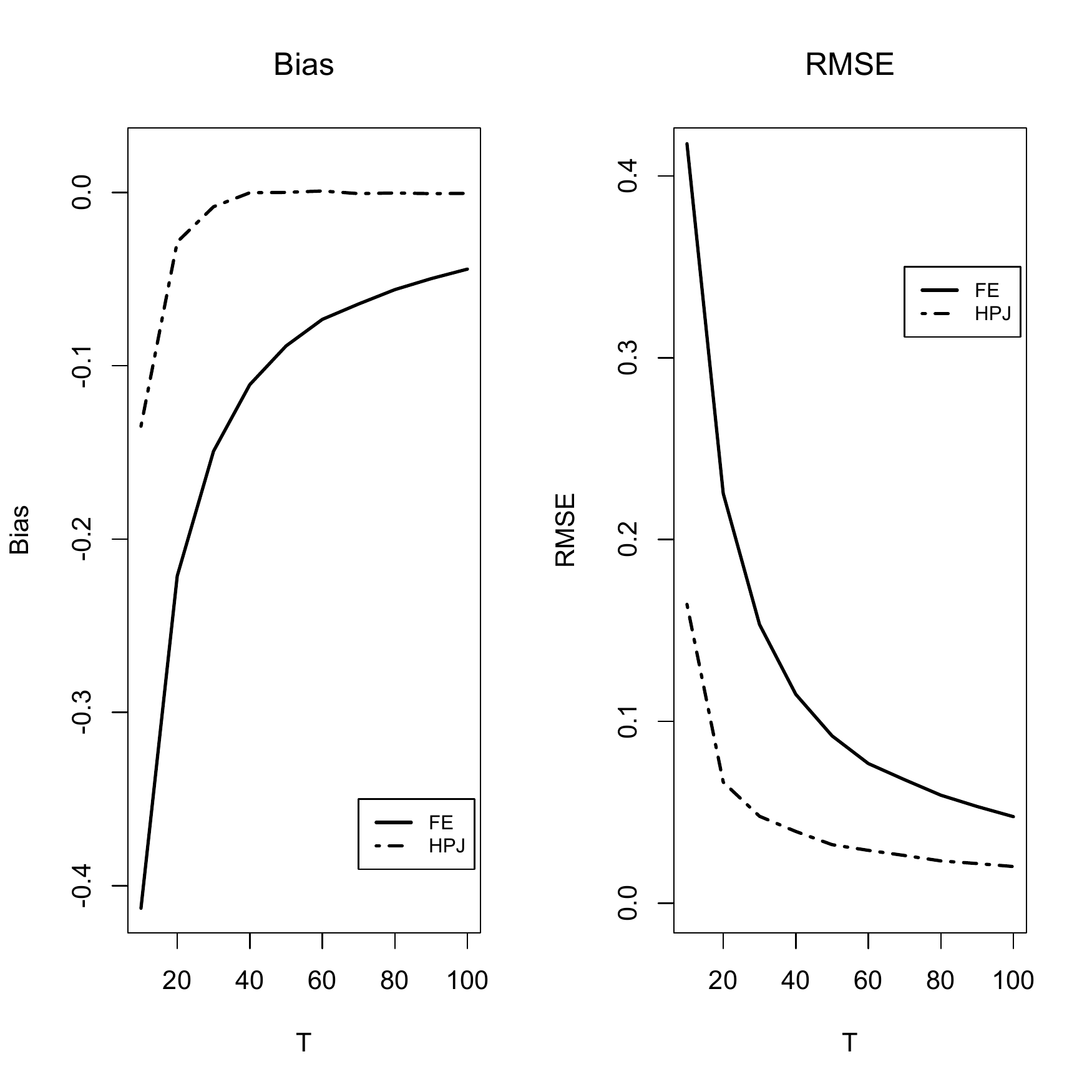}}
 \caption{\emph{Small sample performance of the estimators when $T$ increases for $n=50$.} \label{fig1}}
 \end{center}
 \end{figure}

\subsection{Additional simulation design: Panel EXPAR model}

Here we consider the case where the true DGP is
 \begin{equation*}
 y_{it} = c_{i} + \rho_1( y_{i,t-1}-c_{i}) + \rho_2 \exp( -(y_{i,t-1}- c_{i})^2) + u_{it},
 \end{equation*}
 where  $u_{it} \sim \ \text{i.i.d.} \ N(0,1)$, $c_{i} \sim \ \text{i.i.d.} \ U(-0.5,0.5)$, $\rho_{1}=0.8$ and $\rho_{2}=1$. 
 This model is a panel-data version of exponential AR (EXPAR) models \citep[see][]{O85}. Note that by \citet[][Example 3.2]{AH96}, for each $i \geq 1$, the process $\{ (y_{it} - c_{i}) \}_{t \in \mathbb{Z}}$ (which is independent of $c_{i}$) is geometrically ergodic, so that condition (A1) is satisfied. As in the previous case, we incorrectly fit a panel AR(1) model and estimate the slope parameter. The value of the pseudo-true parameter is $\beta_{0}=0.63$. The simulation results for this case are presented in Table \ref{table.mc.3}. 

 {\bf Table \ref{table.mc.3}}: The FE estimate is severely biased  and the FE-CCM option performs poorly because of the presence of large bias. On the other hand, the HPJ estimate is able to largely remove the bias at the cost of 
 slight variance inflation relative to the FE estimate in the finite sample. 
 In this case, the HK and GMM estimates are able to reduce the bias to some extent, although GMM still presenting larger bias. The empirical coverage of the HK option is reasonably well (which is partly due to the fact that the panel EXPAR model is ``close'' to the panel AR(1) model). GMM is under coverage. In addition, the empirical coverage of the HK and GMM options worsen as $n$ becomes large. 
 Among the other options,  HPJ-HPJPB works particularly  well.

\section{Proofs for Section \ref{sec: asymptotics}}
\label{Appendix A}

We shall recall here the notational convention. For a generic vector $\bm{z}$, $z^{a}$ denotes the $a$-th element of $\bm{z}$, and for a generic matrix $A$, $A^{ab}$ denotes the $(a,b)$-th element of $A$. Moreover, for a sequence $\{ z_{it} \}$ indexed by $(i,t)$, we write $\bar{z}_{i} = T^{-1} \sum_{t=1}^{T} z_{it}$.

\subsection{Inequalities for $\alpha$-mixing processes}

In this section, we introduce some inequalities for $\alpha$-mixing processes, which will be used in the proofs below. 
Let $\{ \xi_{t} \}$ denote a stationary process taking values in some Polish space $S$, and let $\alpha (k)$ denote its $\alpha$-mixing coefficients. 

\begin{theorem}[\cite{D68}]
\label{thmA1}
Let $\mathcal{A}_{i}^{j}$ denote the $\sigma$-field generated by $\xi_{i},\dots,\xi_{j}$ ($i < j$). Pick any integer $k \geq 1$. Let $\xi$ and $\eta$ be real-valued random variables measurable with respect to 
$\mathcal{A}_{-\infty}^{0}$ and $\mathcal{A}_{k}^{\infty}$, respectively. If $\E[ | \xi |^{q} ] < \infty$ and $\E[ | \eta |^{r} ] < \infty$ for some $q > 1$ and $r > 1$ such that $q^{-1} + r^{-1} < 1$, then we have
\begin{equation*}
| \E[ \xi \eta ] - \E[ \xi ] \E[ \eta ] | \leq 12  (\E[ | \xi |^{q} ])^{q^{-1}}(\E[ | \eta |^{r} ])^{r^{-1}} \alpha(k)^{1-q^{-1}-r^{-1}}.
\end{equation*}
\end{theorem}

To illustrate an application of Davidov's inequality, suppose that $S=\mathbb{R}$, and assume that $\E[ |\xi_{1}|^{q}] < \infty$ and $\sum_{k=1}^{\infty} \alpha(k)^{1-2/q} < \infty$ for some $q > 2$. 
Then by Davidov's inequality, we have $\Var (\sum_{t=1}^{T} \xi_{t}) \leq CT$ with
$C=12 \E[| \xi_{1} |^{q}]^{2/q} \sum_{k=0}^{\infty} \alpha(k)^{1-2/q}$. For bounding  higher order moments, we make use of Yokoyama's (1980) Theorem 3. 

\begin{theorem}[\cite{Y80}, Theorem 3]
\label{thmA2}
Suppose that $S=\mathbb{R}$. Assume that $\E[ \xi_{1} ] = 0$ and for some constants $\delta > 0$ and  $r > 2$, $\E[ | \xi_{1} |^{r+\delta}] < \infty$. 
If $\sum_{k=0}^{\infty} (k+1)^{r/2-1} \alpha (k)^{\delta/(r+\delta)} < \infty$, then there exists a constant $C$ independent of $T$ such that 
\begin{equation*}
\E\left [ \left | \sum_{t=1}^{T}  \xi_{t} \right |^{r} \right ] \leq C T^{r/2}.
\end{equation*}
\end{theorem}

In Theorem \ref{thmA2}, the constant $C$ depends only on $r, \delta$ and bounds on $\E[| \xi_{1} |^{r+\delta}]$ and $\sum_{k=0}^{\infty} (k+1)^{r/2-1} \alpha (k)^{\delta/(r+\delta)}$.

\subsection{Proof of Lemma \ref{lem1}}

The lemma is deduced directly from Theorem \ref{thmA1} and condition (A1).  \qed

\subsection{Proof of Proposition \ref{prop1}}

We provide a proof of Proposition \ref{prop1}. Define
\begin{align*}
\hat{S}_{1} = \frac{1}{nT} \sum_{i=1}^{n} \sum_{t=1}^{T} \tilde{\bx}_{it}\epsilon_{it}, \ \hat{S}_{2} = \frac{1}{n} \sum_{i=1}^{n} \bar{\tilde{\bx}}_{i}\bar{\epsilon}_{i}.
\end{align*}
By definition, we have $\hat{\bbeta} - \bbeta_{0} = \hat{A}^{-1} (\hat{S}_{1} - \hat{S}_{2})$. We  prepare a technical lemma. 

\begin{lemma}
\label{lem2}
Suppose that conditions (A1)-(A3) are satisfied. 
As $n \to \infty$ (and automatically $T=T_{n} \to \infty$), we have: (i) $\hat{A}$ is nonsingular with probability approaching one, and \\
 $\hat{A}^{-1} = A^{-1} \sum_{m=0}^{\infty} T^{-m} (D_{T}A^{-1})^{m} + O_{\Pr} (n^{-1/2})$; (ii) $\E[ \hat{S}_{2} ] = T^{-1} B_{T} = O(T^{-1})$; (iii) $\hat{S}_{2} - \E[ \hat{S}_{2}] = O_{\Pr} ( n^{-1/2} T^{-1} )$;
(iv) $\sqrt{nT} (\hat{S}_{1} - \E[\hat{S}_{1} \mid \{ c_{i} \}_{i=1}^{n}]) \stackrel{d}{\to} N(\bm{0},V_{1})$, where $V_{1}=\sum_{k=-\infty}^{\infty} \E[ (\tilde{\bx}_{1,1}\epsilon_{1,1} - \E[\tilde{\bx}_{1,1}\epsilon_{1,1} \mid c_{1}])(\tilde{\bx}_{1,1+k}\epsilon_{1,1+k}-\E[\tilde{\bx}_{1,1+k}\epsilon_{1,1+k} \mid c_{1}])']$ (the right side is absolutely convergent in $\| \cdot \|_{\op}$).
\end{lemma}
\begin{proof}[Proof of Lemma \ref{lem2}]
{\bf Part (i)}: Recall that $A=\E[\tilde{\bx}_{1,1} \tilde{\bx}_{1,1}']$. Observe that $T^{-1}D_{T} = \E[ \bar{\tilde{\bx}}_{1}\bar{\tilde{\bx}}_{1}']$ and
\begin{align*}
\hat{A} - A + T^{-1} D_{T} &= \frac{1}{nT} \sum_{i=1}^{n} \sum_{t=1}^{T} (\tilde{\bx}_{it} \tilde{\bx}_{it}' - \E[\tilde{\bx}_{i1} \tilde{\bx}_{i1}' \mid c_{i}]) \\
&\quad +\frac{1}{n} \sum_{i=1}^{n} (\E[\tilde{\bx}_{i1} \tilde{\bx}_{i1}' \mid c_{i}] - \E[\tilde{\bx}_{1,1} \tilde{\bx}_{1,1}']) - \frac{1}{n} \sum_{i=1}^{n} (\bar{\tilde{\bx}}_{i}\bar{\tilde{\bx}}_{i}' - \E[ \bar{\tilde{\bx}}_{1}\bar{\tilde{\bx}}_{1}']) \\
&=: \hat{D}_{1} + \hat{D}_{2} - \hat{D}_{3}.
\end{align*}
Fix any $1 \leq a,b \leq p$. We wish to show that $\hat{D}_{1}^{ab} = O_{\Pr}\{  (nT)^{-1/2} \}$, $\hat{D}_{2}^{ab} = O_{\Pr}(n^{-1/2})$ and $\hat{D}_{3}^{ab} = O_{\Pr} (n^{-1/2}T^{-1})$. We make use of Theorems \ref{thmA1} and \ref{thmA2}. Put $\xi_{it} = \tilde{x}_{it}^{a} \tilde{x}_{it}^{b} - \E[ \tilde{x}_{i1}^{a} \tilde{x}_{i1}^{b} \mid c_{i} ]$. Then $\hat{D}_{1}^{ab} = (nT)^{-1} \sum_{i=1}^{n}\sum_{t=1}^{T} \xi_{it}$ and 
\[
\Var (\hat{D}_{1}^{ab}) = n^{-2} T^{-2} \sum_{i=1}^{n} \E[  (\sum_{t=1}^{T} \xi_{it} )^{2} ] = n^{-1}T^{-2}  \E[  (\sum_{t=1}^{T} \xi_{1t} )^{2} ] = n^{-1} T^{-2} \E[\E[(\sum_{t=1}^{T} \xi_{1t} )^{2} \mid c_{1}]].
\]
Using Theorem \ref{thmA1} to bound $\E[(\sum_{t=1}^{T} \xi_{1t} )^{2} \mid c_{1}]$, we have $\Var (\hat{D}_{1}^{ab}) = O\{ (nT)^{-1} \}$, which implies that $\hat{D}_{1}^{ab} = O_{\Pr}\{  (nT)^{-1/2} \}$.
The fact that $\hat{D}_{2}^{ab} = O_{\Pr}(n^{-1/2})$ is deduced from a direct evaluation of the variance. 
We next show that $\hat{D}_{3}^{ab} = O_{\Pr}(n^{-1/2} T^{-1})$. By the cross section independence and the Cauchy-Schwarz inequality, we have
\begin{equation*}
\E[| \hat{D}_{3}^{ab} |^{2}] =  \frac{1}{n^{2}} \sum_{i=1}^{n} \Var (\bar{\tilde{x}}_{i}^{a} \bar{\tilde{x}}_{i}^{b}) = n^{-1} \Var (\bar{\tilde{x}}_{1}^{a} \bar{\tilde{x}}_{1}^{b}) \leq n^{-1}  (\E[ \E[ (\bar{\tilde{x}}_{1}^{a})^{4} \mid c_{1}]])^{1/2} (\E[ \E[ (\bar{\tilde{x}}_{1}^{b})^{4} \mid c_{1}] ])^{1/2}.
\end{equation*}
Using Theorem \ref{thmA2} to bound $\E[ (\bar{\tilde{x}}_{1}^{a})^{4} \mid c_{1}]$ and $\E[ (\bar{\tilde{x}}_{1}^{b})^{4} \mid c_{1}]$, we have $\E[| \hat{D}_{2}^{ab} |^{2}] = O(n^{-1}T^{-2})$, which implies that $\hat{D}_{2}^{ab} = O_{\Pr}(n^{-1/2}T^{-1})$.
Therefore, we have $\hat{A} = A - T^{-1} D_{T} + O_{\Pr}(n^{-1/2})$. 
By condition (A3), there exists a constant $\rho > 0$ such that $\lambda_{\min}(A) \geq \rho$.
By Lemma \ref{lem1},  $\lambda_{\min}(\hat{A}) \geq \lambda_{\min}(A) - T^{-1} \| D_{T} \|_{\op} - O_{\Pr}(n^{-1/2}) = \rho - o_{\Pr}(1)$, which implies that $\hat{A}$ is nonsingular with probability approaching one.

We wish to obtain the expansion of $\hat{A}$. By condition (A3) and Lemma \ref{lem1}, $A - T^{-1} D_{T}$ is nonsingular for large $n$, and 
$\| (A - T^{-1} D_{T})^{-1} \|_{\op} \leq (1+o(1)) \rho^{-1}$. 
Put $\hat{R} = (\hat{A} - A + T^{-1} D_{T})(A - T^{-1} D_{T})^{-1}$ so that $\hat{A} = (I + \hat{R})(A - T^{-1} D_{T})$. 
Since $\hat{R} = O_{\Pr}(n^{-1/2})$, applying the Taylor expansion to the term $(I+\hat{R})^{-1}$, we have 
\begin{align*}
\hat{A}^{-1} &= (A - T^{-1} D_{T})^{-1} (I + O_{\Pr}(\hat{R}) ) = (A - T^{-1} D_{T})^{-1} (I + O_{\Pr} (n^{-1/2}) ) \\
&= A^{-1} (I  - T^{-1} D_{T}A^{-1})^{-1} (I + O_{\Pr} (n^{-1/2}) ).
\end{align*}
Since $\| T^{-1} D_{T}A^{-1} \|_{\op} \leq T^{-1} \| D_{T} \|_{\op} \| A^{-1} \|_{\op} = O(T^{-1})$, applying the Taylor expansion to the term $(I  - T^{-1} D_{T}A^{-1})^{-1}$, we have 
\begin{equation*}
(I  - T^{-1} D_{T}A^{-1})^{-1} = \sum_{m=0}^{\infty} T^{-m} (D_{T}A^{-1})^{m}, \ \text{(Neumann series})
\end{equation*}
where the right side is absolutely convergent in $\| \cdot \|_{\op}$ for large $n$.

{\bf Part (ii)}: This follows from a direct calculation and Lemma \ref{lem1}. 

{\bf Part (iii)}: Fix any $1 \leq a \leq p$. It suffices to show that $\Var (\hat{S}^{a}_{2}) = O( n^{-1} T^{-2} )$. By  the cross section independence and the  Cauchy-Schwarz inequality, we have 
\begin{equation*}
\Var (\hat{S}^{a}_{2}) \leq \frac{1}{n^{2}} \sum_{i=1}^{n} \E[ (\bar{\tilde{x}}_{i}^{a} \bar{\epsilon}_{i})^{2} ] = n^{-1} \E[ (\bar{\tilde{x}}_{1}^{a} \bar{\epsilon}_{1})^{2} ] \leq n^{-1}  (\E[ \E[ (\bar{\tilde{x}}_{1}^{a})^{4} \mid c_{1}]])^{1/2} (\E[ \E[ (\bar{\epsilon}_{1})^{4} \mid c_{1}] ])^{1/2}.
\end{equation*}
Using Theorem \ref{thmA2} to bound $\E[ (\bar{\tilde{x}}_{1}^{a})^{4} \mid c_{1}]$ and $\E[ (\bar{\epsilon}_{1})^{4} \mid c_{1}]$, we have $\Var (\hat{S}^{a}_{2}) = O(n^{-1}T^{-2})$. 

{\bf Part (iv)}: The fact that $V_{1}$ is absolutely convergent is deduced from Theorem \ref{thmA1}. We wish to show the asymptotic normality. 
By the Cram\'{e}r-Wald device, it suffices to show that for any fixed $\bm{r} \in \mathbb{R}^{p}$, $\sqrt{nT} \bm{r}'(\hat{S}_{1} - \E[ \hat{S}_{1} \mid \{ c_{i} \}_{i=1}^{n}]) \stackrel{d}{\to} N(0,\bm{r}'V_{1}\bm{r})$. 
Define $u_{ni} = T^{-1/2} \sum_{t=1}^{T} \bm{r}'(\tilde{\bx}_{it} \epsilon_{it} - \E[ \tilde{\bx}_{it} \epsilon_{it} \mid c_{i}])$. Then,
\begin{equation*}
\sqrt{nT} \bm{r}'(\hat{S}_{1} - \E[ \hat{S}_{1} \mid \{ c_{i} \}_{i=1}^{n}]) = \frac{1}{\sqrt{n}} \sum_{i=1}^{n} u_{ni}.
\end{equation*}
Since $u_{n1},\dots,u_{nn}$ are i.i.d., we can apply the Lyapunov central limit theorem to the right sum. Since $\Var  (u_{n1})  \to \bm{r}'V_{1}\bm{r}$, it suffices to show that $\E[ | u_{n1} |^{3} ] = o(n^{1/2})$. Here, using Theorem \ref{thmA2} to bound $\E[ | u_{n1} |^{3} \mid c_{1} ]$, 
we have $\E[ | u_{n1} |^{3} ] = \E[ \E[ | u_{n1} |^{3} \mid c_{1} ] ] = O(1) = o(n^{1/2})$. Therefore, we obtain the desired result. 
\end{proof}

We are now in position to prove Proposition \ref{prop1}. 

\begin{proof}[Proof of Proposition \ref{prop1}]
The expansion (\ref{Bahadur}) follows from Lemma \ref{lem2} (note that we use the fact that $\hat{S}_{1}=O_{\Pr}(d_{nT}^{-1/2})$). 
Suppose that $\E[ \tilde{\bx}_{it} \epsilon_{it} \mid c_{i} ] = \bm{0}$ a.s. Then 
$\hat{S}_{1} = \hat{S}_{1} - \E[ \hat{S}_{1} \mid \{ c_{i} \}_{i=1}^{n}]$, so that by Lemma \ref{lem2} (iv), 
$\sqrt{nT} \hat{S}_{1} \stackrel{d}{\to} N(\bm{0},V_{1})$. Suppose now that $\E[ \tilde{\bx}_{it} \epsilon_{it} \mid c_{i} ] \neq \bm{0}$ with  positive probability.  
Then we have 
\begin{align*}
\sqrt{n} \hat{S}_{1} &= \sqrt{n} \E[ \hat{S}_{1} \mid \{ c_{i} \}_{i=1}^{n} ] + \sqrt{n} (\hat{S}_{1}-\E[\hat{S} \mid \{ c_{i} \}_{i=1}^{n}] ) \\
&=\sqrt{n} \E[\hat{S}_{1} \mid \{ c_{i} \}_{i=1}^{n}] + O_{\Pr}(T^{-1/2}) \\
&\stackrel{d}{\to} N(\bm{0},V_{2}),
\end{align*}
where $V_{2} = \E[ \E[ \tilde{\bx}_{1,1} \epsilon_{1,1} \mid c_{1}]^{\otimes 2}]$. The asymptotic normality follows from the fact that $\Sigma = V_{1}$ if $\E[ \tilde{\bx}_{it} \epsilon_{it} \mid c_{i} ] = \bm{0}$ a.s. and $\Sigma = V_{2}$ otherwise. 
\end{proof}

\subsection{Proof of Corollary \ref{cor1}}

By Lemma \ref{lem1}, we have $B_{T} = \sum_{|k| \leq T-1}  \E[\tilde{\bx}_{1,1} \epsilon_{1,1+k}]  + O(T^{-1})$.
Observe that
\begin{align*}
\sum_{|k| \geq T} \|  \E[\tilde{\bx}_{1,1} \epsilon_{1,1+k}] \| &=  \sum_{|k| \geq T} |k|^{-1} \cdot |k|\|  \E[\tilde{\bx}_{1,1} \epsilon_{1,1+k}] \| \\
&\leq \frac{1}{T}  \sum_{|k| \geq T} |k| \| \E[\tilde{\bx}_{1,1} \epsilon_{1,1+k}] \| \\
&\leq \frac{1}{T}  \sum_{k=-\infty}^{\infty} |k| \| \E[\tilde{\bx}_{1,1} \epsilon_{1,1+k}] \| = O(T^{-1}).
\end{align*}
Therefore, we have $B_{T} = B + O(T^{-1})$, which implies the desired result. \qed

\section{Proofs for Section \ref{sec: inference} and Appendix \ref{Appendix}}
\label{Appendix B}

\subsection{Proof of Proposition \ref{prop2}}

Since $\hat{\epsilon}_{it}=-(\tilde{\bx}_{it}-\bar{\tilde{\bx}}_{i})'(\tilde{\bbeta}-\bbeta_{0}) + \epsilon_{it} - \bar{\epsilon}_{i}$, 
we have 
\begin{equation*}
\hat{\Sigma}_{nT} = \tilde{\Sigma}_{nT} + \frac{1}{n^{2}} \sum_{i=1}^{n} (-R_{1i} - R_{2i} - R_{1i}' - R_{2i}'+ R_{3i} + R_{3i}' + R_{4i} + R_{5i}),
\end{equation*}
where
\begin{align*} 
\tilde{\Sigma}_{nT} &= n^{-2} {\textstyle \sum}_{i=1}^{n}   ( T^{-1}  {\textstyle \sum}_{t=1}^{T} \tilde{\bx}_{it} \epsilon_{it} )^{\otimes 2}, \
R_{1i} = (T^{-1}  {\textstyle \sum}_{t=1}^{T} \tilde{\bx}_{it} \epsilon_{it}  ) (\bar{\tilde{\bx}}_{i} \bar{\epsilon}_{i})', \\
R_{2i} &=  ( T^{-1}  {\textstyle \sum}_{t=1}^{T} \tilde{\bx}_{it} \epsilon_{it} )\{ T^{-1} {\textstyle \sum}_{t=1}^{T} (\tilde{\bx}_{it} - \bar{\tilde{\bx}}_{i})^{\otimes 2} (\tilde{\bbeta}-\bbeta_{0}) \}', \\
R_{3i} &= (\bar{\tilde{\bx}}_{i} \bar{\epsilon}_{i}) \{ T^{-1} {\textstyle \sum}_{t=1}^{T} (\tilde{\bx}_{it} - \bar{\tilde{\bx}}_{i})^{\otimes 2} (\tilde{\bbeta}-\bbeta_{0}) \}',  \
R_{4i} = (\bar{\tilde{\bx}}_{i} \bar{\epsilon}_{i})^{\otimes 2}, \\
R_{5i} &= \{ T^{-1}  {\textstyle \sum}_{t=1}^{T} (\tilde{\bx}_{it} - \bar{\tilde{\bx}}_{i})^{\otimes 2} (\tilde{\bbeta}-\bbeta_{0}) \}^{\otimes 2}. 
\end{align*}
By definition, we have 
\begin{align*}
\| R_{1i} \|_{\op} &\leq \| T^{-1} {\textstyle \sum}_{t=1}^{T} \tilde{\bx}_{it} \epsilon_{it} \| \| \bar{\tilde{\bx}}_{i} \bar{\epsilon}_{i} \|, \\
\| R_{2i} \|_{\op} &\leq \| T^{-1} {\textstyle \sum}_{t=1}^{T} \tilde{\bx}_{it} \epsilon_{it} \| \| T^{-1} {\textstyle \sum}_{t=1}^{T} (\tilde{\bx}_{it} - \bar{\tilde{\bx}}_{i})^{\otimes 2} \|_{\op} \| \tilde{\bbeta} - \bbeta_{0} \| \\ 
&\leq  \| T^{-1} {\textstyle \sum}_{t=1}^{T} \tilde{\bx}_{it} \epsilon_{it} \| \| T^{-1} {\textstyle \sum}_{t=1}^{T} \tilde{\bx}_{it} \tilde{\bx}_{it}' \|_{\op} \| \tilde{\bbeta} - \bbeta_{0} \|, \\
\| R_{3i} \|_{\op} &\leq \| \bar{\tilde{\bx}}_{i} \bar{\epsilon}_{i} \| \| T^{-1} {\textstyle \sum}_{t=1}^{T} \tilde{\bx}_{it} \tilde{\bx}_{it}' \|_{\op} \| \tilde{\bbeta} - \bbeta_{0} \|, \ \| R_{4i} \|_{\op} \leq \| \bar{\tilde{\bx}}_{i} \bar{\epsilon}_{i} \|^{2}, \\
\| R_{5i} \|_{\op} &\leq \| T^{-1} {\textstyle \sum}_{t=1}^{T} (\tilde{\bx}_{it} - \bar{\tilde{\bx}}_{i})^{\otimes 2} \|^{2}_{\op} \| \tilde{\bbeta} - \bbeta_{0} \|^{2} \\
&\leq  \| T^{-1} {\textstyle \sum}_{t=1}^{T} \tilde{\bx}_{it} \tilde{\bx}_{it}' \|^{2}_{\op} \| \tilde{\bbeta} - \bbeta_{0} \|^{2}.
\end{align*}
By Theorems \ref{thmA1} and \ref{thmA2}, we can show that $\E[ \| T^{-1} \sum_{t=1}^{T} \tilde{\bx}_{it} \epsilon_{it} \|^{2}] \leq C nd_{nT}^{-1}$, 
$\E[  \| \bar{\tilde{\bx}}_{i} \bar{\epsilon}_{i} \|^{2} ] \leq (\E[ \| \bar{\tilde{\bx}}_{i} \|^{4}])^{1/2} (\E[ \bar{\epsilon}_{i}^{4}])^{1/2} \leq C T^{-2}$ and $\E[ \| T^{-1} \sum_{t=1}^{T} \tilde{\bx}_{it} \tilde{\bx}_{it}' \|_{\op}^{2}] \leq C$ for some constant $C > 0$. 
Therefore, we have 
\[
\| \hat{\Sigma}_{nT} - \tilde{\Sigma}_{nT} \|_{\op}  
=O_{\Pr}[ \max \{ n^{-1/2}T^{-1} d_{nT}^{-1/2}, n^{-1/2} d_{nT}^{-1} \}].
\]
In what follows, we wish to show that $\tilde{\Sigma}_{nT} = \Sigma_{nT} + O_{\Pr}(n^{-1/2} d_{nT}^{-1})$. Fix any $1 \leq a,b \leq p$. By definition, we have 
\begin{align*}
\tilde{\Sigma}_{nT}^{ab} &= \frac{1}{n^{2}} \sum_{i=1}^{n} \left ( \frac{1}{T^{2}}\sum_{s=1}^{T} \sum_{t=1}^{T} \tilde{x}_{is}^{a} \epsilon_{is} \tilde{x}_{it}^{b} \epsilon_{it} \right ) \\
&=\frac{1}{n^{2}} \sum_{i=1}^{n} \left \{ \frac{1}{T^{2}}\sum_{s=1}^{T} \sum_{t=1}^{T} (\tilde{x}_{is}^{a} \epsilon_{is} - \E[\tilde{x}_{i1}^{a} \epsilon_{i1} \mid c_{i}]) ( \tilde{x}_{it}^{b} \epsilon_{it} - \E[ \tilde{x}_{i1}^{b} \epsilon_{i1} \mid c_{i}]) \right \} \\
&\quad + \frac{1}{n^{2}} \sum_{i=1}^{n} \E[\tilde{x}_{i1}^{a} \epsilon_{i1} \mid c_{i}] \left \{ \frac{1}{T} \sum_{t=1}^{T} (\tilde{x}_{it}^{b} \epsilon_{it} -\E[\tilde{x}_{i1}^{b} \epsilon_{i1} \mid c_{i}]) \right \} \\
&\quad + \frac{1}{n^{2}} \sum_{i=1}^{n} \E[\tilde{x}_{i1}^{b} \epsilon_{i1} \mid c_{i}] \left \{  \frac{1}{T} \sum_{t=1}^{T} (\tilde{x}_{it}^{a} \epsilon_{it}-\E[\tilde{x}_{i1}^{a} \epsilon_{i1} \mid c_{i}]) \right \} \\
&\quad -  \frac{1}{n^{2}} \sum_{i=1}^{n} \E[\tilde{x}_{i1}^{a}\epsilon_{i1} \mid c_{i}] \E[\tilde{x}_{i1}^{b}\epsilon_{i1} \mid c_{i}] \\
&=:(I) + (II) + (III) - (IV). 
\end{align*}
It suffices to show that the variance of each term in (I)-(IV) is $O(n^{-1} d_{nT}^{-2})$. Letting $\xi_{it} = \tilde{x}_{it}^{a} \epsilon_{it} - \E[\tilde{x}_{i1}^{a} \epsilon_{i1} \mid c_{i}]$ and $\eta_{it} = \tilde{x}_{it}^{b} \epsilon_{it} - \E[\tilde{x}_{i1}^{b} \epsilon_{i1} \mid c_{i}]$, we have 
\begin{align*}
\Var (I) &\leq \frac{1}{n^{4}T^{4} } \sum_{i=1}^{n}\sum_{s=1}^{T} \sum_{t=1}^{T} \sum_{u=1}^{T} \sum_{v=1}^{T} \E[\xi_{is}\eta_{it}\xi_{iu}\eta_{iv}] \\
&= \frac{1}{n^{3}T^{2}} \left ( \frac{1}{T^{2}} \sum_{s=1}^{T} \sum_{t=1}^{T} \sum_{u=1}^{T} \sum_{v=1}^{T} \E[\xi_{1s}\eta_{1t}\xi_{1u}\eta_{1v}] \right ).
\end{align*}
Using the same argument as in the proof of \citet[][Theorem 1]{Y80}, we can show  that the parenthesis on the right side is $O(1)$.\footnote{It suffices to show that $\sum_{s=1}^{T} \sum_{t=1}^{T} \sum_{u=1}^{T} \sum_{v=1}^{T} |\E[\xi_{1s} \eta_{1t} \xi_{1u} \eta_{1v} \mid c_{1} ]| \leq C T^{2}$ for some constant $C$. When $a=b$, so that $\xi_{it} = \eta_{it}$, the assertion directly follows from the proof of \citet[][Theorem 1]{Y80} with $r=4$. The proof for the $a \neq b$ case is almost the same as that for the $a=b$ case. } Therefore, we have $\Var (I) = O(n^{-3}T^{-2})$. If $\E[ \tilde{\bx}_{it} \epsilon_{it} \mid c_{i} ] = \bm{0}$ a.s., (II)-(IV) are zeros a.s., so that 
$\Var (\tilde{\Sigma}^{ab}) = O(n^{-3}T^{-2}) = O(n^{-1} d_{nT}^{-2})$. Suppose now that $\E[ \tilde{\bx}_{it} \epsilon_{it} \mid c_{i} ] \neq \bm{0}$ with  positive probability. 
Clearly, $\Var (IV) = O(n^{-3})$. Likewise, we have 
\begin{equation*}
\Var (II) \leq \frac{1}{n^{3}T}  \E \left [ \E[\tilde{x}_{1,1}^{a} \epsilon_{1,1} \mid c_{1}]^{2}  \E \left [  \left ( \frac{1}{\sqrt{T}} \sum_{t=1}^{T} \eta_{1t} \right )^{2} \mid c_{1} \right ]  \right ].
\end{equation*}
By Theorem \ref{thmA1}, we have $\E[(T^{-1/2} \sum_{t=1}^{T} \eta_{1t})^{2} \mid c_{1}] \leq C$ a.s. for some constant $C > 0$, so that $\Var (II) = O(n^{-3}T^{-1})$. Similarly, we have $\Var (III) = O(n^{-3}T^{-1})$. Therefore, we have $\Var (\tilde{\Sigma}_{nT}^{ab}) = O(n^{-3}) = O(n^{-1} d_{nT}^{-2})$. 

\subsection{Proof of Proposition \ref{prop3}}
We first note a lemma on a relation between stochastic orders. 

\begin{lemma}[\cite{CH10}, Lemma 3]
\label{lemB1}
Given a vector valued statistic $\Delta_{n}$ depending on both $c_{1},\dots,c_{n},\bm{z}_{1},\dots,\bm{z}_{n}$ and $w_{n1},\dots,w_{nn}$,  and a deterministic sequence $a_{n} > 0$, we have:
\begin{equation*}
\Delta_{n} = O_{\Pr_{W}}(a_{n}) \ \text{in probability} \  \Leftrightarrow \ \Delta_{n} = O_{\Pr}(a_{n}) \ (\text{unconditionally}).
\end{equation*}
\end{lemma}

By Lemma \ref{lemB1}, it suffices to evaluate the remainder term in the expansion (\ref{BahadurB}) unconditionally since it is translated to the evaluation under the conditional probability without changing rates.

We start to prove Proposition \ref{prop3}. Define 

\begin{align*}
% &\hat{A}^{*} = \frac{1}{nT} \sum_{i=1}^{n} w_{ni} \sum_{t=1}^{T} (\bx_{it} - \bar{\bx}_{i}) (\bx_{it} - \bar{\bx}_{i})', \\
\hat{S}_{1}^{*} = \frac{1}{nT} \sum_{i=1}^{n} w_{ni} \sum_{t=1}^{T} \tilde{\bx}_{it}\epsilon_{it}, \ \hat{S}_{2}^{*} = \frac{1}{n} \sum_{i=1}^{n} w_{ni}\bar{\tilde{\bx}}_{i}\bar{\epsilon}_{i}.
\end{align*}
By definition, we have $\hat{\bbeta}^{*} - \bbeta_{0} = (\hat{A}^{*})^{-1} (\hat{S}^{*}_{1} - \hat{S}^{*}_{2})$. Also define $\hat{S}_{1}$ and $\hat{S}_{2}$ as in the proof of Proposition \ref{prop1}. 
We prepare some technical lemmas. 

\begin{lemma}
\label{lem3}
Suppose that conditions (A1)-(A3) are satisfied. 
As $n \to \infty$ (and automatically $T=T_{n} \to \infty$), we have: (i) $\hat{A}^{*}$ is nonsingular with probability approaching one, and $(\hat{A}^{*})^{-1} = A^{-1} \sum_{m=0}^{\infty} T^{-m} (D_{T}A^{-1})^{m} + O_{\Pr} (n^{-1/2})$; 
(ii)  $\E[\hat{S}^{*}_{2}]  = T^{-1}B_{T} = O(T^{-1})$; (iii) $\hat{S}^{*}_{2} - \E[ \hat{S}^{*}_{2}] = O_{\Pr} ( n^{-1/2} T^{-1} )$.
\end{lemma}

In the proof of Lemma \ref{lem3}, we use some elementary properties of multinomial distributions.
Recall that $(w_{n1},\dots,w_{nn})'$ is multinomially distributed with parameters $n$ and (probabilities) $n^{-1},\dots,n^{-1}$. 
Then $\E_{\Pr_{W}}[ w_{ni} ] = 1$ and, for any fixed $a_{1},\dots,a_{n}$, $\E_{\Pr_{W}}[ \{ \sum_{i=1}^{n} (w_{ni}-1) a_{i} \}^{2} ] = \sum_{i=1}^{n} a_{i}^{2} - n(n^{-1}\sum_{i=1}^{n} a_{i})^{2}$ (which can be directly deduced from the fact that $w_{n1},\dots,w_{nn}$ are nonparametric bootstrap weights).

\begin{proof}[Proof of Lemma \ref{lem3}]
We follow the notation used in Appendix \ref{Appendix A}. 

{\bf Part (i)}: By Lemma \ref{lem2} (i), it suffices to show that $\hat{A}^{*} - \hat{A} = O_{\Pr}( n^{-1/2})$. Decompose $\hat{A}^{*}$ as 
\begin{equation*}
\hat{A}^{*} = \frac{1}{nT} \sum_{i=1}^{n} w_{ni} \sum_{t=1}^{T} \tilde{\bx}_{it} \tilde{\bx}_{it}'  - \frac{1}{n} \sum_{i=1}^{n} w_{ni} \bar{\tilde{\bx}}_{i}\bar{\tilde{\bx}}_{i}' =: \hat{A}_{1}^{*} - \hat{A}_{2}^{*}.
\end{equation*}
Similarly, decompose $\hat{A}$ as 
\begin{equation*}
\hat{A} = \frac{1}{nT} \sum_{i=1}^{n} \sum_{t=1}^{T} \tilde{\bx}_{it} \tilde{\bx}_{it}'  - \frac{1}{n} \sum_{i=1}^{n}  \bar{\tilde{\bx}}_{i}\bar{\tilde{\bx}}_{i}' =: \hat{A}_{1} - \hat{A}_{2}.
\end{equation*}
We wish to show that $\hat{A}_{1}^{*} = \hat{A}_{1} + O_{\Pr}( n^{-1/2} )$ and $\hat{A}_{2}^{*}=\hat{A}_{2} + O_{\Pr}(n^{-1/2} T^{-1})$, which implies the desired result. Fix any $1 \leq a,b \leq p$. Observe that 
\begin{align*}
\E_{\Pr_{W}} [\{ (\hat{A}_{1}^{*} - \hat{A}_{1})^{ab} \}^{2} ] &\leq \frac{1}{n^{2}T^{2}} \sum_{i=1}^{n} \left (\sum_{t=1}^{T} \tilde{x}_{it}^{a} \tilde{x}_{it}^{b} \right )^{2} \\
&\leq \frac{2}{n^{2}T^{2}} \sum_{i=1}^{n} \left (\sum_{t=1}^{T} \tilde{x}_{it}^{a} \tilde{x}_{it}^{b} - \E[\tilde{x}_{i1}^{a} \tilde{x}_{i1}^{b} \mid c_{i}] \right )^{2} + \frac{2}{n^{2}} \sum_{i=1}^{n} \E[ \tilde{x}_{i1}^{a} \tilde{x}_{i1}^{b} \mid c_{i}]^{2}.
\end{align*}
By the proof of Lemma \ref{lem1} (i), the expectation of the first term is $O\{ (nT)^{-1} \}$, while the expectation of the second term is $O (n^{-1})$. Therefore, we have 
$(\hat{A}_{1}^{*} - \hat{A}_{1})^{ab} = O_{\Pr}(n^{-1/2})$. It remains to show that $(\hat{A}_{2}^{*} - \hat{A}_{2})^{ab} = O_{\Pr}(n^{-1/2}T^{-1})$. Observe that
\begin{equation*}
\E_{\Pr_{W}} [\{ (\hat{A}_{2}^{*} - \hat{A}_{2})^{ab} \}^{2} ]   \leq \frac{1}{n^{2}} \sum_{i=1}^{n}  (\bar{\tilde{x}}^{a}_{i} \bar{\tilde{x}}^{b}_{i})^{2}.
\end{equation*}
By the proof of Lemma \ref{lem2} (i), the expectation of the right side is $O(n^{-1}T^{-2})$, which implies the desired  result. 

{\bf Part (ii)}: (ii) follows from the fact that $\E[\hat{S}_{2}^{*}] = \E[\hat{S}_{2}]$ and Lemma \ref{lem2} (ii). 

{\bf Part (iii)}: Since $\E[ \hat{S}^{*}_{2} ] = \E[ \hat{S}_{2} ]$, we have 
\begin{align*}
\hat{S}_{2}^{*} - \E[ \hat{S}_{2}^{*}] &= \hat{S}_{2}^{*} -\hat{S}_{2} + \hat{S}_{2} - \E[ \hat{S}_{2} ] \\
&=\frac{1}{n} \sum_{i=1}^{n} (w_{ni}-1) \bar{\tilde{\bx}}_{i} \bar{\epsilon}_{i} + (\hat{S}_{2} - \E[ \hat{S}_{2} ]).
\end{align*}
By Lemma \ref{lem2} (iii), the second term is $O_{\Pr}(n^{-1/2} T^{-1})$. It remains to show that the first term is $O_{\Pr}(n^{-1/2} T^{-1})$.
Fix any $1 \leq a \leq p$. Observe that 
\begin{equation*}
\E_{\Pr_{W}}\left [ \left \{ \frac{1}{n} \sum_{i=1}^{n} (w_{ni}-1) \bar{\tilde{x}}^{a}_{i} \bar{\epsilon}_{i} \right \}^{2} \right ] \leq \frac{1}{n^{2}}\sum_{i=1}^{n} (\bar{\tilde{x}}^{a}_{i} \bar{\epsilon}_{i})^{2}.
\end{equation*}
By the proof of Lemma \ref{lem2} (iii), the expectation of the right side is $O(n^{-1} T^{-2})$, which implies the desired result. 
\end{proof}

\begin{lemma}
\label{lem4}
Under conditions (A1)-(A3), we have: 
\begin{equation*}
\sup_{\bx \in \mathbb{R}^{p}} | \Pr_{W}\{ (nT)^{-1/2} {\textstyle \sum_{i=1}^{n}(w_{ni} -1) \sum_{t=1}^{T}} (\tilde{\bx}_{it}\epsilon_{it} - \E[\tilde{\bx}_{i1}\epsilon_{i1} \mid c_{i}] )\leq \bx \} - \Pr\{ N(\bm{0}, V_{1}) \leq \bx \} | \stackrel{\Pr}{\to} 0,
\end{equation*}
where $V_{1}=\sum_{k=-\infty}^{\infty} \E[ (\tilde{\bx}_{1,1}\epsilon_{1,1} - \E[\tilde{\bx}_{1,1}\epsilon_{1,1} \mid c_{1}])(\tilde{\bx}_{1,1+k}\epsilon_{1,1+k}-\E[\tilde{\bx}_{1,1+k}\epsilon_{1,1+k} \mid c_{1}])']$, provided that $V_{1}$ is nonsingular. 
\end{lemma}

\begin{proof}
The proof of Lemma \ref{lem4} needs some effort. 
Define $\bm{u}_{ni}= T^{-1/2} \sum_{t=1}^{T} (\epsilon_{it} \bx_{it} - \E[\epsilon_{it} \bx_{it} \mid c_{i}])$. We first show that 
\begin{equation}
\frac{1}{n} \sum_{i=1}^{n} \bm{u}_{ni} \stackrel{\Pr}{\to} \bm{0}, \ \frac{1}{n} \sum_{i=1}^{n} \bm{u}_{ni}\bm{u}_{ni}' \stackrel{\Pr}{\to} V_{1}, \ \frac{1}{n^{3/2}} \sum_{i=1}^{n} \| \bm{u}_{ni} \|^{3}  \stackrel{\Pr}{\to}  0.
\label{event}
\end{equation}
We divide the proof of (\ref{event}) into three steps. 

{\bf Step 1}: (first assertion of (\ref{event})) Fix any $1 \leq a \leq p$. Since $\E[u_{ni}^{a}] = 0$, it suffices to show that 
$\Var (n^{-1} \sum_{i=1}^{n} u_{ni}^{a}) \to 0$. By the cross section independence, $\Var (n^{-1} \sum_{i=1}^{n} u_{ni}^{a}) = n^{-1} \Var (u^{a}_{n1})  = O(n^{-1})$, which implies the desired result.

{\bf Step 2}: (second assertion of (\ref{event})) Fix any $1 \leq a,b \leq p$. We wish to show that $n^{-1} \sum_{i=1}^{n} u_{ni}^{a} u_{ni}^{b} \stackrel{\Pr}{\to} V_{1}^{ab}$.
Observe first that $\E[ n^{-1} \sum_{i=1}^{n} u_{ni}^{a} u_{ni}^{b} ] \to V_{1}^{ab}$.
Thus, it suffices to show that $\Var( n^{-1} \sum_{i=1}^{n} u_{ni}^{a} u_{ni}^{b} )\to 0$ as $n \to \infty$. By the cross section independence, 
\begin{equation}
\Var\left ( \frac{1}{n} \sum_{i=1}^{n} u_{ni}^{a} u_{ni}^{b} \right )
=\frac{1}{n^{2}}\sum_{i=1}^{n} \Var(u_{ni}^{a} u_{ni}^{b} ) = n^{-1} \Var(u_{n1}^{a} u_{n1}^{b} ) \leq n^{-1} \E[(u_{n1}^{a} u_{n1}^{b})^{2}].
\label{moment}
\end{equation}
As in the proof of Proposition \ref{prop2}, we can show that $\E[(u_{n1}^{a} u_{n1}^{b})^{2}] = O(1)$, so that the right side on (\ref{moment}) is $O(n^{-1})$, which implies the desired result.

{\bf Step 3}: (third assertion of (\ref{event})) 
From the proof of Lemma \ref{lem2} part ($iv$) above, one can show that $\E[ | u_{n1}^{a} |^{3}] = O(1) = o(n^{1/2})$ for any $1 \leq a \leq p$, which in turn implies the desired result.

We are now in position to prove the lemma. 
Define $\bm{u}_{n1}^{*},\dots,\bm{u}_{nn}^{*}$ in such a way that $\bm{u}_{ni}^{*} = \bm{u}_{nj}$ if $\bm{z}_{i}^{*} = \bm{z}_{j}$ for some $1 \leq j \leq n$ for $1 \leq i \leq n$. Letting
$\tilde{\bm{u}}^{*}_{ni} = \bm{u}_{ni}^{*} - n^{-1}\sum_{i=1}^{n} \bm{u}_{ni}$, we have
$(nT)^{-1/2} \sum_{i=1}^{n} (w_{ni} -1) \sum_{t=1}^{T} (\tilde{\bx}_{it}\epsilon_{it} - \E[\tilde{\bx}_{i1}\epsilon_{i1} \mid c_{i}] ) = n^{-1/2} \sum_{i=1}^{n}\tilde{\bm{u}}_{ni}^{*}$. Observe that $\tilde{\bm{u}}_{n1}^{*},\dots,\tilde{\bm{u}}_{nn}^{*}$ are i.i.d. with mean zero and covariance matrix $n^{-1} \sum_{i=1}^{n} (\bm{u}_{ni} - n^{-1}\sum_{j=1}^{n}\bm{u}_{nj})(\bm{u}_{ni} - n^{-1}\sum_{j=1}^{n}\bm{u}_{nj})'$ conditional on $\{ c_{i}, \bm{z}_{it} : i \geq 1, t \geq 1 \}$. 
We use the following  fact: let $\{ X_{n} \}$ be a sequence of random variables and let $c$ be a constant; if for any subsequence $\{ n' \}$ of $\{ n \}$ there exists a further subsequence $\{ n'' \}$ such that $X_{n''} \to c$ almost surely, then $X_{n} \stackrel{\Pr}{\to} c$.
Recall that $T=T_{n}$ is indexed by $n$. 
Take any subsequence $\{ n' \}$ of $\{ n \}$. Then there exists a further subsequence $\{ n'' \}$ such that along with the subsequence $\{ n'' \}$, (\ref{event}) holds almost surely. This means that for almost every realization of $c_{i}, \bm{z}_{it}, i \geq 1, t \geq 1$, along with the subsequence $\{n''\}$, 
$n^{-1} \sum_{i=1}^{n} (\bm{u}_{ni} - n^{-1}\sum_{j=1}^{n}\bm{u}_{nj})(\bm{u}_{ni} - n^{-1}\sum_{j=1}^{n}\bm{u}_{nj})' \to V_{1}$ and 
the Lyapunov condition is satisfied for $n^{-1/2} \sum_{i=1}^{n}\tilde{\bm{u}}^{*}_{ni}$. Therefore, along with the subsequence $\{ n'' \}$, conditional on $\{ c_{i}, \bm{z}_{it} : i \geq 1, t \geq 1 \}$, $(nT)^{-1/2} \sum_{i=1}^{n} (w_{ni} -1) \sum_{t=1}^{T} (\tilde{\bx}_{it}\epsilon_{it} - \E[\tilde{\bx}_{i1}\epsilon_{i1} \mid c_{i}] ) \stackrel{d}{\to} N(\bm{0},V_{1})$ for almost every realization of $c_{i}, \bm{z}_{it}, i \geq 1, t \geq 1$. By the above fact, we obtain the desired result.
\end{proof}

We are now in position to prove Proposition \ref{prop3}.

\begin{proof}[Proof of Proposition \ref{prop3}]
We wish to verify that $\hat{S}_{1}^{*} = O_{\Pr}(d_{nT}^{-1/2})$. Observe that 
\begin{align}
\hat{S}_{1}^{*} &= \hat{S}_{1} + \hat{S}_{1}^{*} - \hat{S}_{1} = O_{\Pr}(d_{nT}^{-1/2}) + (\hat{S}_{1}^{*} - \hat{S}_{1}), \notag \\
\intertext{where}
\hat{S}_{1}^{*} - \hat{S}_{1} &= \frac{1}{nT} \sum_{i=1}^{n} (w_{ni}-1) \sum_{t=1}^{T} (\tilde{\bx}_{it} \epsilon_{it} - \E[ \tilde{\bx}_{i1} \epsilon_{i1} \mid c_{i} ]) \notag \\
&\quad + \frac{1}{n} \sum_{i=1}^{n} (w_{ni}-1)   \E[ \tilde{\bx}_{i1} \epsilon_{i1} \mid c_{i} ]. \label{expansion}
\end{align}
Here the first term is $O_{\Pr}\{ (nT)^{-1/2} \}$ and the second term is zero if $ \E[ \tilde{\bx}_{it} \epsilon_{it} \mid c_{i} ] = \bm{0}$ and is $O_{\Pr}(n^{-1/2})$ otherwise. Hence we conclude that $\hat{S}_{1}^{*} = O_{\Pr}(d_{nT}^{-1/2})$. 

By Lemma \ref{lem3}, we have 
\begin{align*}
\hat{\bbeta}^{*} - \bbeta_{0} &= (\hat{A}^{*})^{-1}(\hat{S}_{1}^{*} - \hat{S}_{2}^{*}) =  (\hat{A}^{*})^{-1}(\hat{S}_{1}^{*} - T^{-1}B_{T} + O_{\Pr}(n^{-1/2}T^{-1})) \\
&= A^{-1} \hat{S}_{1}^{*}- A^{-1} \{ {\textstyle \sum}_{m=0}^{\infty} T^{-m-1} (D_{T}A^{-1})^{m} \} B_{T} + O_{\Pr}(n^{-1/2} \max\{T^{-1},d_{nT}^{-1/2} \}).
\end{align*}
Combining this expansion with Proposition \ref{prop1}, we obtain the expansion (\ref{BahadurB}) (note the equivalence in Lemma \ref{lemB1}). 

For (\ref{bootd}), suppose first that $\E[ \tilde{\bx}_{it} \epsilon_{it} \mid c_{i} ] = \bm{0}$. Then the second term in (\ref{expansion}) vanishes, and the assertion (\ref{bootd}) follows from Lemma \ref{lem4}. 
Suppose that $\E[ \tilde{\bx}_{it} \epsilon_{it} \mid c_{i} ] \neq \bm{0}$ with positive probability. Then the second term dominates the first term in (\ref{expansion}), it is routine to verify that 
\begin{equation*}
\sup_{\bx \in \mathbb{R}^{p}} | \Pr_{W}\{  n^{-1/2} {\textstyle \sum}_{i=1}^{n}(w_{ni} -1) \E[ \tilde{\bx}_{i1} \epsilon_{i1} \mid c_{i} ] \leq \bx \} - \Pr\{ N(\bm{0}, V_{2}) \leq \bx \} | \stackrel{\Pr}{\to} 0,
\end{equation*}
where $V_{2} = \E[ \E[ \tilde{\bx}_{1,1} \epsilon_{1,1} \mid c_{1} ]^{\otimes 2} ]$. This completes the proof. 
\end{proof}

\subsection{Proof of Proposition \ref{prop4}}
The proof is similar to that of Proposition \ref{prop2}. Since $\hat{\epsilon}_{it}(\bbeta)=-(\tilde{\bx}_{it}-\bar{\tilde{\bx}}_{i})'(\bbeta-\bbeta_{0}) + \epsilon_{it} - \bar{\epsilon}_{i}$, 
we have 
\begin{equation*}
\hat{\Sigma}_{nT}^{*} = \tilde{\Sigma}_{nT}^{*} + \frac{1}{n^{2}} \sum_{i=1}^{n} w_{ni} (-R_{1i} - R_{2i}^{*} + R_{3i}^{*} - R_{1i}' - (R_{2i}^{*})'+ (R_{3i}^{*})' + R_{4i} + R_{5i}^{*}),
\end{equation*}
where
\begin{align*} 
\tilde{\Sigma}_{nT}^{*} &= n^{-2} {\textstyle \sum}_{i=1}^{n} w_{ni}   ( T^{-1}  {\textstyle \sum}_{t=1}^{T} \tilde{\bx}_{it} \epsilon_{it} )^{\otimes 2}, \\
R_{1i} &= (T^{-1}  {\textstyle \sum}_{t=1}^{T} \tilde{\bx}_{it} \epsilon_{it}  ) (\bar{\tilde{\bx}}_{i} \bar{\epsilon}_{i})', \\
R_{2i}^{*} &=  ( T^{-1}  {\textstyle \sum}_{t=1}^{T} \tilde{\bx}_{it} \epsilon_{it} )\{ T^{-1} {\textstyle \sum}_{t=1}^{T} (\tilde{\bx}_{it} - \bar{\tilde{\bx}}_{i})^{\otimes 2} (\tilde{\bbeta}^{*}-\bbeta_{0}) \}', \\
R_{3i}^{*} &= (\bar{\tilde{\bx}}_{i} \bar{\epsilon}_{i}) \{ T^{-1} {\textstyle \sum}_{t=1}^{T} (\tilde{\bx}_{it} - \bar{\tilde{\bx}}_{i})^{\otimes 2} (\tilde{\bbeta}^{*}-\bbeta_{0}) \}',  \
R_{4i} = (\bar{\tilde{\bx}}_{i} \bar{\epsilon}_{i})^{\otimes 2}, \\
R_{5i}^{*} &= \{ T^{-1}  {\textstyle \sum}_{t=1}^{T} (\tilde{\bx}_{it} - \bar{\tilde{\bx}}_{i})^{\otimes 2} (\tilde{\bbeta}^{*}-\bbeta_{0}) \}^{\otimes 2}.
\end{align*}
By the proof of Proposition \ref{prop2}, together with the assumption that 
$\| \tilde{\bbeta}^{*}-\bbeta_{0} \|  = \\ O_{\Pr_{W}}[ \max \{ d_{nT}^{-1/2}, T^{-1}  \}]$, and Lemma \ref{lemB1}, we have 
\begin{equation*}
\| \hat{\Sigma}_{nT}^{*} - \tilde{\Sigma}_{nT}^{*} \|_{\op} = O_{\Pr_{W}}[ \max \{ n^{-1/2}T^{-1} d_{nT}^{-1/2}, n^{-1/2} d_{nT}^{-1} \}],
\end{equation*}
in probability. 
Let 
\begin{equation*}
\tilde{\Sigma}_{nT} = n^{-2} {\textstyle \sum}_{i=1}^{n}   ( T^{-1}  {\textstyle \sum}_{t=1}^{T} \tilde{\bx}_{it} \epsilon_{it} )^{\otimes 2}.
\end{equation*}
By the proof of Proposition \ref{prop2}, we have $\| \tilde{\Sigma}_{nT} - \Sigma_{nT} \|_{\op} = O_{\Pr}(n^{-1/2}d_{nT}^{-1})$, and hence we only need to show that 
$\| \tilde{\Sigma}_{nT}^{*} - \tilde{\Sigma}_{nT} \|_{\op} =  O_{\Pr_{W}}(n^{-1/2}d_{nT}^{-1})$ in probability. Fix any $1 \leq a,b \leq p$, and observe that
\begin{align*}
&(\tilde{\Sigma}_{nT}^{*} - \tilde{\Sigma}_{nT})^{ab} \\
&= \frac{1}{n^{2}} \sum_{i=1}^{n}(w_{ni}-1) \left ( \frac{1}{T^{2}}\sum_{s=1}^{T} \sum_{t=1}^{T} \tilde{x}_{is}^{a} \epsilon_{is} \tilde{x}_{it}^{b} \epsilon_{it} \right ) \\
&=\frac{1}{n^{2}} \sum_{i=1}^{n} (w_{ni}-1)\left \{ \frac{1}{T^{2}}\sum_{s=1}^{T} \sum_{t=1}^{T} (\tilde{x}_{is}^{a} \epsilon_{is} - \E[\tilde{x}_{i1}^{a} \epsilon_{i1} \mid c_{i}]) ( \tilde{x}_{it}^{b} \epsilon_{it} - \E[ \tilde{x}_{i1}^{b} \epsilon_{i1} \mid c_{i}]) \right \} \\
&\quad + \frac{1}{n^{2}} \sum_{i=1}^{n} (w_{ni}-1)\E[\tilde{x}_{i1}^{a} \epsilon_{i1} \mid c_{i}] \left \{ \frac{1}{T} \sum_{t=1}^{T} (\tilde{x}_{it}^{b} \epsilon_{it} -\E[\tilde{x}_{i1}^{b} \epsilon_{i1} \mid c_{i}]) \right \} \\
&\quad + \frac{1}{n^{2}} \sum_{i=1}^{n} (w_{ni}-1)\E[\tilde{x}_{i1}^{b} \epsilon_{i1} \mid c_{i}] \left \{  \frac{1}{T} \sum_{t=1}^{T} (\tilde{x}_{it}^{a} \epsilon_{it}-\E[\tilde{x}_{i1}^{a} \epsilon_{i1} \mid c_{i}]) \right \} \\
&\quad -  \frac{1}{n^{2}} \sum_{i=1}^{n} (w_{ni}-1)\E[\tilde{x}_{i1}^{a}\epsilon_{i1} \mid c_{i}] \E[\tilde{x}_{i1}^{b}\epsilon_{i1} \mid c_{i}] \\
&=:(I) + (II) + (III) - (IV). 
\end{align*}
By the proof of Proposition \ref{prop2}, we can deduce that $\E[ (I)^{2} ] = \E[ \E_{W} [ (I)^2 ]] = O(n^{-3}T^{-2})$, and when $\E[ \tilde{\bx}_{it} \epsilon_{it} \mid c_{i} ] \neq \bm{0}$ with positive probability, 
$\E[ (II)^{2} ] = O(n^{-3}T^{-1}), \E[ (III)^{2} ] = O(n^{-3}T^{-1})$ and $\E[ (IV)^{2} ] = O(n^{-3})$. Therefore, we obtain the desired assertion. 

\qed

\newpage

\linespread{0.95}
\scriptsize

\begin{landscape}
\begin{longtable} {cc|cccc|ccccccc}
\caption{Bias and Empirical Coverage for panel AR(1) model}
\label{table.mc.0}\\
\hline\hline																									
$n$	&	$T$	&	FE	&	HK	&	GMM & HPJ	&	FE-CCM	&	HK	& GMM	& HPJ-CCM	&	HPJ-FEB	&	HPJ-HPJB	&	HPJ-HPJPB	\\
\hline
                  &           &\multicolumn{4}{c|}{Bias} & \multicolumn{7}{c}{Coverage}\\
\hline                 
50	&	12	&	-0.1829	&	-0.0482	&	-0.0898	&	-0.0023	&	0.0005	&	0.6610	&	0.7580	&	0.7870	&	0.7420	&	0.9145	&	0.8995	\\
	&		&(	0.0366	)&(	0.0397	)&(	0.0717	)&(	0.0605	)&		&		&		&		&		&		&		\\
	&		&[	1.0337	]&[	1.2365	]&[	1.0423	]&[	1.5825	]&		&		&		&		&		&		&		\\
	&	16	&	-0.1351	&	-0.0311	&	-0.0756	&	0.0073	&	0.0025	&	0.7550	&	0.6910	&	0.7845	&	0.7430	&	0.9255	&	0.8970	\\
	&		&(	0.0299	)&(	0.0318	)&(	0.0497	)&(	0.0470	)&		&		&		&		&		&		&		\\
	&		&[	1.0455	]&[	1.2048	]&[	1.0134	]&[	1.5672	]&		&		&		&		&		&		&		\\
	&	20	&	-0.1064	&	-0.0218	&	-0.0649	&	0.0087	&	0.0070	&	0.8115	&	0.6350	&	0.7910	&	0.7515	&	0.9305	&	0.8980	\\
	&		&(	0.0257	)&(	0.0270	)&(	0.0390	)&(	0.0387	)&		&		&		&		&		&		&		\\
	&		&[	1.0458	]&[	1.1849	]&[	1.0103	]&[	1.5475	]&		&		&		&		&		&		&		\\
	&	24	&	-0.0870	&	-0.0156	&	-0.0581	&	0.0091	&	0.0230	&	0.8485	&	0.5815	&	0.7840	&	0.7575	&	0.9210	&	0.8965	\\
	&		&(	0.0226	)&(	0.0236	)&(	0.0330	)&(	0.0329	)&		&		&		&		&		&		&		\\
	&		&[	1.0453	]&[	1.1663	]&[	1.0355	]&[	1.5169	]&		&		&		&		&		&		&		\\
\hline
100	&	12	&	-0.1814	&	-0.0465	&	-0.0471	&	-0.0008	&	0.0000	&	0.4685	&	0.8500	&	0.7865	&	0.7490	&	0.9310	&	0.9250	\\
	&		&(	0.0256	)&(	0.0278	)&(	0.0521	)&(	0.0423	)&		&		&		&		&		&		&		\\
	&		&[	1.0095	]&[	1.2262	]&[	1.0330	]&[	1.5654	]&		&		&		&		&		&		&		\\
	&	16	&	-0.1342	&	-0.0301	&	-0.0417	&	0.0069	&	0.0000	&	0.6340	&	0.8255	&	0.7575	&	0.7385	&	0.9375	&	0.9105	\\
	&		&(	0.0214	)&(	0.0228	)&(	0.0357	)&(	0.0343	)&		&		&		&		&		&		&		\\
	&		&[	1.0412	]&[	1.2222	]&[	0.9890	]&[	1.6165	]&		&		&		&		&		&		&		\\
	&	20	&	-0.1046	&	-0.0199	&	-0.0361	&	0.0098	&	0.0000	&	0.7430	&	0.7730	&	0.7640	&	0.7440	&	0.9265	&	0.9075	\\
	&		&(	0.0180	)&(	0.0189	)&(	0.0289	)&(	0.0277	)&		&		&		&		&		&		&		\\
	&		&[	1.0192	]&[	1.1760	]&[	1.0191	]&[	1.5633	]&		&		&		&		&		&		&		\\
	&	24	&	-0.0863	&	-0.0148	&	-0.0335	&	0.0094	&	0.0000	&	0.7870	&	0.7240	&	0.7705	&	0.7525	&	0.9190	&	0.9050	\\
	&		&(	0.0157	)&(	0.0163	)&(	0.0233	)&(	0.0232	)&		&		&		&		&		&		&		\\
	&		&[	1.0145	]&[	1.1422	]&[	0.9861	]&[	1.5131	]&		&		&		&		&		&		&		\\
\hline
200	&	12	&	-0.1815	&	-0.0466	&	-0.0258	&	-0.0007	&	0.0000	&	0.2195	&	0.9000	&	0.7875	&	0.7515	&	0.9340	&	0.9330	\\
	&		&(	0.0185	)&(	0.0201	)&(	0.0359	)&(	0.0303	)&		&		&		&		&		&		&		\\
	&		&[	1.0203	]&[	1.2507	]&[	0.9874	]&[	1.5811	]&		&		&		&		&		&		&		\\
	&	16	&	-0.1329	&	-0.0287	&	-0.0219	&	0.0075	&	0.0000	&	0.4295	&	0.8880	&	0.7775	&	0.7555	&	0.9295	&	0.9130	\\
	&		&(	0.0143	)&(	0.0152	)&(	0.0254	)&(	0.0232	)&		&		&		&		&		&		&		\\
	&		&[	0.9776	]&[	1.1576	]&[	0.9765	]&[	1.5514	]&		&		&		&		&		&		&		\\
	&	20	&	-0.1045	&	-0.0198	&	-0.0189	&	0.0095	&	0.0000	&	0.5930	&	0.8510	&	0.7440	&	0.7300	&	0.9085	&	0.8995	\\
	&		&(	0.0128	)&(	0.0134	)&(	0.0209	)&(	0.0194	)&		&		&		&		&		&		&		\\
	&		&[	1.0191	]&[	1.1802	]&[	1.0226	]&[	1.5466	]&		&		&		&		&		&		&		\\
	&	24	&	-0.0858	&	-0.0144	&	-0.0177	&	0.0088	&	0.0000	&	0.6890	&	0.8250	&	0.7560	&	0.7445	&	0.9150	&	0.9025	\\
	&		&(	0.0108	)&(	0.0113	)&(	0.0174	)&(	0.0160	)&		&		&		&		&		&		&		\\
	&		&[	0.9808	]&[	1.1181	]&[	1.0198	]&[	1.4731	]&		&		&		&		&		&		&		\\
\hline
\end{longtable}
Notes: Monte Carlo experiments based on 2,000 repetitions. The standard  deviations are inside parenthesis. Inside brackets are the ratio of the averages of the standard errors to the simulation of the standard deviations. The DGP is $
y_{it} = c_{i} + \phi y_{i,t-1} + u_{it},$ with true parameter $\phi=0.8$. 
\end{landscape}

\begin{landscape}
\begin{longtable} {cc|cccc|ccccccc}
\caption{Bias and Empirical Coverage for panel AR(2) model}
\label{table.mc.1}\\
\hline\hline																									
$n$	&	$T$	&	FE	&	HK	&	GMM & HPJ	&	FE-CCM	&	HK	& GMM	& HPJ-CCM	&	HPJ-FEB	&	HPJ-HPJB	&	HPJ-HPJPB	\\
\hline
                  &           &\multicolumn{4}{c|}{Bias} & \multicolumn{7}{c}{Coverage}\\
\hline                 
50	&	12	&	-0.3534	&	-0.2437	&	-0.5201	&	-0.0943	&	0.0000	&	0.0025	&	0.0000	&	0.6000	&	0.5615	&	0.7230	&	0.7960	\\
	&		&(	0.0587	)&(	0.0636	)&(	0.1213	)&(	0.0862	)&		&		&		&		&		&		&		\\
	&		&[	1.0592	]&[	1.6464	]&[	1.4145	]&[	1.4177	]&		&		&		&		&		&		&		\\
	&	16	&	-0.2719	&	-0.1845	&	-0.4207	&	-0.0500	&	0.0005	&	0.0070	&	0.0000	&	0.7440	&	0.7180	&	0.8525	&	0.8770	\\
	&		&(	0.0492	)&(	0.0523	)&(	0.0977	)&(	0.0704	)&		&		&		&		&		&		&		\\
	&		&[	1.0309	]&[	1.6146	]&[	1.4062	]&[	1.3855	]&		&		&		&		&		&		&		\\
	&	20	&	-0.2205	&	-0.1480	&	-0.3500	&	-0.0285	&	0.0000	&	0.0150	&	0.0000	&	0.7810	&	0.7660	&	0.8800	&	0.8910	\\
	&		&(	0.0444	)&(	0.0467	)&(	0.0811	)&(	0.0633	)&		&		&		&		&		&		&		\\
	&		&[	1.0570	]&[	1.6543	]&[	1.4037	]&[	1.4390	]&		&		&		&		&		&		&		\\
	&	24	&	-0.1852	&	-0.1234	&	-0.2984	&	-0.0162	&	0.0005	&	0.0235	&	0.0000	&	0.8245	&	0.8145	&	0.9175	&	0.9175	\\
	&		&(	0.0377	)&(	0.0393	)&(	0.0663	)&(	0.0536	)&		&		&		&		&		&		&		\\
	&		&[	0.9921	]&[	1.5584	]&[	1.3460	]&[	1.3665	]&		&		&		&		&		&		&		\\
\hline
100	&	12	&	-0.3531	&	-0.2433	&	-0.5241	&	-0.0951	&	0.0000	&	0.0000	&	0.0000	&	0.4230	&	0.3945	&	0.5985	&	0.6535	\\
	&		&(	0.0408	)&(	0.0442	)&(	0.0922	)&(	0.0598	)&		&		&		&		&		&		&		\\
	&		&[	1.0240	]&[	1.6167	]&[	1.4195	]&[	1.3820	]&		&		&		&		&		&		&		\\
	&	16	&	-0.2716	&	-0.1842	&	-0.4260	&	-0.0500	&	0.0000	&	0.0000	&	0.0000	&	0.6545	&	0.6350	&	0.7920	&	0.8200	\\
	&		&(	0.0343	)&(	0.0364	)&(	0.0743	)&(	0.0497	)&		&		&		&		&		&		&		\\
	&		&[	0.9987	]&[	1.5902	]&[	1.4024	]&[	1.3721	]&		&		&		&		&		&		&		\\
	&	20	&	-0.2191	&	-0.1466	&	-0.3556	&	-0.0278	&	0.0000	&	0.0000	&	0.0000	&	0.7340	&	0.7220	&	0.8645	&	0.8770	\\
	&		&(	0.0316	)&(	0.0332	)&(	0.0647	)&(	0.0457	)&		&		&		&		&		&		&		\\
	&		&[	1.0415	]&[	1.6647	]&[	1.4574	]&[	1.4550	]&		&		&		&		&		&		&		\\
	&	24	&	-0.1839	&	-0.1220	&	-0.3067	&	-0.0152	&	0.0000	&	0.0010	&	0.0000	&	0.7995	&	0.7970	&	0.9060	&	0.9125	\\
	&		&(	0.0280	)&(	0.0292	)&(	0.0554	)&(	0.0391	)&		&		&		&		&		&		&		\\
	&		&[	1.0268	]&[	1.6383	]&[	1.4494	]&[	1.4042	]&		&		&		&		&		&		&		\\
\hline
200	&	12	&	-0.3518	&	-0.2420	&	-0.5212	&	-0.0929	&	0.0000	&	0.0000	&	0.0000	&	0.2200	&	0.1990	&	0.3795	&	0.4220	\\
	&		&(	0.0282	)&(	0.0305	)&(	0.0652	)&(	0.0414	)&		&		&		&		&		&		&		\\
	&		&[	0.9874	]&[	1.5791	]&[	1.3727	]&[	1.3457	]&		&		&		&		&		&		&		\\
	&	16	&	-0.2714	&	-0.1840	&	-0.4277	&	-0.0494	&	0.0000	&	0.0000	&	0.0000	&	0.5080	&	0.4880	&	0.6705	&	0.7050	\\
	&		&(	0.0246	)&(	0.0262	)&(	0.0560	)&(	0.0361	)&		&		&		&		&		&		&		\\
	&		&[	1.0091	]&[	1.6156	]&[	1.4378	]&[	1.4059	]&		&		&		&		&		&		&		\\
	&	20	&	-0.2187	&	-0.1461	&	-0.3603	&	-0.0268	&	0.0000	&	0.0000	&	0.0000	&	0.6890	&	0.6790	&	0.8340	&	0.8515	\\
	&		&(	0.0219	)&(	0.0230	)&(	0.0481	)&(	0.0308	)&		&		&		&		&		&		&		\\
	&		&[	1.0143	]&[	1.6281	]&[	1.4588	]&[	1.3852	]&		&		&		&		&		&		&		\\
	&	24	&	-0.1837	&	-0.1218	&	-0.3116	&	-0.0152	&	0.0000	&	0.0000	&	0.0000	&	0.7710	&	0.7670	&	0.8935	&	0.9020	\\
	&		&(	0.0198	)&(	0.0207	)&(	0.0406	)&(	0.0278	)&		&		&		&		&		&		&		\\
	&		&[	1.0164	]&[	1.6395	]&[	1.4264	]&[	1.4035	]&		&		&		&		&		&		&		\\\hline
\end{longtable}
Notes: Monte Carlo experiments based on 2,000 repetitions. The standard  deviations are inside parenthesis. Inside brackets are the ratio of the averages of standard errors to simulation of the standard deviations. The DGP is $
y_{it} = c_{i} + \phi_{1} y_{i,t-1}+ \phi_{2} y_{i,t-2} + u_{it},$ with true parameter $\phi_{1}=\phi_{2}=0.4$. 
\end{landscape}

\newpage

\begin{landscape}
\begin{longtable} {cc|cccc|ccccccc}
\caption{Bias and Empirical Coverage for panel AR(2) model}
\label{table.mc.2}\\
\hline\hline																									
$n$	&	$T$	&	FE	&	HK	&	GMM & HPJ	&	FE-CCM	&	HK	& GMM	& HPJ-CCM	&	HPJ-FEB	&	HPJ-HPJB	&	HPJ-HPJPB	\\
\hline
                  &           &\multicolumn{4}{c|}{Bias} & \multicolumn{7}{c}{Coverage}\\
\hline                             
50	&	12	&	-0.0201	&	0.0374	&	0.0364	&	0.0065	&	0.8555	&	0.9180	&	0.9775	&	0.9170	&	0.9040	&	0.9245	&	0.9350	\\
	&		&(	0.0268	)&(	0.0290	)&(	0.0307	)&(	0.0290	)&		&		&		&		&		&		&		\\
	&		&[	1.0289	]&[	0.7480	]&[	0.6187	]&[	1.0853	]&		&		&		&		&		&		&		\\
	&	16	&	-0.0151	&	0.0283	&	0.0255	&	0.0038	&	0.8830	&	0.9435	&	0.9900	&	0.9230	&	0.9170	&	0.9300	&	0.9435	\\
	&		&(	0.0226	)&(	0.0240	)&(	0.0250	)&(	0.0240	)&		&		&		&		&		&		&		\\
	&		&[	1.0096	]&[	0.7137	]&[	0.5954	]&[	1.0522	]&		&		&		&		&		&		&		\\
	&	20	&	-0.0107	&	0.0242	&	0.0199	&	0.0045	&	0.9040	&	0.9515	&	0.9900	&	0.9225	&	0.9105	&	0.9290	&	0.9350	\\
	&		&(	0.0209	)&(	0.0219	)&(	0.0227	)&(	0.0217	)&		&		&		&		&		&		&		\\
	&		&[	1.0543	]&[	0.7275	]&[	0.6117	]&[	1.0772	]&		&		&		&		&		&		&		\\
	&	24	&	-0.0081	&	0.0212	&	0.0159	&	0.0046	&	0.9100	&	0.9650	&	0.9950	&	0.9390	&	0.9320	&	0.9415	&	0.9530	\\
	&		&(	0.0182	)&(	0.0190	)&(	0.0197	)&(	0.0187	)&		&		&		&		&		&		&		\\
	&		&[	1.0066	]&[	0.6881	]&[	0.5893	]&[	1.0189	]&		&		&		&		&		&		&		\\
\hline
100	&	12	&	-0.0206	&	0.0369	&	0.0421	&	0.0057	&	0.7925	&	0.8005	&	0.8960	&	0.9145	&	0.9005	&	0.9230	&	0.9330	\\
	&		&(	0.0192	)&(	0.0208	)&(	0.0221	)&(	0.0206	)&		&		&		&		&		&		&		\\
	&		&[	1.0344	]&[	0.7598	]&[	0.6246	]&[	1.0830	]&		&		&		&		&		&		&		\\
	&	16	&	-0.0147	&	0.0288	&	0.0328	&	0.0043	&	0.8350	&	0.8530	&	0.9310	&	0.9295	&	0.9165	&	0.9335	&	0.9355	\\
	&		&(	0.0162	)&(	0.0172	)&(	0.0181	)&(	0.0170	)&		&		&		&		&		&		&		\\
	&		&[	1.0184	]&[	0.7236	]&[	0.6017	]&[	1.0508	]&		&		&		&		&		&		&		\\
	&	20	&	-0.0110	&	0.0239	&	0.0269	&	0.0042	&	0.8675	&	0.8875	&	0.9475	&	0.9315	&	0.9255	&	0.9325	&	0.9340	\\
	&		&(	0.0144	)&(	0.0152	)&(	0.0158	)&(	0.0149	)&		&		&		&		&		&		&		\\
	&		&[	1.0156	]&[	0.7113	]&[	0.5954	]&[	1.0344	]&		&		&		&		&		&		&		\\
	&	24	&	-0.0086	&	0.0206	&	0.0227	&	0.0040	&	0.8915	&	0.9065	&	0.9675	&	0.9320	&	0.9280	&	0.9340	&	0.9385	\\
	&		&(	0.0129	)&(	0.0134	)&(	0.0139	)&(	0.0134	)&		&		&		&		&		&		&		\\
	&		&[	0.9987	]&[	0.6898	]&[	0.5782	]&[	1.0225	]&		&		&		&		&		&		&		\\
\hline
200	&	12	&	-0.0210	&	0.0364	&	0.0446	&	0.0053	&	0.6475	&	0.5475	&	0.6295	&	0.9185	&	0.9105	&	0.9325	&	0.9315	\\
	&		&(	0.0133	)&(	0.0144	)&(	0.0153	)&(	0.0141	)&		&		&		&		&		&		&		\\
	&		&[	1.0078	]&[	0.7440	]&[	0.6046	]&[	1.0382	]&		&		&		&		&		&		&		\\
	&	16	&	-0.0144	&	0.0291	&	0.0365	&	0.0046	&	0.7530	&	0.6240	&	0.6705	&	0.9160	&	0.9065	&	0.9215	&	0.9205	\\
	&		&(	0.0113	)&(	0.0120	)&(	0.0127	)&(	0.0119	)&		&		&		&		&		&		&		\\
	&		&[	1.0002	]&[	0.7146	]&[	0.5927	]&[	1.0359	]&		&		&		&		&		&		&		\\
	&	20	&	-0.0106	&	0.0243	&	0.0308	&	0.0046	&	0.8190	&	0.6935	&	0.7225	&	0.9205	&	0.9090	&	0.9205	&	0.9230	\\
	&		&(	0.0101	)&(	0.0106	)&(	0.0109	)&(	0.0105	)&		&		&		&		&		&		&		\\
	&		&[	1.0011	]&[	0.7024	]&[	0.5797	]&[	1.0266	]&		&		&		&		&		&		&		\\
	&	24	&	-0.0084	&	0.0208	&	0.0267	&	0.0043	&	0.8475	&	0.7465	&	0.7570	&	0.9140	&	0.9090	&	0.9120	&	0.9195	\\
	&		&(	0.0093	)&(	0.0097	)&(	0.0100	)&(	0.0096	)&		&		&		&		&		&		&		\\
	&		&[	1.0103	]&[	0.7022	]&[	0.5852	]&[	1.0275	]&		&		&		&		&		&		&		\\
\hline
\end{longtable}
Notes: Monte Carlo experiments based on 2,000 repetitions. The standard deviations are inside parenthesis. Inside brackets are the ratio of the averages of the standard errors to the simulation of the standard deviations. The DGP is $
y_{it} = c_{i} + \phi_{1} y_{i,t-1}+ \phi_{2} y_{i,t-2} + u_{it},$ with true parameter $\phi_{1}=\phi_{2}=-0.4$. 
\end{landscape}

\newpage

\begin{landscape}
\begin{longtable} {cc|cccc|ccccccc}
\caption{Bias and Empirical Coverage for panel AR(2) model}
\label{table.mc.5}\\
\hline\hline																									
$n$	&	$T$	&	FE	&	HK	&	GMM & HPJ	&	FE-CCM	&	HK	& GMM	& HPJ-CCM	&	HPJ-FEB	&	HPJ-HPJB	&	HPJ-HPJPB	\\
\hline
                  &           &\multicolumn{4}{c|}{Bias} & \multicolumn{7}{c}{Coverage}\\
\hline                             
50	&	12	&	-0.3161	&	-0.1476	&	-0.3222	&	-0.0649	&	0.0000	&	0.0765	&	0.1395	&	0.6760	&	0.6480	&	0.8020	&	0.9470	\\
	&		&(	0.0518	)&(	0.0557	)&(	0.1365	)&(	0.0767	)&		&		&		&		&		&		&		\\
	&		&[	1.4428	]&[	1.5008	]&[	1.3745	]&[	1.4312	]&		&		&		&		&		&		&		\\
	&	16	&	-0.2415	&	-0.1005	&	-0.2524	&	-0.0329	&	0.0000	&	0.1600	&	0.0910	&	0.7705	&	0.7580	&	0.8775	&	0.9650	\\
	&		&(	0.0422	)&(	0.0437	)&(	0.0961	)&(	0.0617	)&		&		&		&		&		&		&		\\
	&		&[	1.4489	]&[	1.4195	]&[	1.3400	]&[	1.4107	]&		&		&		&		&		&		&		\\
	&	20	&	-0.1931	&	-0.0703	&	-0.2074	&	-0.0154	&	0.0000	&	0.3180	&	0.0610	&	0.8210	&	0.8025	&	0.9100	&	0.9695	\\
	&		&(	0.0366	)&(	0.0373	)&(	0.0748	)&(	0.0532	)&		&		&		&		&		&		&		\\
	&		&[	1.4951	]&[	1.3995	]&[	1.3368	]&[	1.4228	]&		&		&		&		&		&		&		\\
	&	24	&	-0.1632	&	-0.0531	&	-0.1791	&	-0.0112	&	0.0000	&	0.4380	&	0.0395	&	0.8015	&	0.7965	&	0.9030	&	0.9690	\\
	&		&(	0.0336	)&(	0.0341	)&(	0.0601	)&(	0.0480	)&		&		&		&		&		&		&		\\
	&		&[	1.5805	]&[	1.4346	]&[	1.3171	]&[	1.4633	]&		&		&		&		&		&		&		\\
\hline
100	&	12	&	-0.3138	&	-0.1455	&	-0.2227	&	-0.0620	&	0.0000	&	0.0040	&	0.2885	&	0.5865	&	0.5690	&	0.7205	&	0.9215	\\
	&		&(	0.0366	)&(	0.0392	)&(	0.1084	)&(	0.0548	)&		&		&		&		&		&		&		\\
	&		&[	1.4146	]&[	1.4938	]&[	1.2980	]&[	1.4406	]&		&		&		&		&		&		&		\\
	&	16	&	-0.2394	&	-0.0984	&	-0.1682	&	-0.0309	&	0.0000	&	0.0295	&	0.2305	&	0.7245	&	0.7175	&	0.8530	&	0.9600	\\
	&		&(	0.0302	)&(	0.0313	)&(	0.0776	)&(	0.0449	)&		&		&		&		&		&		&		\\
	&		&[	1.4477	]&[	1.4367	]&[	1.3110	]&[	1.4475	]&		&		&		&		&		&		&		\\
	&	20	&	-0.1932	&	-0.0703	&	-0.1388	&	-0.0172	&	0.0000	&	0.0970	&	0.1855	&	0.7945	&	0.7835	&	0.9020	&	0.9690	\\
	&		&(	0.0264	)&(	0.0269	)&(	0.0585	)&(	0.0382	)&		&		&		&		&		&		&		\\
	&		&[	1.5053	]&[	1.4262	]&[	1.2883	]&[	1.4448	]&		&		&		&		&		&		&		\\
	&	24	&	-0.1607	&	-0.0505	&	-0.1175	&	-0.0078	&	0.0000	&	0.2265	&	0.1635	&	0.8030	&	0.7950	&	0.9215	&	0.9775	\\
	&		&(	0.0236	)&(	0.0237	)&(	0.0478	)&(	0.0340	)&		&		&		&		&		&		&		\\
	&		&[	1.5410	]&[	1.4149	]&[	1.2828	]&[	1.4585	]&		&		&		&		&		&		&		\\
\hline
200	&	12	&	-0.3128	&	-0.1447	&	-0.1126	&	-0.0604	&	0.0000	&	0.0000	&	0.6185	&	0.4285	&	0.4050	&	0.6100	&	0.8805	\\
	&		&(	0.0253	)&(	0.0268	)&(	0.0828	)&(	0.0381	)&		&		&		&		&		&		&		\\
	&		&[	1.3624	]&[	1.4475	]&[	1.1826	]&[	1.4054	]&		&		&		&		&		&		&		\\
	&	16	&	-0.2387	&	-0.0979	&	-0.0749	&	-0.0299	&	0.0000	&	0.0000	&	0.6560	&	0.6550	&	0.6505	&	0.8255	&	0.9515	\\
	&		&(	0.0214	)&(	0.0220	)&(	0.0570	)&(	0.0308	)&		&		&		&		&		&		&		\\
	&		&[	1.4342	]&[	1.4296	]&[	1.1797	]&[	1.4002	]&		&		&		&		&		&		&		\\
	&	20	&	-0.1923	&	-0.0693	&	-0.0577	&	-0.0159	&	0.0000	&	0.0095	&	0.6360	&	0.7580	&	0.7485	&	0.8855	&	0.9675	\\
	&		&(	0.0188	)&(	0.0193	)&(	0.0444	)&(	0.0271	)&		&		&		&		&		&		&		\\
	&		&[	1.5034	]&[	1.4465	]&[	1.2020	]&[	1.4456	]&		&		&		&		&		&		&		\\
	&	24	&	-0.1607	&	-0.0505	&	-0.0444	&	-0.0087	&	0.0000	&	0.0475	&	0.6585	&	0.7950	&	0.7915	&	0.9175	&	0.9765	\\
	&		&(	0.0167	)&(	0.0169	)&(	0.0362	)&(	0.0236	)&		&		&		&		&		&		&		\\
	&		&[	1.5366	]&[	1.4254	]&[	1.2127	]&[	1.4285	]&		&		&		&		&		&		&		\\
\hline
\end{longtable}
Notes: Monte Carlo based on 2,000 repetitions. Standard deviations are inside parenthesis. Inside brackets are ratios of averages of the standard errors to simulation of the standard deviations. DGP is $
y_{it} = c_{i} + \phi_{1} y_{i,t-1}+ \phi_{2} y_{i,t-2} + \rho_{1}x_{it} + \rho_{2}x_{it-1} + u_{it},$ with true parameter $\phi_{1}=\phi_{2}=0.4$ and $\rho_{1}=\rho_{2}=0.5$. 
\end{landscape}

\newpage

\begin{landscape}
\begin{longtable} {cc|cccc|ccccccc}
\caption{Bias and Empirical Coverage for random coefficients AR model}
\label{table.mc.4}\\
\hline\hline																									
$n$	&	$T$	&	FE	&	HK	&	GMM & HPJ	&	FE-CCM	&	HK	& GMM	& HPJ-CCM	&	HPJ-FEB	&	HPJ-HPJB	&	HPJ-HPJPB	\\
\hline
                  &           &\multicolumn{4}{c|}{Bias} & \multicolumn{7}{c}{Coverage}\\
\hline 
50	&	12	&	-0.2076	&	-0.0949	&	-0.1267	&	-0.0411	&	0.0750	&	0.3740	&	0.4210	&	0.7225	&	0.7250	&	0.8375	&	0.8570	\\
	&		&(	0.0609	)&(	0.0660	)&(	0.0795	)&(	0.0844	)&		&		&		&		&		&		&		\\
	&		&[	1.0806	]&[	1.7323	]&[	1.3986	]&[	1.5547	]&		&		&		&		&		&		&		\\
	&	16	&	-0.1591	&	-0.0715	&	-0.1087	&	-0.0241	&	0.1720	&	0.4460	&	0.3815	&	0.7935	&	0.8090	&	0.8780	&	0.8990	\\
	&		&(	0.0541	)&(	0.0574	)&(	0.0654	)&(	0.0733	)&		&		&		&		&		&		&		\\
	&		&[	1.0101	]&[	1.7774	]&[	1.4647	]&[	1.4632	]&		&		&		&		&		&		&		\\
	&	20	&	-0.1290	&	-0.0575	&	-0.0956	&	-0.0158	&	0.3155	&	0.4660	&	0.3550	&	0.7945	&	0.8205	&	0.8885	&	0.9090	\\
	&		&(	0.0541	)&(	0.0568	)&(	0.0622	)&(	0.0712	)&		&		&		&		&		&		&		\\
	&		&[	1.0530	]&[	1.9933	]&[	1.6559	]&[	1.4994	]&		&		&		&		&		&		&		\\
	&	24	&	-0.1099	&	-0.0494	&	-0.0888	&	-0.0125	&	0.4295	&	0.4930	&	0.3420	&	0.8235	&	0.8365	&	0.8975	&	0.9180	\\
	&		&(	0.0530	)&(	0.0552	)&(	0.0580	)&(	0.0660	)&		&		&		&		&		&		&		\\
	&		&[	1.0572	]&[	2.1466	]&[	1.7639	]&[	1.4307	]&		&		&		&		&		&		&		\\
\hline
100	&	12	&	-0.2060	&	-0.0932	&	-0.1075	&	-0.0379	&	0.0015	&	0.1715	&	0.3070	&	0.7095	&	0.7280	&	0.8430	&	0.8615	\\
	&		&(	0.0421	)&(	0.0456	)&(	0.0565	)&(	0.0590	)&		&		&		&		&		&		&		\\
	&		&[	1.0136	]&[	1.6894	]&[	1.3783	]&[	1.4919	]&		&		&		&		&		&		&		\\
	&	16	&	-0.1583	&	-0.0707	&	-0.0917	&	-0.0233	&	0.0320	&	0.2700	&	0.2690	&	0.7730	&	0.8055	&	0.8805	&	0.8985	\\
	&		&(	0.0395	)&(	0.0420	)&(	0.0493	)&(	0.0527	)&		&		&		&		&		&		&		\\
	&		&[	1.0111	]&[	1.8367	]&[	1.5291	]&[	1.4531	]&		&		&		&		&		&		&		\\
	&	20	&	-0.1276	&	-0.0559	&	-0.0789	&	-0.0135	&	0.0950	&	0.3310	&	0.2770	&	0.8085	&	0.8370	&	0.9005	&	0.9135	\\
	&		&(	0.0380	)&(	0.0399	)&(	0.0447	)&(	0.0494	)&		&		&		&		&		&		&		\\
	&		&[	1.0051	]&[	1.9788	]&[	1.6517	]&[	1.4253	]&		&		&		&		&		&		&		\\
	&	24	&	-0.1079	&	-0.0474	&	-0.0724	&	-0.0093	&	0.1680	&	0.3700	&	0.2615	&	0.8340	&	0.8635	&	0.9195	&	0.9290	\\
	&		&(	0.0366	)&(	0.0382	)&(	0.0407	)&(	0.0470	)&		&		&		&		&		&		&		\\
	&		&[	1.0008	]&[	2.0981	]&[	1.7193	]&[	1.4012	]&		&		&		&		&		&		&		\\
\hline
200	&	12	&	-0.2051	&	-0.0922	&	-0.0955	&	-0.0374	&	0.0000	&	0.0365	&	0.1610	&	0.6450	&	0.6750	&	0.8145	&	0.8370	\\
	&		&(	0.0291	)&(	0.0315	)&(	0.0394	)&(	0.0411	)&		&		&		&		&		&		&		\\
	&		&[	0.9832	]&[	1.6532	]&[	1.3494	]&[	1.4615	]&		&		&		&		&		&		&		\\
	&	16	&	-0.1575	&	-0.0699	&	-0.0817	&	-0.0217	&	0.0005	&	0.0995	&	0.1415	&	0.7370	&	0.7710	&	0.8685	&	0.8815	\\
	&		&(	0.0287	)&(	0.0305	)&(	0.0355	)&(	0.0391	)&		&		&		&		&		&		&		\\
	&		&[	1.0241	]&[	1.8885	]&[	1.5463	]&[	1.5043	]&		&		&		&		&		&		&		\\
	&	20	&	-0.1258	&	-0.0541	&	-0.0693	&	-0.0114	&	0.0085	&	0.1780	&	0.1640	&	0.7760	&	0.8080	&	0.8935	&	0.9050	\\
	&		&(	0.0283	)&(	0.0297	)&(	0.0332	)&(	0.0371	)&		&		&		&		&		&		&		\\
	&		&[	1.0488	]&[	2.0891	]&[	1.7197	]&[	1.5009	]&		&		&		&		&		&		&		\\
	&	24	&	-0.1071	&	-0.0465	&	-0.0630	&	-0.0089	&	0.0160	&	0.2255	&	0.1655	&	0.8175	&	0.8455	&	0.9155	&	0.9310	\\
	&		&(	0.0265	)&(	0.0276	)&(	0.0303	)&(	0.0336	)&		&		&		&		&		&		&		\\
	&		&[	1.0112	]&[	2.1444	]&[	1.7938	]&[	1.4004	]&		&		&		&		&		&		&		\\
\hline
\end{longtable}
Notes: Monte Carlo experiments based on 2,000 repetitions. The standard deviations are inside parenthesis. Inside brackets are the ratio of the averages of the standard errors to the simulation of the standard deviations. The DGP is $y_{it} = c_{i} y_{i,t-1} + u_{it}$. 
\end{landscape}

\newpage

 \begin{landscape}
 \begin{longtable} {cc|cccc|ccccccc}
 \caption{Bias and Empirical Coverage for panel EXPAR model}
 \label{table.mc.3}\\
 \hline\hline																									
 $n$	&	$T$	&	FE	&	HK	&	GMM & HPJ	&	FE-CCM	&	HK	& GMM	& HPJ-CCM	&	HPJ-FEB	&	HPJ-HPJB	&	HPJ-HPJPB	\\
 \hline
                   &           &\multicolumn{4}{c|}{Bias} & \multicolumn{7}{c}{Coverage}\\
 \hline                             
 50	&	12	&	-0.1750	&	-0.0538	&	-0.1056	&	-0.0063	&	0.0540	&	0.6400	&	0.6525	&	0.7695	&	0.7835	&	0.9005	&	0.8955	\\
 	&		&(	0.0489	)&(	0.0530	)&(	0.0765	)&(	0.0710	)&		&		&		&		&		&		&		\\
 	&		&[	1.0552	]&[	1.4596	]&[	1.1175	]&[	1.6163	]&		&		&		&		&		&		&		\\
 	&	16	&	-0.1299	&	-0.0361	&	-0.0888	&	0.0002	&	0.1415	&	0.6750	&	0.5820	&	0.7875	&	0.8020	&	0.9120	&	0.9100	\\
 	&		&(	0.0424	)&(	0.0451	)&(	0.0594	)&(	0.0588	)&		&		&		&		&		&		&		\\
 	&		&[	1.0391	]&[	1.4749	]&[	1.1636	]&[	1.5574	]&		&		&		&		&		&		&		\\
 	&	20	&	-0.1050	&	-0.0288	&	-0.0789	&	0.0016	&	0.2165	&	0.7045	&	0.5190	&	0.8040	&	0.8230	&	0.9200	&	0.9270	\\
 	&		&(	0.0384	)&(	0.0404	)&(	0.0496	)&(	0.0513	)&		&		&		&		&		&		&		\\
 	&		&[	1.0296	]&[	1.5015	]&[	1.2033	]&[	1.5004	]&		&		&		&		&		&		&		\\
 	&	24	&	-0.0857	&	-0.0213	&	-0.0707	&	0.0031	&	0.3425	&	0.7200	&	0.4760	&	0.7940	&	0.8205	&	0.9160	&	0.9190	\\
 	&		&(	0.0372	)&(	0.0388	)&(	0.0444	)&(	0.0480	)&		&		&		&		&		&		&		\\
 	&		&[	1.0691	]&[	1.6044	]&[	1.2667	]&[	1.5075	]&		&		&		&		&		&		&		\\
 \hline
 100	&	12	&	-0.1721	&	-0.0506	&	-0.0746	&	-0.0039	&	0.0015	&	0.5060	&	0.6745	&	0.7635	&	0.7785	&	0.9165	&	0.9205	\\
 	&		&(	0.0344	)&(	0.0373	)&(	0.0563	)&(	0.0510	)&		&		&		&		&		&		&		\\
 	&		&[	1.0268	]&[	1.4532	]&[	1.1218	]&[	1.6227	]&		&		&		&		&		&		&		\\
 	&	16	&	-0.1297	&	-0.0359	&	-0.0654	&	0.0002	&	0.0095	&	0.5795	&	0.5835	&	0.7915	&	0.8205	&	0.9265	&	0.9240	\\
 	&		&(	0.0302	)&(	0.0321	)&(	0.0446	)&(	0.0420	)&		&		&		&		&		&		&		\\
 	&		&[	1.0174	]&[	1.4837	]&[	1.1899	]&[	1.5444	]&		&		&		&		&		&		&		\\
 	&	20	&	-0.1022	&	-0.0258	&	-0.0572	&	0.0042	&	0.0390	&	0.6515	&	0.5285	&	0.7970	&	0.8275	&	0.9350	&	0.9330	\\
 	&		&(	0.0270	)&(	0.0283	)&(	0.0368	)&(	0.0367	)&		&		&		&		&		&		&		\\
 	&		&[	0.9925	]&[	1.4936	]&[	1.2130	]&[	1.4868	]&		&		&		&		&		&		&		\\
 	&	24	&	-0.0846	&	-0.0202	&	-0.0506	&	0.0039	&	0.1075	&	0.6570	&	0.5060	&	0.8100	&	0.8370	&	0.9265	&	0.9320	\\
 	&		&(	0.0261	)&(	0.0272	)&(	0.0333	)&(	0.0337	)&		&		&		&		&		&		&		\\
 	&		&[	1.0329	]&[	1.5900	]&[	1.2918	]&[	1.4699	]&		&		&		&		&		&		&		\\
 \hline
 200	&	12	&	-0.1721	&	-0.0506	&	-0.0591	&	-0.0052	&	0.0000	&	0.2680	&	0.6205	&	0.7835	&	0.8065	&	0.9270	&	0.9300	\\
 	&		&(	0.0241	)&(	0.0261	)&(	0.0418	)&(	0.0352	)&		&		&		&		&		&		&		\\
 	&		&[	1.0084	]&[	1.4413	]&[	1.1601	]&[	1.5806	]&		&		&		&		&		&		&		\\
 	&	16	&	-0.1276	&	-0.0337	&	-0.0487	&	0.0024	&	0.0000	&	0.4325	&	0.5700	&	0.8000	&	0.8345	&	0.9390	&	0.9395	\\
 	&		&(	0.0209	)&(	0.0222	)&(	0.0307	)&(	0.0297	)&		&		&		&		&		&		&		\\
 	&		&[	0.9847	]&[	1.4549	]&[	1.1398	]&[	1.5331	]&		&		&		&		&		&		&		\\
 	&	20	&	-0.1011	&	-0.0247	&	-0.0436	&	0.0043	&	0.0005	&	0.5260	&	0.4985	&	0.8175	&	0.8500	&	0.9405	&	0.9405	\\
 	&		&(	0.0195	)&(	0.0205	)&(	0.0270	)&(	0.0260	)&		&		&		&		&		&		&		\\
 	&		&[	1.0041	]&[	1.5272	]&[	1.2305	]&[	1.4792	]&		&		&		&		&		&		&		\\
 	&	24	&	-0.0843	&	-0.0199	&	-0.0402	&	0.0042	&	0.0050	&	0.5590	&	0.4365	&	0.8135	&	0.8445	&	0.9390	&	0.9435	\\
 	&		&(	0.0186	)&(	0.0194	)&(	0.0241	)&(	0.0238	)&		&		&		&		&		&		&		\\
 	&		&[	1.0302	]&[	1.6044	]&[	1.2930	]&[	1.4541	]&		&		&		&		&		&		&		\\
 \hline
 \end{longtable}
 Notes: Monte Carlo based on 2,000 repetitions. The standard deviations are inside parenthesis. Inside brackets are ratios of averages of the standard errors to simulation of the standard deviations. The DGP is $y_{it} = c_{i} + \rho_1( y_{i,t-1}-c_{i}) + \rho_2 \exp( -(y_{i,t-1}- c_{i})^2) + u_{it}$, with true parameter $\rho_{1}=0.8$ and $\rho_{2}=1$. 
 \end{landscape}

\newpage

\linespread{1.08}

\begin{longtable}{cccccc}
\caption{Empirical Results for the Unemployment-Growth Model}\\
\label{t.app}\\
\hline\hline
\multicolumn{6}{c}{\textbf{Panel A: (1976-2010)}}\\
\hline
&  FE-CCM & HPJ-CCM & HPJ-HPJPB & GMM & TSLS\\ 
\hline
$\hat{\gamma}$	&$	0.790	$&$	0.830	$&$	0.830	$&$	0.800	$&$	0.336	$\\
$95\%$ CI	&$	[0.755,0.825]	$&$	[0.799,0.861]	$&$	[0.780,0.867]	$&$	[0.765,0.835]	$&$	[0.248,0.423]	$\\
$90\%$ CI	&$	[0.761,0.820]	$&$	[0.804,0.856]	$&$	[0.789,0.862]	$&$	[0.770,0.829]	$&$	[0.262,0.409]	$\\
$\hat{\beta}$ 	&$	-0.088	$&$	-0.079	$&$	-0.079	$&$	-0.084	$&$	-0.003	$\\
$95\%$ CI	&$	[-0.106,-0.070]	$&$	[-0.117,-0.041]	$&$	[-0.118,-0.048]	$&$	[-0.126,-0.042]	$&$	[-0.021,0.015]	$\\
$90\%$ CI	&$	[-0.103,-0.072]	$&$	[-0.111,-0.047]	$&$	[-0.108,-0.053]	$&$	[-0.120,-0.049]	$&$	[-0.018,0.012]	$\\
\hline
\multicolumn{6}{c}{\textbf{Panel B: (1976-2001)}}\\
\hline
&  FE-CCM & HPJ-CCM & HPJ-HPJPB & GMM & TSLS\\ 
\hline
$\hat{\gamma}$	&$	0.803	$&$	0.892	$&$	0.892	$&$	0.821	$&$	0.381	$\\
$95\%$ CI	&$	[0.763,0.842]	$&$	[0.852,0.932]	$&$	[0.854,0.926]	$&$	[0.780,0.862]	$&$	[0.273,0.488]	$\\
$90\%$ CI	&$	[0.770,0.836]	$&$	[0.858,0.925]	$&$	[0.860,0.921]	$&$	[0.787,0.856]	$&$	[0.291,0.471]	$\\
$\hat{\beta}$ 	&$	-0.081	$&$	-0.064	$&$	-0.064	$&$	-0.072	$&$	-0.010	$\\
$95\%$ CI	&$	[-0.106,-0.055]	$&$	[-0.108,-0.020]	$&$	[-0.112,-0.027]	$&$	[-0.122,-0.022]	$&$	[-0.028,0.008]	$\\
$90\%$ CI	&$	[-0.102,-0.059]	$&$	[-0.101,-0.027]	$&$	[-0.104,-0.034]	$&$	[-0.114,-0.030]	$&$	[-0.025,0.005]	$\\
\hline
\multicolumn{6}{c}{\textbf{Panel C: (1976-1991)}}\\
\hline
&  FE-CCM & HPJ-CCM & HPJ-HPJPB & GMM & TSLS\\ 
\hline
$\hat{\gamma}$	&$	0.676	$&$	0.721	$&$	0.721	$&$	0.669	$&$	0.446	$\\
$95\%$ CI	&$	[0.634,0.719]	$&$	[0.670,0.771]	$&$	[0.666,0.771]	$&$	[0.623,0.715]	$&$	[0.302,0.589]	$\\
$90\%$ CI	&$	[0.641,0.712]	$&$	[0.678,0.763]	$&$	[0.674,0.763]	$&$	[0.631,0.710]	$&$	[0.325,0.566]	$\\
$\hat{\beta}$ 	&$	-0.113	$&$	-0.129	$&$	-0.129	$&$	-0.095	$&$	0.000	$\\
$95\%$ CI	&$	[-0.138,-0.090]	$&$	[-0.2146,-0.0427]	$&$	[-0.215,-0.067]	$&$	[-0.167,-0.022]	$&$	[-0.029,0.029]	$\\
$90\%$ CI	&$	[-0.133,-0.093]	$&$	[-0.2008,-0.0565]	$&$	[-0.202,-0.073]	$&$	[-0.156,-0.034]	$&$	[-0.024,0.024]	$\\
\hline
\end{longtable}
\noindent 
Notes: Number of bootstrap for HPJ-HPJPB is 1,000.

\end{document}